\documentclass[12pt,reqno]{amsart}

\usepackage{eucal,color,graphicx,subfigure,mathrsfs}

\numberwithin{equation}{section}

\newtheorem{theorem}{Theorem}[section]
\newtheorem{lemma}[theorem]{Lemma}

\theoremstyle{definition}

\theoremstyle{remark}
\newtheorem{remark}[theorem]{Remark}

\numberwithin{equation}{section}

\newcommand{\uu}{{\scriptstyle{\mathscr{U}}}}
\newcommand{\uuh}{\uu_{\!h}}
\newcommand{\vv}{{\scriptstyle{\mathscr{V}}}}
\newcommand{\ww}{{\scriptstyle{\mathscr{W}}}}
\newcommand{\RRR}{\mathbb{R}}

\title[First order least squares method for convection diffusion problems]
{First order least squares method with weakly imposed boundary condition for convection dominated diffusion problems}

\author{Huangxin Chen}
\address{School of Mathematical Sciences, Xiamen University,
Xiamen, 361005, People's Republic of China}
\email{chx@xmu.edu.cn}

\author{Guosheng Fu}
\address{School of Mathematics, University of Minnesota,
Minneapolis, MN 55455, USA}
\email{fuxxx165@math.umn.edu}

\author{Jingzhi Li}
\address{Department of Financial Mathematics and Financial Engineering,
 South University of Science and Technology of China,
Shenzhen, 518055, People's Republic of China}
\email{li.jz@sustc.edu.cn}

\author{Weifeng Qiu}
\address{Department of Mathematics, City University of Hong Kong,
83 Tat Chee Avenue, Kowloon, Hong Kong, China}
\email{weifeqiu@cityu.edu.hk}
\thanks{Corresponding author: Weifeng Qiu}

\keywords{error estimate; least squares method; DPG method; convection diffusion problems.}

\subjclass[2000]{65N30, 65L12, 35L15}

\begin{document}

\begin{abstract}
We present and analyze a first order least squares method for convection dominated diffusion problems, 
which provides robust $L^{2}$ a priori error estimate for the scalar variable even if the given data $f\in L^{2}(\Omega)$. 
The novel theoretical approach is to rewrite the method in the framework of discontinuous Petrov–-Galerkin (DPG) method, 
and then show numerical stability by using a key equation discovered by J. Gopalakrishnan and W. Qiu 
[{\em Math. Comp. 83(2014), pp. 537-552}]. 
This new approach gives an alternative way to do numerical analysis for least squares methods 
for a large class of differential equations. We also show that the condition number of 
the global matrix is independent of the diffusion coefficient.
A key feature of the method is that there is no stabilization parameter chosen empirically.
In addition, Dirichlet boundary condition is weakly imposed. 
Numerical experiments verify our theoretical results and, in particular, 
show our way of weakly imposing Dirichlet boundary condition is essential to the design of least squares methods - 
numerical solutions on subdomains away from interior layers or boundary layers have remarkable accuracy 
even on coarse meshes, which are unstructured quasi-uniform.
\end{abstract}

\maketitle


\section{Introduction}
In this paper, we present a robust a priori analysis of first order least squares method with weakly imposed boundary condition 
for the following convection dominated diffusion equation
\begin{subequations}
\label{cd_eqs}
\begin{align}
-\epsilon\Delta u + \boldsymbol{\beta}\cdot \nabla u + cu = &\; f  \quad \text{ in $\Omega$, }
\\
u = &\; g \quad \text{ on $\partial \Omega$, }
\end{align}
\end{subequations}
where $\Omega \in R^d$ ($d = 2,3$) is a polyhedral domain, $0<\epsilon\leq 1$, $c$ a function in $L^{\infty}(\Omega)$, 
$f$ a function in $L^2 (\Omega)$ and $g$ a function in $H^{1/2}(\partial \Omega)$. Here, the variable flux 
$\boldsymbol{\beta}$ satisfies the following assumption:
\begin{subequations}
\label{AS_beta}
\begin{align}
\label{AS1}
& \boldsymbol{\beta}\cdot \nabla\psi \geq b_{0}>0 \text{ in } \Omega,\text{ for some function }\psi\in W^{1,\infty}(\Omega),\\
\label{AS2}
& c-\frac{1}{2}\nabla\cdot \boldsymbol{\beta}\geq 0 \text{ in }\Omega.
\end{align}
\end{subequations}
According to \cite{AyusoMarini:cdf}, the assumption (\ref{AS1}) is satisfied if 
\begin{align*}
\boldsymbol{\beta} \text{ has no closed curves and } \vert \boldsymbol{\beta}(x)\vert\neq 0 \text{ for all }x\in\Omega.
\end{align*}

Least squares methods have been frequently used to simulate solutions of partial differential equations 
arising from fluid dynamics and continuum mechanics. We refer to \cite{Bochev:ls_book,Jiang:ls_book} for 
comprehensive summary. It is well known that least squares methods have the following desirable features:
it leads to a minimization problem; its numerical stability is {\em not} sensitive to the choice of finite element 
space or meshes; the resulting global stiffness matrices are symmetric and positive definite; a practical a posteriori 
error estimator can be given without any additional cost, and so on 
(see \cite{Bochev_stokes,Bochev_review,Cai_plate,Cai_fols1,Cai_newton,Cai_fols2,Cai_elas,Chang1,Chang2,Chang3,Deang1,Fiard1,Pehlivanov1}).

Unfortunately, primitive least squares methods for convection dominated diffusion problems (\ref{cd_eqs}) have the following drawbacks.
Firstly, if the term $c-\frac{1}{2}\nabla\cdot \boldsymbol{\beta}$ is {\em not} uniformly bounded from below by a positive constant, 
$L^{2}$ a priori error estimate of primitive least squares methods will deteriorate as the diffusion coefficient $\epsilon$ goes to zero, 
even when the exact solution has no interior layers or boundary layers (see error estimates in \cite{CaiWestphal:ls_div,HsiehYang:EnrichedLS,Lazarov1}). 
Secondly, primitive least squares methods show a very poor performance for convection diffusion problem (\ref{cd_eqs}) 
with a sufficiently small diffusion coefficient, because large spurious oscillations are observed 
(see numerical experiments in \cite{HsiehYang:EnrichedLS}).
We notice that in \cite{HsiehYang:EnrichedLS}, residual-free bubble strategy is used to address the second drawback. 
But, the least squares method in \cite{HsiehYang:EnrichedLS} needs to compute basis functions element-wisely, 
which is relatively not easy to implement.

It is well known that streamline diffusion method \cite{BrooksHughes}, residual free bubble methods 
\cite{BrezziHughesMariniRussoSuli1999,BrezziMariniSuli2000,BurmanErn09}, and DG methods \cite{AyusoMarini:cdf,HoustonSchwabSuli2002,HughesScovazziBochevBuffa2006,DiPietro08,Ern09} 
do not suffer from the above two drawbacks of primitive least squares methods. 
We refer to \cite{Roos08,Stynes:acta_cd} as comprehensive summaries of numerical methods suitable for 
convection dominated diffusion problems. We would like to emphasize that none of these numerical methods 
(streamline diffusion method, residual free bubble methods, DG methods) results in symmetric global stiffness matrices. 
Hence from the point of view of solver design, the least squares method is more attractive than the other methods mentioned 
before and many works have been contributed to this subject (e.g. \cite{Cai_fols2, Lee}). Moreover, we derive that 
the condition number of linear system from our first order least squares method is at most $\mathcal{O}(h^{-2})$, 
where $h$ is the mesh size. In particular, the condition number is independent of the diffusion coefficient. 
This property is important for designing efficient solver, e.g., multilevel method, for the 
first order least squares approximation of convection dominated diffusion equation.

In this paper, we propose and analyze our first order least squares method to address these two drawbacks for 
primitive least squares methods. In fact, it is difficult to provide robust $L^{2}$ error estimate by the traditional 
approach of numerical analysis for least squares methods in \cite{Bochev:ls_book}. So, it is necessary to look for 
an alternative approach. We notice that a weighted test function was used in \cite{JohnsonPitkaranta86} to obtain 
the $L^{2}$ stability of the original DG method \cite{ReedHill73} for the transportation reaction equation, and 
this idea was generalized to convection-diffusion-reaction equation in \cite{AyusoMarini:cdf} using the IP--DG methods. 
In this paper, we rewrite our method in the framework of discontinuous Petrov–-Galerkin (DPG) method, then show 
numerical stability by using a key equation discovered in \cite{GopaQiu:PracticalDPG}.
The advantage of this new approach is that the weight function in \cite{AyusoMarini:cdf} is shown to stay   
in some ``equivalent" test function space (see (\ref{abstract_key_equation}) in section~\ref{sec:abstract}) 
such that numerical stability can be obtained without using any projection as in \cite{AyusoMarini:cdf}.
This approach is novel and useful to numerical analysis of least squares methods for a large class of differential equations. 
This new approach of numerical analysis is also different from traditional ones used for DPG method 
in \cite{Tan:DPG_friedrich,DemkoGopal:2010:DPG1,DemkoGopal:2010:DPG2,DemkoGopal:DPGanl,DemkoGopalNiemi:2010:DPG3}.
We show that, roughly speaking, using polynomials of degree $k+1\geq 1$,
\begin{equation}
\label{main_result_introduction}
\Vert u_{h} - u \Vert_{L^{2}(\Omega)} + \epsilon^{1/2}\Vert \nabla (u - u_{h})\Vert_{L^{2}(\Omega)} \leq C h^{k+1} \Vert u\Vert_{H^{k+2}(\Omega)};
\end{equation}
if $\epsilon^{1/2}\leq h_{K}$ for any $K\in\mathcal{T}_{h}$, 
\begin{align*}
& \Vert u - u_{h}\Vert_{L^{2}(\Omega)} + \epsilon^{1/2}\Vert \nabla (u - u_{h})\Vert_{L^{2}(\Omega)}
+ \Vert \boldsymbol{\beta}\cdot \nabla (u - u_{h})\Vert_{L^{2}(\Omega)} \\
\nonumber
\leq & C h^{k+1} \Vert u\Vert_{H^{k+2}(\Omega)}.
\end{align*} 
Here, the constant $C$ is independent of $\epsilon$. Thus, we can conclude that a priori error estimate 
in (\ref{main_result_introduction}) is {\em robust} with respect to the diffusion coefficient $\epsilon$, 
which addresses the first draw back. We also want to emphasize that the convergence result 
(\ref{main_result_introduction}) shows our method has $L^{2}$ convergence rate even if $f\in L^{2}(\Omega)$, 
which means our method does not have excessive smoothness requirements than other methods.
In order to overcome the second drawback, we impose Dirichlet boundary condition in an weak way, such that 
the error along the boundary layers will not propagate into the whole domain. We show the advantage of imposing 
boundary condition weakly by numerical experiments. We notice that our way of imposing boundary condition 
is similar to the weak imposition of Dirichlet boundary condition in \cite{Burman05,Bazilevs2007,Schieweck}, 
which belongs to Nitsche's method in \cite{Nitsche71}.
However, we do {\em not} have to choose any penalty terms empirically while \cite{Bazilevs2007} needs. 
We would like to emphasize that weakly imposing boundary condition is essential to least squares methods while 
it is incrementally helpful to streamline diffusion method and DG methods (see numerical experiments in \cite{Bazilevs2007}). 
If boundary condition is imposed strongly, the numerical solutions produced by streamline diffusion method and DG methods 
may have artificial oscillation along boundary layers, while the accuracy in subdomains away from boundary layers is still remarkable. 
However, according to our numerical experiments, if we impose boundary condition strongly, then numerical solution of least squares methods will be polluted on almost the whole domain by boundary layers. We have tried to add several stabilization terms, 
which have been utilized by streamline diffusion method or DG methods,  into least squares methods to prevent propagation 
of error from boundary layers. None of them works except weakly imposing boundary condition. 

This paper is the first one which addresses the two drawbacks of least squares methods for convection dominated diffusion problems.
Now, we would like to compare with DPG methods, which can be considered as a special class of least squares methods. We notice 
that all DPG methods \cite{BroersenStevenson:mild_weak_cd,ChanHeuerTanDemkowicz:DPG_CD,DemkoHeuer:2013:DPG_cd} need to 
compute optimal test function space. on each element such that the implementation is more complicated than ours (Methods in \cite{Dahmen1,Dahmen2} are similar to DPG methods.) Next, we would like to compare our results with IP--DG method  \cite{AyusoMarini:cdf}. \cite{AyusoMarini:cdf} is the first paper which gives {\em robust} a priori error estimate for 
variable flux $\boldsymbol{\beta}$. However, IP--DG method \cite{AyusoMarini:cdf} has to choose stabilization parameter 
empirically while our method does not. In addition, IP--DG method \cite{AyusoMarini:cdf} can have {\em robust} $L^{2}$ 
convergence only when the mesh size $h< h_{0}$ where $h_{0}$ is a positive constant depending on $\boldsymbol{\beta}$ 
(see Theorem $4.4$ in \cite{AyusoMarini:cdf}). On the contrast, the convergence result of our method does not have this restriction.   

The remainder of this paper is organized as follows. In section $2$, we introduce our first order least squares method and 
the main theoretical results. In section $3$, we show a novel approach to do numerical analysis for least squares methods 
({\em not} restricted to our first order least squares method for convection dominated diffusion problems). In section $4$, 
we prove a priori error estimates by using the approach introduced in section $3$. In section $5$, we prove the estimate of 
condition number of global stiffness matrix provided by our method. In section $6$, we give numerical experiments which 
verify our theoretical results. In section $7$, we extend our first order least squares method for transportation reaction problems.

\section{First order least squares method and main theoretical result}
In this section, we present setting of meshes, first order least squares method and the main theoretical results.

\subsection{Setting of meshes}
Let $\mathcal{T}_h  =\{K\}$ be a conforming  triangulation of the domain $\Omega$ made of shape-regular simplexes $K$.
For each element $K\in\mathcal{T}_h$, we set $h_K := |K|^{\frac{1}{d}}$ and for each
of its faces $e$, $h_F := |F|^{\frac{1}{d-1}}$, where $|\cdot|$ denotes the Lebesgue measure in $d$ or $d -1$ dimensions.
We associate to $\mathcal{T}_h$ the set of faces $\mathcal{E}_h$ as well as those
of interior faces $\mathcal{E}^i_h$ and boundary faces
$\mathcal{E}_h^{\partial}$. 
We say that $F\in \mathcal{E}^i_h$ if there are two elements $K^+$ and
$K^-$ in $\mathcal{T}_h$ such that $F=\partial K^+ \cap \partial K^-$, and we say that $F\in \mathcal{E}_h^{\partial}$
if there is an element $K$ in $\mathcal{T}_h$ such that $F=\partial K \cap
\partial \Omega$. It is obvious that 
$\mathcal{E}_h = \mathcal{E}_h^i \cup \mathcal{E}_h^{\partial}$.

\subsection{First order least squares method}
We define $\boldsymbol{q} = -\epsilon^{1/2} \nabla u$. 
We can rewrite (\ref{cd_eqs}) as the following first-order equations:
\begin{subequations}
\begin{align}
\label{cd_first_order1}
\boldsymbol{q} + \epsilon^{1/2}\nabla u & = 0\text{ in }\Omega,\\
\label{cd_first_order2}
\epsilon^{1/2}\nabla\cdot \boldsymbol{q} + \boldsymbol{\beta}\cdot \nabla u + cu & = f \text{ in }\Omega,\\
\label{cd_first_order3}
u & = g \text{ on }\partial\Omega.
\end{align}
\end{subequations}

We define the finite element space $\mathcal{U}_{h} = \boldsymbol{Q}_{h}\times W_{h}$, where
\begin{subequations}
\label{fem_space}
\begin{align}
\boldsymbol{Q}_{h} & = \{ \boldsymbol{p}\in H(\text{div},\Omega): 
\boldsymbol{p}|_{K}\in P_{k+1}(K;\mathbb{R}^{d})+\boldsymbol{x} P_{k+1}(K),\quad\forall K\in\mathcal{T}_{h}\}, \\
W_{h} & = \{w|_{K}\in H^{1}(\Omega): w\in P_{k+1}(K),\quad\forall K\in\mathcal{T}_{h} \}.
\end{align}
\end{subequations}
Here, $P_{k}(D)$ is the space of polynomials on the domain $D$ of total degree at most $k$, a non-negative integer.
Obviously, there is a positive constant $C$,
\begin{subequations}
\label{fem_space_assumps}
\begin{align}
\label{fem_space_assump1}
\Vert \boldsymbol{p}\cdot \boldsymbol{n}\Vert_{\partial K\cap \partial\Omega} & \leq C h_{K}^{-1/2} \Vert \boldsymbol{p}\Vert_{K},
\quad \forall \boldsymbol{p}\in \boldsymbol{Q}_{h}, K\in \mathcal{T}_{h},\\
\label{fem_space_assump2}
\Vert \nabla\cdot \boldsymbol{p}\Vert_{K} & \leq C h_{K}^{-1}\Vert \boldsymbol{p}\Vert_{K},\quad \forall \boldsymbol{p}\in 
\boldsymbol{Q}_{h}, K\in \mathcal{T}_{h}.
\end{align}
\end{subequations}

The first order least squares method is to find $(\boldsymbol{q}_{h},u_{h})\in \mathcal{U}_{h}$ satisfying
\begin{align}
\label{ls_formulation}
& \left( \boldsymbol{q}_{h}+\epsilon^{1/2}\nabla u_{h}, \boldsymbol{p}+\epsilon^{1/2}\nabla w\right)_{\Omega} \\
\nonumber
& \quad+ \left( \epsilon^{1/2}\nabla\cdot \boldsymbol{q}_{h} + \boldsymbol{\beta}\cdot \nabla u_{h} + c u_{h}, 
\epsilon^{1/2}\nabla\cdot \boldsymbol{p} + \boldsymbol{\beta}\cdot \nabla w + c w \right)_{\Omega}\\
\nonumber
& \quad +\Sigma_{F\in \mathcal{E}_h^{\partial}} h_{F}^{-1}\langle \left(\epsilon
+\max (-\boldsymbol{\beta}\cdot \boldsymbol{n}(x),0)\right)u_{h}, w\rangle_{F}\\
\nonumber
 = & (f, \epsilon^{1/2}\nabla\cdot \boldsymbol{p} + \boldsymbol{\beta}\cdot \nabla w + c w)_{\Omega}\\
 \nonumber
&\quad +\Sigma_{F\in \mathcal{E}_h^{\partial}} h_{F}^{-1}\langle \left(\epsilon
+\max (-\boldsymbol{\beta}\cdot \boldsymbol{n}(x),0)\right) g, w\rangle_{F},
\quad \forall (\boldsymbol{p},w)\in \mathcal{U}_{h}.
\end{align}
In (\ref{ls_formulation}), the inner products are
defined in the following natural manner:
\begin{align*}
(u,v)_{\Omega} = \int_{\Omega}uv dx\text{ and }\langle u, v\rangle_{F} = \int_{F} u v ds,\quad \forall F\in \mathcal{E}_{h},
\end{align*}
with the obvious modifications for vector-valued functions.

\begin{remark}
Notice that the boundary condition imposed on outflow boundary 
($\Gamma^{+}=\{x\in\partial\Omega:\boldsymbol{\beta}\cdot\boldsymbol{n}(x)>0\}$) is $O(\epsilon)$.
This is the reason we say {\em Dirichlet boundary condition is weakly imposed}.
If we ignore the weight function used in \cite{CaiWestphal:ls_div}, the first order least squares method in \cite{CaiWestphal:ls_div} 
is to find $(\boldsymbol{q}_{h},u_{h})\in \mathcal{Q}_{h}\times W_{g,h}$ satisfying
\begin{align}
\label{ls_formulation_strong}
& \left( \boldsymbol{q}_{h}+\epsilon^{1/2}\nabla u_{h}, \boldsymbol{p}+\epsilon^{1/2}\nabla w\right)_{\Omega} \\
\nonumber
& \quad+ \left( \epsilon^{1/2}\nabla\cdot \boldsymbol{q}_{h} + \boldsymbol{\beta}\cdot \nabla u_{h} + c u_{h}, 
\epsilon^{1/2}\nabla\cdot \boldsymbol{p} + \boldsymbol{\beta}\cdot \nabla w + c w \right)_{\Omega}\\
\nonumber
 = & (f, \epsilon^{1/2}\nabla\cdot \boldsymbol{p} + \boldsymbol{\beta}\cdot \nabla w + c w)_{\Omega},
\quad \forall (\boldsymbol{p},w)\in \mathcal{Q}_{h}\times W_{0,h},
\end{align}
where $W_{g,h}=\{w\in W_{h}: w|_{\partial \Omega} = g\}$ and $W_{0,h}=\{w\in W_{h}: w|_{\partial \Omega} = 0\}$. 
The main difference between (\ref{ls_formulation}) and (\ref{ls_formulation_strong}) is that in (\ref{ls_formulation_strong}), 
Dirichlet boundary condition is imposed strongly while 
our method (\ref{ls_formulation}) uses weakly imposed boundary condition.
Our convergence analysis in section~\ref{sec:convergence} is also valid for least squares method (\ref{ls_formulation_strong}).
But, if the solution of (\ref{cd_eqs}) has boundary layers or interior layers, then our way of weakly imposing boundary 
condition in (\ref{ls_formulation}) can improve the accuracy of numerical solution in subdomains away from layers 
dramatically. We refer to section $5$ for detailed description.
\end{remark}

\begin{remark}
In (\ref{ls_formulation}), we weakly imposed the Dirichlet boundary condition by 
\begin{align*}
\Sigma_{F\in \mathcal{E}_h^{\partial}} h_{F}^{-1}\langle \left(\epsilon
+\max (-\boldsymbol{\beta}\cdot \boldsymbol{n}(x),0)\right)u_{h}, w\rangle_{F}
= \Sigma_{F\in \mathcal{E}_h^{\partial}} h_{F}^{-1}\langle \left(\epsilon
+\max (-\boldsymbol{\beta}\cdot \boldsymbol{n}(x),0)\right) g, w\rangle_{F}.
\end{align*}
In fact, we can also weakly imposed the Dirichlet boundary condition as
\begin{align*}
\Sigma_{F\in \mathcal{E}_h^{\partial}} \langle \left( h_{F}^{-1} \epsilon
+\max (-\boldsymbol{\beta}\cdot \boldsymbol{n}(x),0)\right)u_{h}, w\rangle_{F}
= \Sigma_{F\in \mathcal{E}_h^{\partial}} \langle \left( h_{F}^{-1} \epsilon
+\max (-\boldsymbol{\beta}\cdot \boldsymbol{n}(x),0)\right) g, w\rangle_{F}.
\end{align*}
Then, the first order least squares method is to find $(\boldsymbol{q}_{h},u_{h})\in \mathcal{U}_{h}$ satisfying
\begin{align}
\label{ls_formulation_alt}
& \left( \boldsymbol{q}_{h}+\epsilon^{1/2}\nabla u_{h}, \boldsymbol{p}+\epsilon^{1/2}\nabla w\right)_{\Omega} \\
\nonumber
& \quad+ \left( \epsilon^{1/2}\nabla\cdot \boldsymbol{q}_{h} + \boldsymbol{\beta}\cdot \nabla u_{h} + c u_{h}, 
\epsilon^{1/2}\nabla\cdot \boldsymbol{p} + \boldsymbol{\beta}\cdot \nabla w + c w \right)_{\Omega}\\
\nonumber
& \quad +\Sigma_{F\in \mathcal{E}_h^{\partial}} \langle \left(h_{F}^{-1}\epsilon
+\max (-\boldsymbol{\beta}\cdot \boldsymbol{n}(x),0)\right)u_{h}, w\rangle_{F}\\
\nonumber
 = & (f, \epsilon^{1/2}\nabla\cdot \boldsymbol{p} + \boldsymbol{\beta}\cdot \nabla w + c w)_{\Omega}\\
 \nonumber
&\quad +\Sigma_{F\in \mathcal{E}_h^{\partial}} \langle \left(h_{F}^{-1}\epsilon
+\max (-\boldsymbol{\beta}\cdot \boldsymbol{n}(x),0)\right) g, w\rangle_{F},
\quad \forall (\boldsymbol{p},w)\in \mathcal{U}_{h}.
\end{align}
We can still obtain the stability estimate similar to Lemma~\ref{lemma_numerical_stability}. 
However, the $L^{2}$ convergence rate of (\ref{ls_formulation_alt}) is the same as that of (\ref{ls_formulation}), 
since the approximation to $\boldsymbol{\beta}\cdot \nabla u$ in $L^{2}$-norm gets involved with the error analysis 
(see the proof of Theorem~\ref{MainTh1}).
\end{remark}

\subsection{Main theoretical results}
We state our main theoretical results on a priori error analysis and condition number estimation of the first order least squares method
 as the following Theorems.

We denote by $C$ a positive constant, which is independent of $\epsilon$ and $h$. We assume the assumptions (\ref{AS_beta}) 
on $\boldsymbol{\beta}$ and $c$ hold.
\begin{theorem}
\label{MainTh1}
\begin{align}
\label{convergence_rate_global}
\Vert u - u_{h}\Vert_{L^{2}(\Omega)}+\epsilon^{1/2}\Vert \nabla (u - u_{h})\Vert_{L^{2}(\Omega)} 
\leq C h^{k}(\epsilon + h) \Vert u\Vert_{H^{k+2}(\Omega)}.
\end{align}
In addition, if we further assume $\boldsymbol{\beta}\in W^{2,\infty}(\mathcal{T}_{h})$ and $c\in W^{1,\infty}(\mathcal{T}_{h})$,
\begin{align}
\label{convergence_rate_global_refined}
\Vert u - u_{h}\Vert_{L^{2}(\Omega)}+\epsilon^{1/2}\Vert \nabla (u - u_{h})\Vert_{L^{2}(\Omega)} 
\leq C h^{k+1} \Vert u\Vert_{H^{k+2}(\Omega)}.
\end{align}
\end{theorem}

\begin{theorem}
\label{MainTh2}
If $0 < \epsilon^{1/2}\leq h_{K}$ for any $K\in\mathcal{T}_{h}$,
\begin{align}
\label{convergence_rate_global_beta}
& \Vert u - u_{h}\Vert_{L^{2}(\Omega)} + \epsilon^{1/2}\Vert \nabla (u - u_{h})\Vert_{L^{2}(\Omega)}
+ \Vert \boldsymbol{\beta}\cdot \nabla (u - u_{h})\Vert_{L^{2}(\Omega)} \\
\nonumber
\leq & C h^{k}(\epsilon + h) \Vert u\Vert_{H^{k+2}(\Omega)}.
\end{align}
In addition, if we further assume $\boldsymbol{\beta}\in W^{2,\infty}(\mathcal{T}_{h})$ and $c\in W^{1,\infty}(\mathcal{T}_{h})$,
\begin{align}
\label{convergence_rate_global_beta_refined}
& \Vert u - u_{h}\Vert_{L^{2}(\Omega)} + \epsilon^{1/2}\Vert \nabla (u - u_{h})\Vert_{L^{2}(\Omega)}
+ \Vert \boldsymbol{\beta}\cdot \nabla (u - u_{h})\Vert_{L^{2}(\Omega)} \\
\nonumber
\leq & C h^{k+1} \Vert u\Vert_{H^{k+2}(\Omega)}.
\end{align}
\end{theorem}

\begin{theorem}
\label{MainTh3}
We denote by $\kappa$  the condition number of global stiffness matrix of first order least squares method 
(\ref{ls_formulation}). If we assume the meshes are quasi-uniform, then
\begin{align}
\kappa \leq C h^{-2}.
\end{align} 
\end{theorem}

\begin{remark}
In Theorem~\ref{MainTh2}, we obtain optimal convergence of $\Vert u - u_{h}\Vert_{L^{2}(\Omega)} 
+ \epsilon^{1/2}\Vert \nabla (u - u_{h})\Vert_{L^{2}(\Omega)}+ \Vert \boldsymbol{\beta}\cdot \nabla (u - u_{h})\Vert_{L^{2}(\Omega)}$ 
if $0 < \epsilon^{1/2}\leq h_{K}$ for any $K\in\mathcal{T}_{h}$. The restriction on mesh size is due to the energy norm of $\mathcal{U}$ 
in Lemma~\ref{lemma_numerical_stability} is 
\begin{align*}
\Vert (\boldsymbol{p},w)\Vert_{\mathcal{U}}^{2} 
= & \Vert w\Vert_{L^{2}(\Omega)}^{2}+\epsilon \Vert \nabla w\Vert_{L^{2}(\Omega)}^{2}
+\Vert w\Vert_{L^{2}(\partial\Omega,\boldsymbol{\beta})}^{2}+\Vert \boldsymbol{p}\Vert_{L^{2}(\Omega)}^{2}\\
&\quad +\Vert \epsilon^{1/2}\nabla\cdot \boldsymbol{p}+\boldsymbol{\beta}\cdot \nabla w\Vert_{L^{2}(\Omega)}^{2},\quad 
\forall (\boldsymbol{p},w)\in\mathcal{U}.
\end{align*}
Thus, we need to utilize inverse inequality and the above restriction on mesh size to obtain upper bound on $\Vert\boldsymbol{\beta}\cdot 
\nabla (u - u_{h})\Vert_{L^{2}(\Omega)}$.
\end{remark}

\section{The approach to analysis} \label{sec:abstract}
In this section, an alternative approach is introduced for numerical analysis of least squares methods in general 
({\em not} restricted to the first order least squares method (\ref{ls_formulation})).
First of all, we show least squares methods and DPG method in abstract settings, respectively. 
Then, we rewrite least squares methods in the framework of DPG method.
A key equation in \cite{GopaQiu:PracticalDPG} is then used to show how to achieve numerical stability of least squares methods.

Suppose we want to approximate solution $\uu\in \mathcal{U}$ satisfying 
\begin{equation*}
\mathcal{L}\uu = f\in \mathcal{V}.
\end{equation*}
Here, $\mathcal{U}$ is a Banach space with norm $\Vert \cdot \Vert_{\mathcal{U}}$, 
and $\mathcal{V}$ is a Hilbert space with norm $\Vert \cdot \Vert_{\mathcal{V}}$.
We assume that the linear operator $\mathcal{L}$ is in $\mathbb{B}(\mathcal{U},\mathcal{V})$, 
which is the space of bounded linear operators from $\mathcal{U}$ to $\mathcal{V}$. 
Then, least squares methods are to find $\uuh\in \mathcal{U}_{h}$ satisfying
\begin{equation}
\label{abstract_LS}
(\mathcal{L}\uuh,\mathcal{L}\ww)_{\mathcal{V}} = (f,\mathcal{L}\ww)_{\mathcal{V}}\quad \forall \ww\in \mathcal{U}_{h}.
\end{equation}
Here, $\mathcal{U}_{h}$ is a finite dimensional trial subspace of $\mathcal{U}$
(where $h$ denotes a parameter determining the finite dimension).

On the other hand, DPG method can be described in the following general context. Suppose we
want to approximate $\uu \in \mathcal{U}$ satisfying
\begin{equation}
  \label{eq:1}
  b(\uu,\vv) = l(\vv), \quad \forall \vv \in \mathcal{V}.
\end{equation}
Here $\mathcal{U}$ is a Banach space with norm $\| \cdot\|_{\mathcal{U}}$ and
$\mathcal{V}$ is a Hilbert space under an inner product $(\cdot,\cdot)_{\mathcal{V}}$ with 
corresponding norm $\| \cdot \|_{\mathcal{V}}$. We
assume that the bi-linear form $b(\cdot,\cdot): \mathcal{U} \times \mathcal{V} \mapsto
\RRR$ is continuous and the linear form $l(\cdot): \mathcal{V}\mapsto \RRR$ is
also continuous. Define $T: \mathcal{U} \mapsto \mathcal{V}$ by
\begin{equation}
  \label{eq:T}
  (T\ww,\vv)_{\mathcal{V}} = b(\ww,\vv),\qquad \forall  \vv \in \mathcal{V}.
\end{equation}
Then, the DPG approximation to $\uu$, lies in a finite dimensional
trial subspace $\mathcal{U}_h \subset \mathcal{U}$. It satisfies
\begin{equation}
  \label{eq:2}
  b(\uuh,\vv) = l(\vv),\quad\forall \vv \in \mathcal{V}_h,
\end{equation}
where $\mathcal{V}_h = T(\mathcal{U}_h)$.  Since $\mathcal{U}_h \ne \mathcal{V}_h$ in general, this is a
Petrov-Galerkin approximation. The method~\eqref{eq:2} is the {\em
  DPG method}. The excellent stability and approximation
properties of this method are well
known~\cite{DemkoGopal:2010:DPG2,DemkoGopal:DPGanl}.

Now, we can rewrite the least squares methods (\ref{abstract_LS}) in the framework of DPG method in the following way.
The corresponding DPG method is to find $\uuh\in\mathcal{U}_{h}$ such that
\begin{equation}
\label{abstract_DPG}
b(\uuh,\vv) = (f,\vv)_{\mathcal{V}}\quad \forall \vv\in \mathcal{V}_{h}= T(\mathcal{U}_{h}),
\end{equation}
where the bi-linear form 
\begin{equation*}
b(\ww, \vv) = (\mathcal{L}\ww, \vv)_{\mathcal{V}},\quad\forall \ww\in \mathcal{U},\vv\in\mathcal{V}.
\end{equation*}
Here, for any $\ww \in\mathcal{U}$, 
\begin{equation*}
(T\ww, \delta \vv)_{\mathcal{V}} = b(\ww, \delta\vv)\quad \forall \delta \vv \in \mathcal{V}.
\end{equation*}

We say the corresponding DPG method (\ref{abstract_DPG}) is numerical stable if there is a constant $C$ independent of mesh size $h$, 
such that for any $\ww\in \mathcal{U}_{h}$,
\begin{equation}
\label{abstract_numerical_stability}
\Vert \ww\Vert_{\mathcal{U}}\leq C \sup_{0\neq \vv\in \mathcal{V}_{h}}\dfrac{ b(\ww, \vv) }{\Vert \vv\Vert_{\mathcal{V}}}.
\end{equation}

In general, it is not easy to know what are elements in the finite dimensional space $\mathcal{V}_{h}$. So, it is usually not easy 
to estimate the right hand side of the inequality (\ref{abstract_numerical_stability}). But, according to the following 
Theorem~\ref{T_property},  we have that for any $\ww\in \mathcal{U}_{h}\subset \mathcal{U}$,
\begin{equation}
\label{abstract_key_equation}
\sup_{0\neq \vv\in \mathcal{V}_{h}}\dfrac{ b(\ww, \vv) }{\Vert \vv\Vert_{\mathcal{V}}} = 
\sup_{0\neq \vv\in \mathcal{V}}\dfrac{ b(\ww, \vv) }{\Vert \vv\Vert_{\mathcal{V}}}.
\end{equation}
(\ref{abstract_key_equation}) is the key equation discovered in work for DPG methods. 
We call $\mathcal{V}$ the ``equivalent" test function space.
By (\ref{abstract_key_equation}), in order to obtain 
(\ref{abstract_numerical_stability}), it is sufficient to achieve
\begin{equation}
\label{abstract_numerical_stability_estimate}
\Vert \ww\Vert_{\mathcal{U}}\leq C \sup_{0\neq \vv\in \mathcal{V}}\dfrac{ b(\ww, \vv) }{\Vert \vv\Vert_{\mathcal{V}}}.
\end{equation}
The inequality (\ref{abstract_numerical_stability_estimate}) is usually easier to be obtained than (\ref{abstract_numerical_stability}).
Notice that the least squares methods (\ref{abstract_LS}) is actually the same as the corresponding DPG method (\ref{abstract_DPG}).
This explains why it is relatively easier to obtain numerical stability of least squares methods. 
In fact, first order least squares methods based on first order Friedrichs' systems in 
\cite{ErnGuermond:2006:friedrich1,ErnGuermond:2006:friedrich2,ErnGuermond:2008:friedrich3} 
can be shown to have numerical stability by using the key equation (\ref{abstract_key_equation}).

The proof of following Theorem~\ref{T_property} is included in that of Theorem~$2.1$ in \cite{GopaQiu:PracticalDPG}.
In \cite{GopaQiu:PracticalDPG}, the solution space $\mathcal{U}$ is restricted to Hilbert spaces. We put the proof here 
in order to make this paper more self-contained.

\begin{theorem}
\label{T_property}
Let $b:\mathcal{U}\times \mathcal{V} \rightarrow \mathbb{R}$ be a continuous bi-linear mapping.
Here, $\mathcal{U}$ is a normed linear space, 
and $\mathcal{V}$ is a Hilbert space with associate norm $\Vert \cdot\Vert_{\mathcal{V}}$.

We define a linear operator $T:\mathcal{U}\mapsto \mathcal{V}$ by
\begin{equation}
\label{def_T_op}
(T\ww, \delta \vv)_{\mathcal{V}} = b(\ww,\delta \vv),\quad \forall \delta \vv\in \mathcal{V}.
\end{equation}

Then, for any $\ww\in \mathcal{U}$, we have that
\begin{equation}
\label{id_dpg}
\sup_{0\neq \vv\in T(\mathcal{U})}\dfrac{b(\ww,\vv)}{\Vert \vv\Vert_{\mathcal{V}}}=
\sup_{0\neq \vv\in \mathcal{V}}\dfrac{ b(\ww,\vv)}{\Vert \vv\Vert_{\mathcal{V}}}.
\end{equation}
\end{theorem}

\begin{proof}
Since the bi-linear mapping is continuous and $\mathcal{V}$ is a Hilbert space, the operator $T$ is well-defined.

We take $\ww\in \mathcal{U}$ arbitrarily. 
It is straightforward to see that 
\begin{equation*}
\sup_{0\neq \vv\in T(\mathcal{U})}\dfrac{ b(\ww,\vv)}{\Vert \vv\Vert_{\mathcal{V}}}\leq
\sup_{0\neq \vv\in \mathcal{V}}\dfrac{ b(\ww,\vv)}{\Vert \vv\Vert_{\mathcal{V}}},
\end{equation*}
since $T(\mathcal{U})\subset \mathcal{V}$.

If $T \ww=0\in\mathcal{U}$, then (\ref{id_dpg}) is obviously true due to (\ref{def_T_op}).
If $T \ww\neq 0$, then by (\ref{def_T_op}) and the fact that $\mathcal{V}$ is a Hilbert space, we have
\begin{equation*}
\sup_{0\neq \vv\in \mathcal{V}}\dfrac{ b(\ww,\vv)}{\Vert \vv\Vert_{\mathcal{V}}} 
= \sup_{0\neq \vv\in \mathcal{V}}\dfrac{ (T \ww, \vv)_{\mathcal{V}} }{\Vert \vv\Vert_{\mathcal{V}}} 
= \dfrac{(T\ww, T\ww)_{\mathcal{V}}}{\Vert T\ww \Vert_{\mathcal{V}}}
= \dfrac{b(\ww,T\ww)}{\Vert T\ww \Vert_{\mathcal{V}}}
\leq \sup_{0\neq \vv\in T(\mathcal{U})}\dfrac{ b(\ww,\vv)}{\Vert \vv\Vert_{\mathcal{V}}}.
\end{equation*}

We can conclude that for any $\ww\in \mathcal{U}$, (\ref{id_dpg}) is true.
\end{proof}

\section{Convergence analysis}\label{sec:convergence}
In this section, we shall prove Theorem~\ref{MainTh1} and Theorem~\ref{MainTh2}.
According to the approach of analysis in section~\ref{sec:abstract}, we rewrite first order least squares method 
(\ref{ls_formulation}) in the framework of DPG method. Then, we show the numerical stability of corresponding DPG method. 
Finally, we prove Theorem~\ref{MainTh1} and Theorem~\ref{MainTh2}. 

Throughout this section, we define $\mathcal{U}=H(\text{div},\Omega)\times H^{1}(\Omega)$ and 
$\mathcal{V}=L^{2}(\Omega;\mathbb{R}^{d})\times L^{2}(\Omega) \times L^{2}(\partial\Omega)$.
The inner product of $\mathcal{V}$ is defined by
\begin{align}
\label{inner_product_V}
& ((\boldsymbol{r},v,\mu),(\delta \boldsymbol{r},\delta v, \delta \mu))_{\mathcal{V}} \\
\nonumber
= &  (\boldsymbol{r},\delta \boldsymbol{r})_{\Omega} + (v,\delta v)_{\Omega}
+\Sigma_{F\in \mathcal{E}_h^{\partial}} h_{F}^{-1}\langle \left(\epsilon+\max (-\boldsymbol{\beta}\cdot \boldsymbol{n}(x),0)\right)
\mu, \delta \mu\rangle_{F},\\
\nonumber
&\quad \forall (\boldsymbol{r},v,\mu),(\delta \boldsymbol{r},\delta v, \delta \mu)\in\mathcal{V}.
\end{align}
Obviously, $\mathcal{V}$ is a Hilbert space with respect to the inner product in (\ref{inner_product_V}).
For any $\mu\in L^{2}(\partial\Omega)$, we define
\begin{equation}
\label{def_norm_boundary_beta}
\Vert \mu\Vert_{L^{2}(\partial\Omega,\boldsymbol{\beta})}^{2}
=\Sigma_{F\in \mathcal{E}_h^{\partial}} h_{F}^{-1}
\Vert  \left(\epsilon+\max (-\boldsymbol{\beta}\cdot \boldsymbol{n}(x),0)\right)^{1/2}  \mu\Vert_{L^{2}(F)}^{2}.
\end{equation}
The norm of $\mathcal{V}$ is
\begin{align}
\label{norm_V}
\Vert (\boldsymbol{r},v,\mu)\Vert_{\mathcal{V}}^{2} = \Vert \boldsymbol{r}\Vert_{L^{2}(\Omega)}^{2}
+\Vert v\Vert_{L^{2}(\Omega)}^{2}+\Vert \mu\Vert_{L^{2}(\partial\Omega,\boldsymbol{\beta})}^{2},\quad 
\forall (\boldsymbol{r},v,\mu)\in \mathcal{V}.
\end{align}
The norm of $\mathcal{U}$ is defined by
\begin{align}
\label{norm_U}
& \Vert (\boldsymbol{p},w)\Vert_{\mathcal{U}}^{2} \\
\nonumber
= & \Vert w\Vert_{L^{2}(\Omega)}^{2}+\epsilon \Vert \nabla w\Vert_{L^{2}(\Omega)}^{2}
+\Vert w\Vert_{L^{2}(\partial\Omega,\boldsymbol{\beta})}^{2}+\Vert \boldsymbol{p}\Vert_{L^{2}(\Omega)}^{2}\\
\nonumber
&\quad +\Vert \epsilon^{1/2}\nabla\cdot \boldsymbol{p}+\boldsymbol{\beta}\cdot \nabla w\Vert_{L^{2}(\Omega)}^{2},\quad 
\forall (\boldsymbol{p},w)\in\mathcal{U}.
\end{align}

\subsection{Rewrite least squares method in DPG framework}\label{subsec:rewrite}
The first order least squares method (\ref{ls_formulation}) is the same as the corresponding 
DPG method (\ref{dpg_formulation}) in the following.

Equivalently, the first order least squares method (\ref{ls_formulation}) is to find $(\boldsymbol{q}_{h},u_{h})\in \mathcal{U}_{h}$ satisfying
\begin{align}
\label{dpg_formulation}
& b ((\boldsymbol{q}_{h}, u_{h}), (\boldsymbol{r},v,\mu))\\
\nonumber
= & (f,v)_{\Omega}
+\Sigma_{F\in \mathcal{E}_h^{\partial}} h_{F}^{-1}
\langle \left(\epsilon+\max (-\boldsymbol{\beta}\cdot \boldsymbol{n}(x),0)\right) g, \mu\rangle_{F},\\
\nonumber
&\qquad \forall (\boldsymbol{r},v,\mu) \in T(\mathcal{U}_{h}).
\end{align}
Here, the bi-linear form $b(\cdot,\cdot)$ is defined by
\begin{align}
\label{dpg_bilinear_form}
 & b ((\boldsymbol{p},w), (\boldsymbol{r},v,\mu)) \\
 \nonumber
= & \left(\boldsymbol{p}+\epsilon^{1/2}\nabla w, \boldsymbol{r}\right)_{\Omega}
+ \left( \epsilon^{1/2}\nabla\cdot \boldsymbol{p} + \boldsymbol{\beta}\cdot \nabla w + c w, v \right)_{\Omega}\\
\nonumber
& +\Sigma_{F\in \mathcal{E}_h^{\partial}} h_{F}^{-1}
\langle \left(\epsilon+\max (-\boldsymbol{\beta}\cdot \boldsymbol{n}(x),0)\right) w, \mu\rangle_{F},
\end{align}
for any $(\boldsymbol{p},w)\in \mathcal{U}$, $(\boldsymbol{r},v,\mu)\in \mathcal{V}$.
Obviously, the bi-linear form (\ref{dpg_bilinear_form}) is continuous on $\mathcal{U}\times\mathcal{V}$ 
with respect to the norm (\ref{norm_U}) and norm (\ref{norm_V}).
The operator $T: \mathcal{U} \rightarrow \mathcal{V}$ is defined by
\begin{align}
\label{T_op}
& (T(\boldsymbol{p},w),(\delta \boldsymbol{r}, \delta v, \delta \mu))_{\mathcal{V}}\\
\nonumber
= &  b((\boldsymbol{p},w), (\delta \boldsymbol{r}, \delta v, \delta \mu)),\quad 
\forall (\delta \boldsymbol{r}, \delta v, \delta \mu)\in \mathcal{V}.
\end{align}
It is easy to see that for any $(\boldsymbol{p},w)\in \mathcal{U}$,
\begin{equation}
\label{T_op_explicit}
T(\boldsymbol{p} ,w) = (\boldsymbol{p}+\epsilon^{1/2}\nabla w, 
\epsilon^{1/2}\nabla\cdot \boldsymbol{p} + \boldsymbol{\beta}\cdot \nabla w + c w, w|_{\partial\Omega}).
\end{equation}

\subsection{Numerical stability}
In the following Lemma~\ref{lemma_numerical_stability},
we show numerical stability of  the corresponding DPG method (\ref{dpg_formulation}). Notice that DPG method  (\ref{dpg_formulation}) 
is the same as first order least squares method (\ref{ls_formulation}).

\begin{lemma}
\label{lemma_numerical_stability}
If the assumptions (\ref{AS_beta}) hold, then there is a positive constant $C$, which is independent of $\epsilon$, 
such that for any 
$(e_{\boldsymbol{q}}, e_{u})\in \mathcal{U}_{h}$, 
\begin{align}
\label{numerical_stability1}
\Vert (e_{\boldsymbol{q}},e_{u})\Vert_{\mathcal{U}} \leq 
C \sup_{0\neq (\boldsymbol{r},v,\mu)\in T(\mathcal{U}_{h})}
\dfrac{b( (e_{\boldsymbol{q}},e_{u}), (\boldsymbol{r},v,\mu) )}{ \Vert (\boldsymbol{r},v,\mu) \Vert_{\mathcal{V}}}.
\end{align}

In addition, if $0 < \epsilon^{1/2}\leq h_{K}$ for any $K\in\mathcal{T}_{h}$, then we have
\begin{align}
\label{numerical_stability2}
\Vert (e_{\boldsymbol{q}},e_{u})\Vert_{\mathcal{U}} + \Vert \boldsymbol{\beta}\cdot \nabla e_{u} \Vert_{L^{2}(\Omega)} \leq 
C \sup_{0\neq (\boldsymbol{r},v,\mu)\in T(\mathcal{U}_{h})}
\dfrac{b( (e_{\boldsymbol{q}},e_{u}), (\boldsymbol{r},v,\mu) )}{ \Vert (\boldsymbol{r},v,\mu) \Vert_{\mathcal{V}}}.
\end{align}
\end{lemma}

\begin{proof}
According to Theorem~\ref{T_property} and (\ref{T_op}), we have
\begin{align}
\label{inf_sup_eq}
 \sup_{0\neq (\boldsymbol{r},v,\mu)\in T(\mathcal{U}_{h})}
\dfrac{b( (e_{\boldsymbol{q}},e_{u}), (\boldsymbol{r},v,\mu) )}{\Vert (\boldsymbol{r},v,\mu)\Vert_{\mathcal{V}}}
 = \sup_{0\neq (\boldsymbol{r},v,\mu)\in \mathcal{V}}
 \dfrac{b( (e_{\boldsymbol{q}},e_{u}), (\boldsymbol{r},v,\mu) )}{\Vert (\boldsymbol{r},v,\mu)\Vert_{\mathcal{V}}}.
\end{align}

Given $(e_{\boldsymbol{q}},e_{u})\in \mathcal{U}_{h}$, we define
\begin{equation}
\label{inf_sup_value}
\theta = \sup_{0\neq (\boldsymbol{r},v,\mu)\in \mathcal{V}}
\dfrac{b( (e_{\boldsymbol{q}},e_{u}), (\boldsymbol{r},v,\mu) )}{\Vert (\boldsymbol{r},v,\mu)\Vert_{\mathcal{V}}}.
\end{equation}

Recall that $\psi$ is introduced in (\ref{AS1}). Let $\kappa$ be a non-negative number, which will be determined below. We take
\begin{align}
\label{optimal_test}
\boldsymbol{r} = (e^{-\psi}+\kappa)e_{\boldsymbol{q}},\quad 
v = (e^{-\psi}+\kappa)e_{u},\quad
\mu = (e^{-\psi}+\kappa) (e_{u} - \epsilon^{1/2}\widehat{e}_{\boldsymbol{q}}),
\end{align}
where $\widehat{e}_{\boldsymbol{q}}|_{F}
=h_{F}\left(\epsilon+\max (-\boldsymbol{\beta}\cdot \boldsymbol{n}(x),0)\right)^{-1}e_{\boldsymbol{q}}\cdot \boldsymbol{n}|_{F}$ 
for any $F\in \mathcal{E}_{h}\cap \partial\Omega$.
It is easy to see that $(\boldsymbol{r},v,\mu)\in \mathcal{V}$, and
\begin{align*}
 & b ((e_{\boldsymbol{q}},e_{u}), (\boldsymbol{r},v,\mu)) \\
= &  \left(e_{\boldsymbol{q}}+\epsilon^{1/2}\nabla e_{u}, (e^{-\psi}+\kappa)e_{\boldsymbol{q}}\right)_{\Omega}
+ \left( \epsilon^{1/2}\nabla\cdot e_{\boldsymbol{q}} + \boldsymbol{\beta}\cdot \nabla e_{u} + c e_{u}, 
(e^{-\psi}+\kappa)e_{u} \right)_{\Omega}\\
& + \Sigma_{F\in \mathcal{E}_h^{\partial}} h_{F}^{-1}
\langle \left(\epsilon+\max (-\boldsymbol{\beta}\cdot \boldsymbol{n}(x),0)\right) e_{u}, (e^{-\psi}+\kappa) e_{u}\rangle_{F}\\
& - \Sigma_{F\in \mathcal{E}_h^{\partial}}
\langle \epsilon^{1/2} e_{\boldsymbol{q}}\cdot \boldsymbol{n}, (e^{-\psi}+\kappa) e_{u} \rangle_{\partial\Omega}.
\end{align*}
Applying integration by parts on the term $(\epsilon^{1/2}\nabla\cdot e_{\boldsymbol{q}}, (e^{-\psi}+\kappa)e_{u})_{\Omega}$, we have
\begin{align*}
 & b ((e_{\boldsymbol{q}},e_{u}), (\boldsymbol{r},v,\mu)) \\
= &   (e_{\boldsymbol{q}},(e^{-\psi}+\kappa)e_{\boldsymbol{q}})_{\Omega}  
+\epsilon^{1/2} (e_{\boldsymbol{q}},(\nabla \psi) e^{-\psi} e_{u})_{\Omega} 
+ \left(\boldsymbol{\beta}\cdot \nabla e_{u} + c e_{u}, (e^{-\psi}+\kappa)e_{u} \right)_{\Omega}\\
& + \Sigma_{F\in \mathcal{E}_h^{\partial}} h_{F}^{-1}
\langle \left(\epsilon+\max (-\boldsymbol{\beta}\cdot \boldsymbol{n}(x),0)\right) e_{u}, (e^{-\psi}+\kappa) e_{u}\rangle_{F}.
\end{align*}
Applying integration by parts on the term $(\boldsymbol{\beta}\cdot \nabla e_{u}, (e^{-\psi}+\kappa)e_{u})_{\Omega}$ 
and assumption (\ref{AS1}) and assumption (\ref{AS2}), we have
\begin{align*}
  & b ((e_{\boldsymbol{q}},e_{u}), (\boldsymbol{r},v,\mu)) \\
 \geq & (e_{\boldsymbol{q}},(e^{-\psi}+\kappa)e_{\boldsymbol{q}})_{\Omega}  
 +\epsilon^{1/2} (e_{\boldsymbol{q}},(\nabla \psi) e^{-\psi} e_{u})_{\Omega}
+\dfrac{b_{0}}{2} (e_{u}, e^{-\psi}e_{u})_{\Omega} \\
 & \quad  + \frac{1}{2}\Sigma_{F\in \mathcal{E}_h^{\partial}} h_{F}^{-1}
 \langle \left(\epsilon+\max (-\boldsymbol{\beta}\cdot \boldsymbol{n}(x),0)\right) e_{u},  (e^{-\psi}+\kappa) e_{u}\rangle_{F}.
\end{align*}
If we take $\kappa \geq 0$ big enough (which depends only on $b_{0}$ and $\Vert \psi\Vert_{W^{1,\infty}(\Omega)}$), 
\begin{align}
\label{ns_ineq1}
  & b ((e_{\boldsymbol{q}},e_{u}), (\boldsymbol{r},v,\mu)) \\
\nonumber
 \geq & (e_{\boldsymbol{q}},e^{-\psi}e_{\boldsymbol{q}})_{\Omega} 
 +\frac{b_{0}}{4}(e_{u}, e^{-\psi}e_{u})_{\Omega}\\
\nonumber 
&  + \frac{1}{2}\Sigma_{F\in \mathcal{E}_h^{\partial}} h_{F}^{-1}
\langle \left(\epsilon+\max (-\boldsymbol{\beta}\cdot \boldsymbol{n}(x),0)\right) e_{u}, e^{-\psi} e_{u}\rangle_{F}.
\end{align}

According to (\ref{def_norm_boundary_beta}) and the definition of $\mu$ in (\ref{optimal_test}), we have 
\begin{align*}
\Vert \mu\Vert_{L^{2}(\partial\Omega,\boldsymbol{\beta})} 
\leq C \left( \Vert e_{u}\Vert_{L^{2}(\partial\Omega,\boldsymbol{\beta})}^{2}
+ \Sigma_{F\in \mathcal{E}_h^{\partial}}h_{F} 
\Vert e_{\boldsymbol{q}}\cdot \boldsymbol{n} \Vert_{F}^{2}   \right)^{1/2}.
\end{align*}
By (\ref{fem_space_assump1}) and shape regularity assumption, we have
\begin{align*}
\Vert \mu\Vert_{L^{2}(\partial\Omega,\boldsymbol{\beta})} 
\leq C \left( \Vert e_{u}\Vert_{L^{2}(\partial\Omega,\boldsymbol{\beta})} 
+\Vert e_{\boldsymbol{q}}\Vert_{L^{2}(\Omega)} \right).
\end{align*}
We recall $\Vert (\boldsymbol{r},v,\mu)\Vert_{\mathcal{V}}^{2} = \Vert \boldsymbol{r}\Vert_{L^{2}(\Omega)}^{2}
+\Vert v\Vert_{L^{2}(\Omega)}^{2}+\Vert \mu\Vert_{L^{2}(\partial\Omega,\boldsymbol{\beta})}^{2}$. 
So, we have
\begin{align}
\label{ns_ineq2}
\Vert (\boldsymbol{r},v,\mu)\Vert_{\mathcal{V}}
\leq C\left( \Vert e_{\boldsymbol{q}}\Vert_{L^{2}(\Omega)} 
+\Vert e_{u}\Vert_{L^{2}(\Omega)}+\Vert e_{u}\Vert_{L^{2}(\partial\Omega,\boldsymbol{\beta})} \right).
\end{align}

According to (\ref{ns_ineq1}, \ref{ns_ineq2}), we have
\begin{align*}
& \Vert e_{\boldsymbol{q}}\Vert_{L^{2}(\Omega)} + \Vert e_{u}\Vert_{L^{2}(\Omega)}
+ \Vert e_{u}\Vert_{L^{2}(\partial\Omega,\boldsymbol{\beta})}\leq C \theta .
\end{align*}
It is easy to see $\Vert e_{\boldsymbol{q}}+\epsilon^{1/2}\nabla e_{u}\Vert_{L^{2}(\Omega)}
+\Vert \epsilon^{1/2}\nabla\cdot e_{\boldsymbol{q}} + \boldsymbol{\beta}\cdot \nabla e_{u} + c e_{u}\Vert_{L^{2}(\Omega)}
 \leq \sqrt{2}\theta$.
So, we have
\begin{align*}
\Vert e_{\boldsymbol{q}}\Vert_{L^{2}(\Omega)} + \Vert e_{u}\Vert_{L^{2}(\Omega)} 
+\epsilon^{1/2}\Vert \nabla e_{u}\Vert_{L^{2}(\Omega)}
+ \Vert e_{u}\Vert_{L^{2}(\partial\Omega,\boldsymbol{\beta})}
+\Vert \epsilon^{1/2}\nabla\cdot e_{\boldsymbol{q}} + \boldsymbol{\beta}\cdot \nabla e_{u}\Vert_{L^{2}(\Omega)}
\leq C \theta.
\end{align*}
By the definition (\ref{norm_U}) of norm of $\mathcal{U}$, we have
\begin{align*}
\Vert (e_{\boldsymbol{q}}, e_{u})\Vert_{\mathcal{U}}\leq C\theta.
\end{align*}
Then, by (\ref{inf_sup_eq}), we have (\ref{numerical_stability1}) immediately.

By (\ref{fem_space_assump2}), if $0<\epsilon^{1/2}\leq  h_{K}$ for any $K\in \mathcal{T}_{h}$, we have
\begin{equation*}
\epsilon^{1/2}\Vert \nabla \cdot e_{\boldsymbol{q}}\Vert_{L^{2}(\Omega)} \leq C \Vert e_{\boldsymbol{q}}\Vert_{L^{2}(\Omega)}.
\end{equation*}
Combing the above inequality with (\ref{numerical_stability1}),
we can conclude that (\ref{numerical_stability2}) holds if $0<\epsilon^{1/2}\leq h_{K}$ for any $K\in \mathcal{T}_{h}$.
\end{proof}

\begin{remark}
We want to emphasize that the proof of Lemma~\ref{lemma_numerical_stability} holds for any subspace 
$\boldsymbol{Q}_{h}\times W_{h}$ of $H(\text{div},\Omega)\times H^{1}(\Omega)$, which satisfies (\ref{fem_space_assump1}) 
and (\ref{fem_space_assump2}). It implies that Lemma~\ref{lemma_numerical_stability} is still true if we impose 
Dirichlet boundary condition strongly, like in \cite{CaiWestphal:ls_div}.
\end{remark}

\begin{lemma}
\label{lemma_norm_equivalent}
If the assumptions (\ref{AS_beta}) hold, 
\begin{align}
\label{norm_equivalent}
& C_{0} \Vert (e_{\boldsymbol{q}}, e_{u})\Vert_{\mathcal{U}}\\
\nonumber
\leq & \left( \Vert e_{\boldsymbol{q}}+\epsilon^{1/2}\nabla e_u\Vert_{L^{2}(\Omega)}
+ \Vert \epsilon^{1/2}\nabla\cdot e_{\boldsymbol{q}} +\boldsymbol{\beta}\cdot \nabla e_u + c e_u\Vert_{L^{2}(\Omega)}
+\Vert e_{u}\Vert_{L^{2}(\partial\Omega,\boldsymbol{\beta})}\right)\\
\nonumber
\leq & C_{1} \Vert (e_{\boldsymbol{q}}, e_{u})\Vert_{\mathcal{U}},\quad \forall (e_{\boldsymbol{q}}, e_{u})\in \mathcal{U}_{h}.
\end{align}
Here, the constants $C_{0}$ and $C_{1}$ are independent of $\epsilon$ and $h$.
\end{lemma}

\begin{proof}
According to Lemma~\ref{lemma_numerical_stability} and explicit formulation of operator $T$ in (\ref{T_op_explicit}), 
\begin{align*}
& C_{0} \Vert (e_{\boldsymbol{q}}, e_{u})\Vert_{\mathcal{U}}\\
\leq & \sup_{0\neq (\boldsymbol{p},w)\in \mathcal{U}_{h}}\dfrac{b((e_{\boldsymbol{q}}, e_{u}) , (\boldsymbol{p}+\epsilon^{1/2}\nabla w, 
\epsilon^{1/2}\nabla\cdot\boldsymbol{p}+\boldsymbol{\beta}\cdot \nabla w + cw , w|_{\partial\Omega} ) )}
{ \Vert (\boldsymbol{p}+\epsilon^{1/2}\nabla w, 
\epsilon^{1/2}\nabla\cdot\boldsymbol{p}+\boldsymbol{\beta}\cdot \nabla w + cw , w|_{\partial\Omega} )\Vert_{\mathcal{V}} }.
\end{align*}
By the definition of bi-linear form $b$ in (\ref{dpg_bilinear_form}) and the inner product of $\mathcal{V}$ in (\ref{inner_product_V}), we have 
\begin{align*}
& \sup_{0\neq (\boldsymbol{p},w)\in \mathcal{U}_{h}}\dfrac{b((e_{\boldsymbol{q}}, e_{u}) , (\boldsymbol{p}+\epsilon^{1/2}\nabla w, 
\epsilon^{1/2}\nabla\cdot\boldsymbol{p}+\boldsymbol{\beta}\cdot \nabla w + cw , w|_{\partial\Omega} ) )}
{ \Vert (\boldsymbol{p}+\epsilon^{1/2}\nabla w, 
\epsilon^{1/2}\nabla\cdot\boldsymbol{p}+\boldsymbol{\beta}\cdot \nabla w + cw , w|_{\partial\Omega} )\Vert_{\mathcal{V}} }\\
= & \left( \Vert e_{\boldsymbol{q}}+\epsilon^{1/2}\nabla e_u\Vert_{L^{2}(\Omega)}^{2}
+ \Vert \epsilon^{1/2}\nabla\cdot e_{\boldsymbol{q}} +\boldsymbol{\beta}\cdot \nabla e_u + c e_u\Vert_{L^{2}(\Omega)}^{2}
+\Vert e_{u}\Vert_{L^{2}(\partial\Omega,\boldsymbol{\beta})}^{2}\right)^{1/2}.
\end{align*}
So, we have 
\begin{align*}
& C_{0} \Vert (e_{\boldsymbol{q}}, e_{u})\Vert_{\mathcal{U}}\\
\leq & \left( \Vert e_{\boldsymbol{q}}+\epsilon^{1/2}\nabla e_u\Vert_{L^{2}(\Omega)}
+ \Vert \epsilon^{1/2}\nabla\cdot e_{\boldsymbol{q}} +\boldsymbol{\beta}\cdot \nabla e_u + c e_u\Vert_{L^{2}(\Omega)}
+\Vert e_{u}\Vert_{L^{2}(\partial\Omega,\boldsymbol{\beta})}\right).
\end{align*}
On the other hand, by the definition of norm of $\mathcal{U}$ in (\ref{norm_U}), we have
\begin{align*}
& \left( \Vert e_{\boldsymbol{q}}+\epsilon^{1/2}\nabla e_u\Vert_{L^{2}(\Omega)}
+ \Vert \epsilon^{1/2}\nabla\cdot e_{\boldsymbol{q}} +\boldsymbol{\beta}\cdot \nabla e_u + c e_u\Vert_{L^{2}(\Omega)}
+\Vert e_{u}\Vert_{L^{2}(\partial\Omega,\boldsymbol{\beta})}\right)\\
\leq & C_{1} \Vert (e_{\boldsymbol{q}}, e_{u})\Vert_{\mathcal{U}}.
\end{align*}
So, we can conclude that the proof is complete.
\end{proof}

\subsection{Error analysis}
Let $u$ be the solution of equation (\ref{cd_eqs}) and $\boldsymbol{q}= - \epsilon^{1/2} \nabla u$, then, 
\begin{align}
\label{ls_formulaiton_exact_solution}
& \left( \boldsymbol{q}+\epsilon^{1/2}\nabla u, \boldsymbol{p}+\epsilon^{1/2}\nabla w\right)_{\Omega} \\
\nonumber
& \quad+ \left( \epsilon^{1/2}\nabla\cdot \boldsymbol{q} + \boldsymbol{\beta}\cdot \nabla u + c u, 
\epsilon^{1/2}\nabla\cdot \boldsymbol{p} + \boldsymbol{\beta}\cdot \nabla w + c w \right)_{\Omega}\\
\nonumber
& \quad +\Sigma_{F\in \mathcal{E}_h^{\partial}} h_{F}^{-1}
\langle \left(\epsilon+\max (-\boldsymbol{\beta}\cdot \boldsymbol{n}(x),0)\right) u, w\rangle_{F}\\
\nonumber
 = & (f, \epsilon^{1/2}\nabla\cdot \boldsymbol{p} + \boldsymbol{\beta}\cdot \nabla w + c w)_{\Omega}\\
 \nonumber
&\quad +\Sigma_{F\in \mathcal{E}_h^{\partial}} h_{F}^{-1}
\langle \left(\epsilon+\max (-\boldsymbol{\beta}\cdot \boldsymbol{n}(x),0)\right) u, w\rangle_{F},
\quad \forall (\boldsymbol{p},w)\in \mathcal{U}_{h}.
\end{align}

We define the projection errors
\begin{subequations}
\label{projection_errors}
\begin{align}
\label{projection_error1}
& e_{\boldsymbol{q}} := \Pi_{\boldsymbol{Q}} \boldsymbol{q} - \boldsymbol{q}_{h},\\
\label{projection_error2}
& e_{u} := \Pi_{W}u - u_{h}.
\end{align}
\end{subequations}
Here, $\Pi_{\boldsymbol{Q}}:H^{1}(\Omega;\mathbb{R}^{d})\rightarrow\boldsymbol{Q}_{h}$ is the standard Raviart-Thomas projection, 
and $\Pi_{W}:H^{1}(\Omega)\rightarrow W_{h}$ is the standard interpolation.
Then, we have the following error equation:
\begin{align}
\label{error_eq}
& \left( e_{\boldsymbol{q}}+\epsilon^{1/2}\nabla e_u, \boldsymbol{p}+\epsilon^{1/2}\nabla w\right)_{\Omega} \\
\nonumber
& \quad+ \left( \epsilon^{1/2}\nabla\cdot e_{\boldsymbol{q}} + \boldsymbol{\beta}\cdot \nabla e_u + c e_u, 
\epsilon^{1/2}\nabla\cdot \boldsymbol{p} + \boldsymbol{\beta}\cdot \nabla w + c w \right)_{\Omega}\\
\nonumber
& \quad +\Sigma_{F\in \mathcal{E}_h^{\partial}} h_{F}^{-1}
\langle \left(\epsilon+\max (-\boldsymbol{\beta}\cdot \boldsymbol{n}(x),0)\right) e_u, w\rangle_{F}\\
\nonumber
= & \left(  (\Pi_{\boldsymbol{Q}}\boldsymbol{q}-\boldsymbol{q})+\epsilon^{1/2}\nabla (\Pi_{W}u -u), 
\boldsymbol{p}+\epsilon^{1/2}\nabla w\right)_{\Omega} \\
\nonumber
& \quad+\left( \epsilon^{1/2}\nabla\cdot (\Pi_{\boldsymbol{Q}}\boldsymbol{q}-\boldsymbol{q}) 
+\boldsymbol{\beta}\cdot \nabla (\Pi_{W}u -u) +c(\Pi_{W}u - u), 
\epsilon^{1/2}\nabla\cdot \boldsymbol{p} + \boldsymbol{\beta}\cdot \nabla w + c w \right)_{\Omega}\\
\nonumber
& \quad +\Sigma_{F\in \mathcal{E}_h^{\partial}} h_{F}^{-1}
\langle \left(\epsilon+\max (-\boldsymbol{\beta}\cdot \boldsymbol{n}(x),0)\right) (\Pi_{W}u - u), w\rangle_{F},
\end{align}
for any $(\boldsymbol{p},w)\in \mathcal{U}_{h}$.

Since the proof of Theorem~\ref{MainTh2} is similar to that of Theorem~\ref{MainTh1}, we only prove 
Theorem~\ref{MainTh1} in the following.
\begin{proof}(Proof of Theorem~\ref{MainTh1})
According to Lemma~\ref{lemma_numerical_stability}, we have
\begin{align*}
\Vert (e_{\boldsymbol{q}},e_{u})\Vert_{\mathcal{U}}
\leq C \sup_{0\neq (\boldsymbol{r},v,\mu)\in T(\mathcal{U}_{h})}
\dfrac{b( (e_{\boldsymbol{q}},e_{u}), (\boldsymbol{r},v,\mu) )}{ \Vert (\boldsymbol{r},v,\mu)\Vert_{\mathcal{V}} }.
\end{align*}
By the definition of bi-linear form $b$ in (\ref{dpg_bilinear_form}) and above inequality, we have
\begin{align*}
& \Vert (e_{\boldsymbol{q}},e_{u})\Vert_{\mathcal{U}}\\
\leq & C \sup_{0\neq (\boldsymbol{r},v,\mu)\in T(\mathcal{U}_{h})}
\dfrac{b( (\Pi_{\boldsymbol{Q}}\boldsymbol{q}-\boldsymbol{q}, \Pi_{W}u - u), (\boldsymbol{r},v,\mu) )}
{ \Vert (\boldsymbol{r},v,\mu)\Vert_{\mathcal{V}} }\\
\leq & C \Vert (\Pi_{\boldsymbol{Q}}\boldsymbol{q}-\boldsymbol{q}, \Pi_{W}u - u)\Vert_{\mathcal{U}}.
\end{align*}
Then, by approximation properties of $(\Pi_{\boldsymbol{Q}},\Pi_{W})$, we immediately have (\ref{convergence_rate_global}).

In remaining part of the proof, we show (\ref{convergence_rate_global_refined}) holds under the assumption 
$\boldsymbol{\beta}\in W^{2,\infty}(\mathcal{T}_{h})$ and $c\in W^{1,\infty}(\mathcal{T}_{h})$.
 
By the error equation (\ref{error_eq}) and Lemma~\ref{lemma_norm_equivalent}, we have
\begin{align*}
& C_{0}^{2} \Vert (e_{\boldsymbol{q}}, e_{u})\Vert_{\mathcal{U}}^{2}\\
\leq  & \left(  (\Pi_{\boldsymbol{Q}}\boldsymbol{q}-\boldsymbol{q})+\epsilon^{1/2}\nabla (\Pi_{W}u -u), 
e_{\boldsymbol{q}}+\epsilon^{1/2}\nabla e_{u}\right)_{\Omega} \\
& \quad+\left( \epsilon^{1/2}\nabla\cdot (\Pi_{\boldsymbol{Q}}\boldsymbol{q}-\boldsymbol{q}) 
+\boldsymbol{\beta}\cdot \nabla (\Pi_{W}u -u) +c(\Pi_{W}u - u), 
\epsilon^{1/2}\nabla\cdot e_{\boldsymbol{q}} + \boldsymbol{\beta}\cdot \nabla e_{u} + c e_{u} \right)_{\Omega}\\
& \quad +\Sigma_{F\in \mathcal{E}_h^{\partial}} h_{F}^{-1}
\langle \left(\epsilon+\max (-\boldsymbol{\beta}\cdot \boldsymbol{n}(x),0)\right) (\Pi_{W}u - u), e_{u}\rangle_{F}.
\end{align*}

Since $\Pi_{\boldsymbol{Q}}$ is standard Raviart-Thomas projection, 
we have that $\nabla\cdot\Pi_{\boldsymbol{Q}}\boldsymbol{q}=P_{k+1,h}\nabla\cdot\boldsymbol{q}$ 
where $P_{k+1,h}$ is $L^{2}$ orthogonal projection onto $P_{k+1}(\mathcal{T}_{h})$. Then, we have
\begin{align}
\label{error_ineq}
& C_{0}^{2} \Vert (e_{\boldsymbol{q}}, e_{u})\Vert_{\mathcal{U}}^{2}\\
\nonumber
\leq  & \left(  (\Pi_{\boldsymbol{Q}}\boldsymbol{q}-\boldsymbol{q})+\epsilon^{1/2}\nabla (\Pi_{W}u -u), 
e_{\boldsymbol{q}}+\epsilon^{1/2}\nabla e_{u}\right)_{\Omega} \\
\nonumber
& \quad+\epsilon^{1/2}\left( P_{k+1,h}(\nabla\cdot\boldsymbol{q})-\nabla\cdot\boldsymbol{q}, 
\epsilon^{1/2}\nabla\cdot e_{\boldsymbol{q}} + \boldsymbol{\beta}\cdot \nabla e_{u} + c e_{u} \right)_{\Omega}\\
\nonumber
& \quad+\left(\boldsymbol{\beta}\cdot \nabla (\Pi_{W}u -u) +c(\Pi_{W}u - u), 
\epsilon^{1/2}\nabla\cdot e_{\boldsymbol{q}} + \boldsymbol{\beta}\cdot \nabla e_{u} + c e_{u} \right)_{\Omega}\\
\nonumber
& \quad +\Sigma_{F\in \mathcal{E}_h^{\partial}} h_{F}^{-1}
\langle \left(\epsilon+\max (-\boldsymbol{\beta}\cdot \boldsymbol{n}(x),0)\right) (\Pi_{W}u - u), e_{u}\rangle_{F}.
\end{align}
In order to have (\ref{convergence_rate_global_refined}), we only need to show that 
\begin{align}
\label{ineq_refined}
& \epsilon^{1/2}\left( P_{k+1,h}(\nabla\cdot\boldsymbol{q})-\nabla\cdot\boldsymbol{q}, 
\epsilon^{1/2}\nabla\cdot e_{\boldsymbol{q}} + \boldsymbol{\beta}\cdot \nabla e_{u} + c e_{u} \right)_{\Omega}\\
\nonumber
\leq & C h^{k+1}\Vert (e_{\boldsymbol{q}}, e_{u})\Vert_{\mathcal{U}}\cdot \Vert u\Vert_{H^{k+2}(\Omega)}.
\end{align}
Since $\nabla\cdot e_{\boldsymbol{q}}\in P_{k+1}(\mathcal{T}_{h})$, we have 
\begin{align*}
\epsilon^{1/2}\left( P_{k+1,h}(\nabla\cdot\boldsymbol{q})-\nabla\cdot\boldsymbol{q}, 
\epsilon^{1/2}\nabla\cdot e_{\boldsymbol{q}}\right)_{\Omega}=0.
\end{align*}
We define $\boldsymbol{P}_{1,h}$ and $P_{0,h}$ to be $L^{2}$ orthogonal projections onto $P_{1}(\mathcal{T}_{h};\mathbb{R}^{d})$ and 
$P_{0}(\mathcal{T}_{h})$, respectively. Then, we have
\begin{align*}
& \epsilon^{1/2}\left( P_{k+1,h}(\nabla\cdot\boldsymbol{q})-\nabla\cdot\boldsymbol{q}, 
\epsilon^{1/2}\nabla\cdot e_{\boldsymbol{q}} + \boldsymbol{\beta}\cdot \nabla e_{u} + c e_{u} \right)_{\Omega}\\
= & \epsilon^{1/2}\left( P_{k+1,h}(\nabla\cdot\boldsymbol{q})-\nabla\cdot\boldsymbol{q}, 
( \boldsymbol{\beta}-\boldsymbol{P}_{1,h}\boldsymbol{\beta})\cdot \nabla e_{u} + (c - P_{0,h}c) e_{u} \right)_{\Omega}.
\end{align*} 
By approximation properties of $\boldsymbol{P}_{1,h}$ and $P_{0,h}$, we have that for any $K\in\mathcal{T}_{h}$,
\begin{align*}
\Vert \boldsymbol{\beta}-\boldsymbol{P}_{1,h}\boldsymbol{\beta} \Vert_{L^{\infty}(K)}\leq C h_{K}^{2} 
\Vert \boldsymbol{\beta}\Vert_{W^{2,\infty}(K)}, 
\quad \Vert c - P_{0,h}c\Vert_{L^{\infty}(K)} \leq C h_{K} \Vert c\Vert_{W^{1,\infty}(K)}.
\end{align*}
Then, by inverse inequality on each element, it is easy to see that (\ref{ineq_refined}) holds under the assumption 
$\boldsymbol{\beta}\in W^{2,\infty}(\mathcal{T}_{h})$ and $c\in W^{1,\infty}(\mathcal{T}_{h})$. 
So, we can conclude that (\ref{convergence_rate_global_refined}) is true.
\end{proof}

\section{Estimate of condition number}
In this section, we prove Theorem~\ref{MainTh3}.

\begin{proof}(Proof of Theorem~\ref{MainTh3}) 
Given $(\boldsymbol{p},w)\in \mathcal{U}_{h}$, we define
\begin{align*}
\theta = & \left( \boldsymbol{p}+\epsilon^{1/2}\nabla w, \boldsymbol{p}+\epsilon^{1/2}\nabla w\right)_{\Omega} \\
& \quad+ \left( \epsilon^{1/2}\nabla\cdot \boldsymbol{p} + \boldsymbol{\beta}\cdot \nabla w + c w, 
\epsilon^{1/2}\nabla\cdot \boldsymbol{p} + \boldsymbol{\beta}\cdot \nabla w + c w \right)_{\Omega}\\
& \quad +\Sigma_{F\in \mathcal{E}_h^{\partial}} h_{F}^{-1}\langle \left(\epsilon
+\max (-\boldsymbol{\beta}\cdot \boldsymbol{n}(x),0)\right)w, w\rangle_{F}.
\end{align*}

Since we assume that meshes are quasi-uniform, then by trace inequality and inverse inequality, 
\begin{align*}
\theta \leq C_{1} h^{-2}\left( \Vert \boldsymbol{p}\Vert_{L^{2}(\Omega)}^{2}+ \Vert w\Vert_{L^{2}(\Omega)}^{2} \right).
\end{align*}
By Lemma~\ref{lemma_norm_equivalent}, we have 
\begin{align*}
C_{0}^{2}\left(\Vert \boldsymbol{p}\Vert_{L^{2}(\Omega)}^{2}+ \Vert w\Vert_{L^{2}(\Omega)}^{2}\right) \leq  \theta.
\end{align*}
So, we can conclude that the proof is complete.
\end{proof}

\section{Numerical results}

In this section, we present numerical studies for some model problems in two dimensional domain. Our test problems are the same to 
those four studied in \cite{AyusoMarini:cdf}. We fix the domain to be the unit square in all the experiments. In order to illustrate the 
effect of weakly imposing Dirichlet boundary condition 
in (\ref{ls_formulation}), we compare numerical results of first order least squares method (\ref{ls_formulation}) 
with those of (\ref{ls_formulation_strong}), which strongly imposes boundary condition. we denote the first order least 
method (\ref{ls_formulation})  as LS--weak, and the first order least squares method (\ref{ls_formulation_strong}) as LS--strong. 
We use unstructured quasi-uniform meshes in all computations.

\subsection{A smooth solution test}
We take $\boldsymbol{\beta} = [1, 1]^T$, 
and choose the diffusion coefficient $\epsilon = 1, 10^{-3}, 10^{-9}$. The source term $f$ is chosen such that
the exact solution $ u(x,y) = \sin(2\pi\, x)\sin(2\pi\,y)$.

Fig.~\ref{smooth_test} shows the $L^2$ convergence 
results for $u_h$ for LS-weak. Actually, that optimal convergence rates are obtained, which is better than 
the theoretical result in Theorem \ref{MainTh1}. The convergence rates of LS-strong is the same.

\begin{figure}
\hfill{}\includegraphics[width=0.32\textwidth]{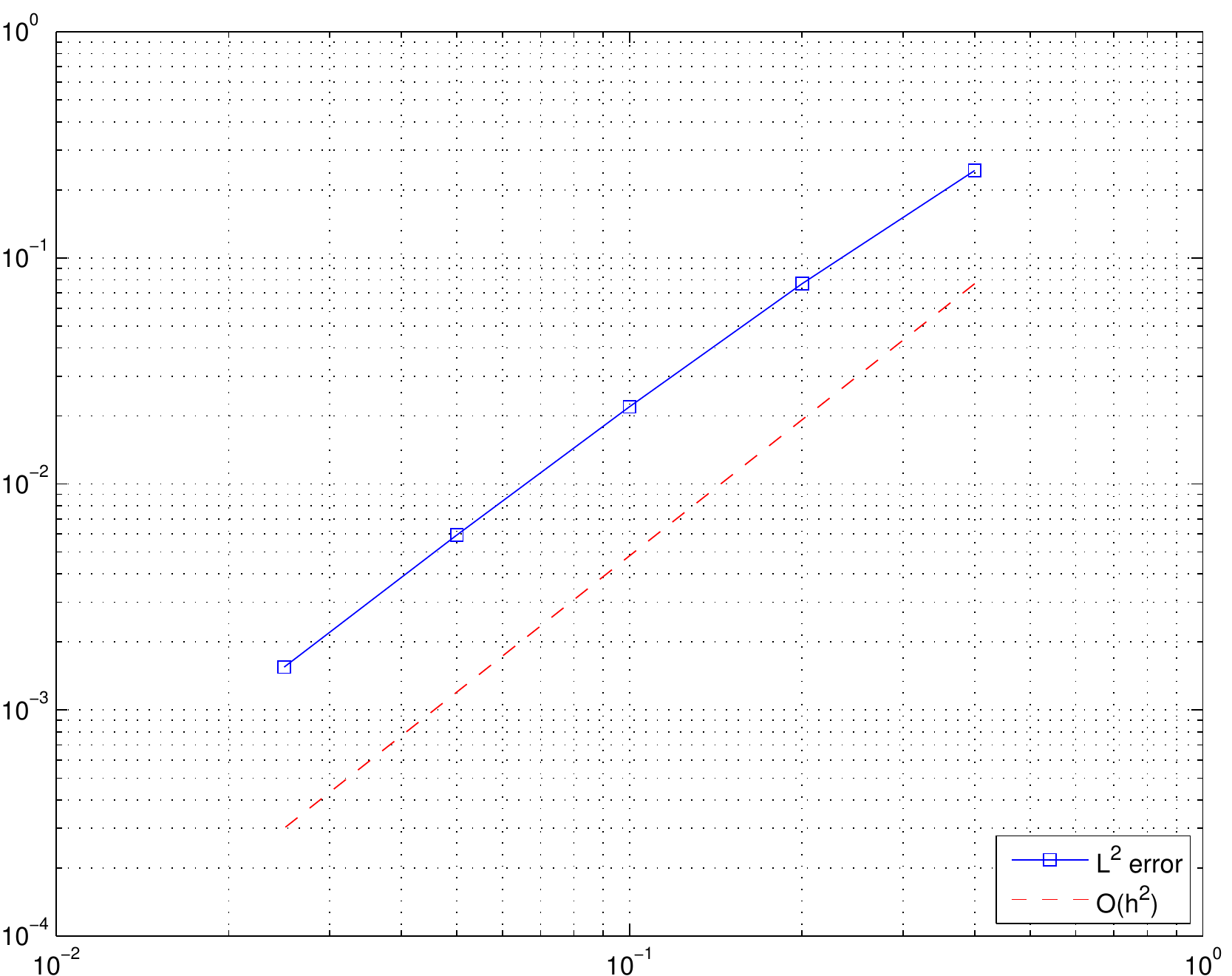}
\hfill{}\includegraphics[width=0.32\textwidth]{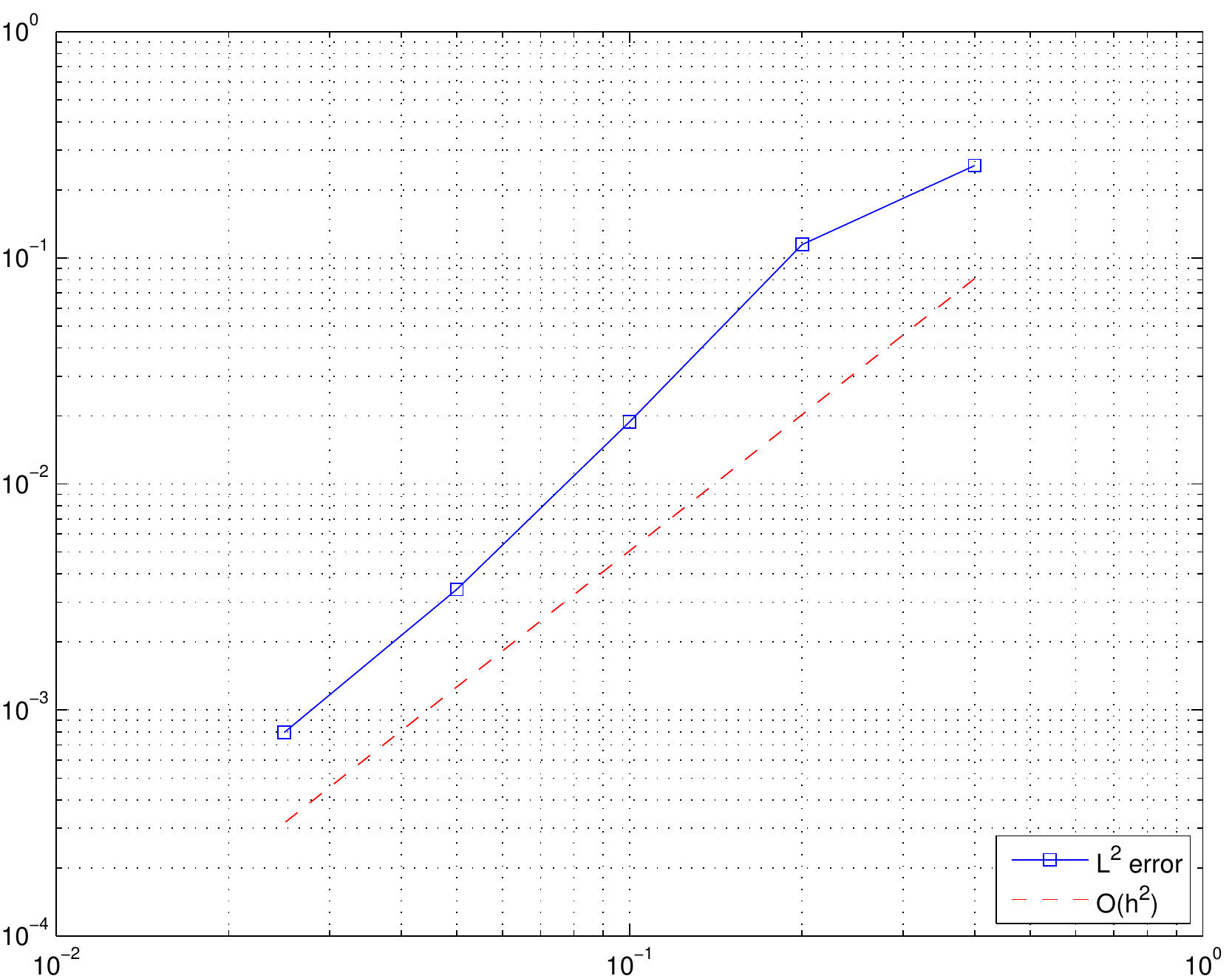}
\hfill{}\includegraphics[width=0.32\textwidth]{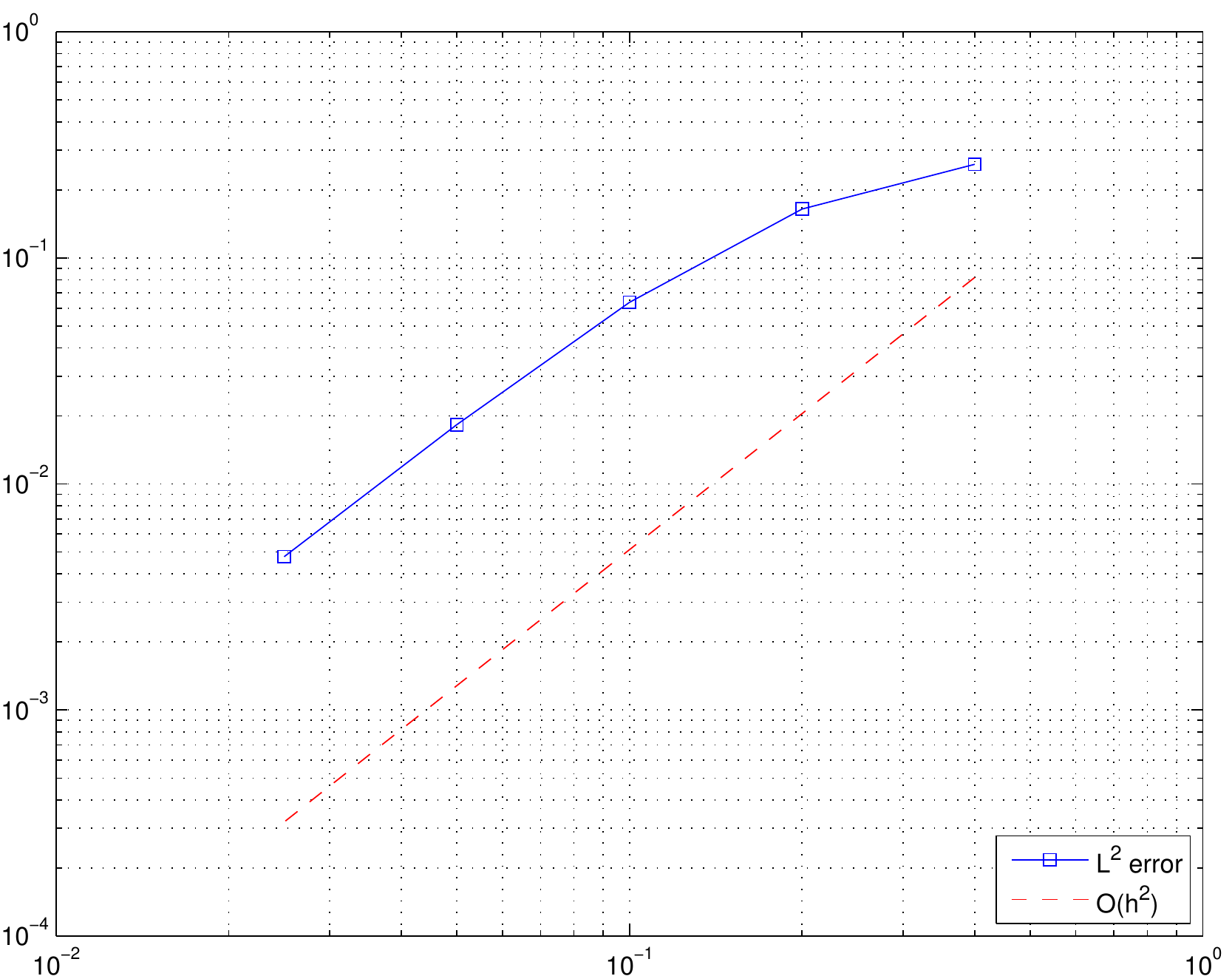}\hfill{}

\hfill{}\includegraphics[width=0.32\textwidth]{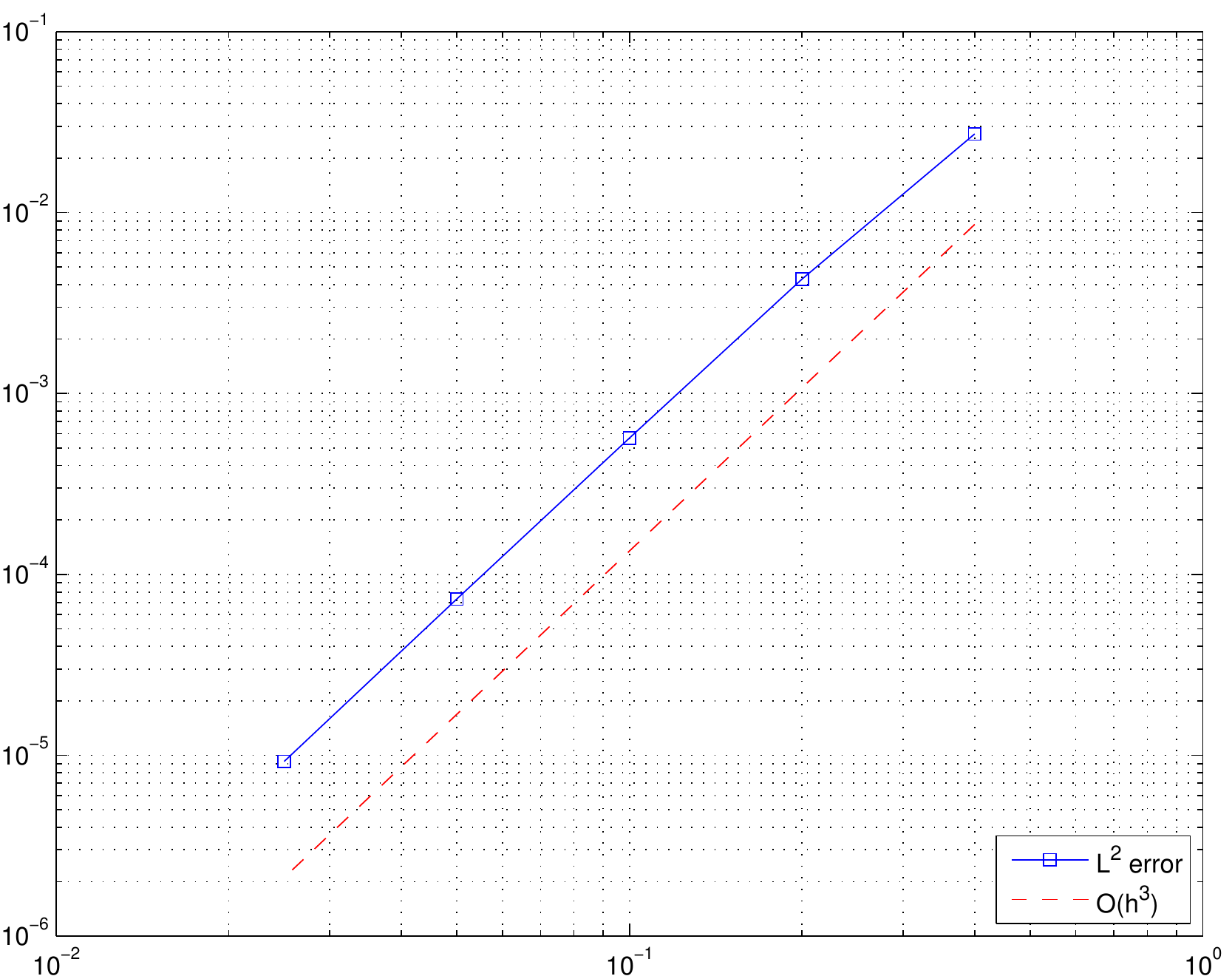}
\hfill{}\includegraphics[width=0.32\textwidth]{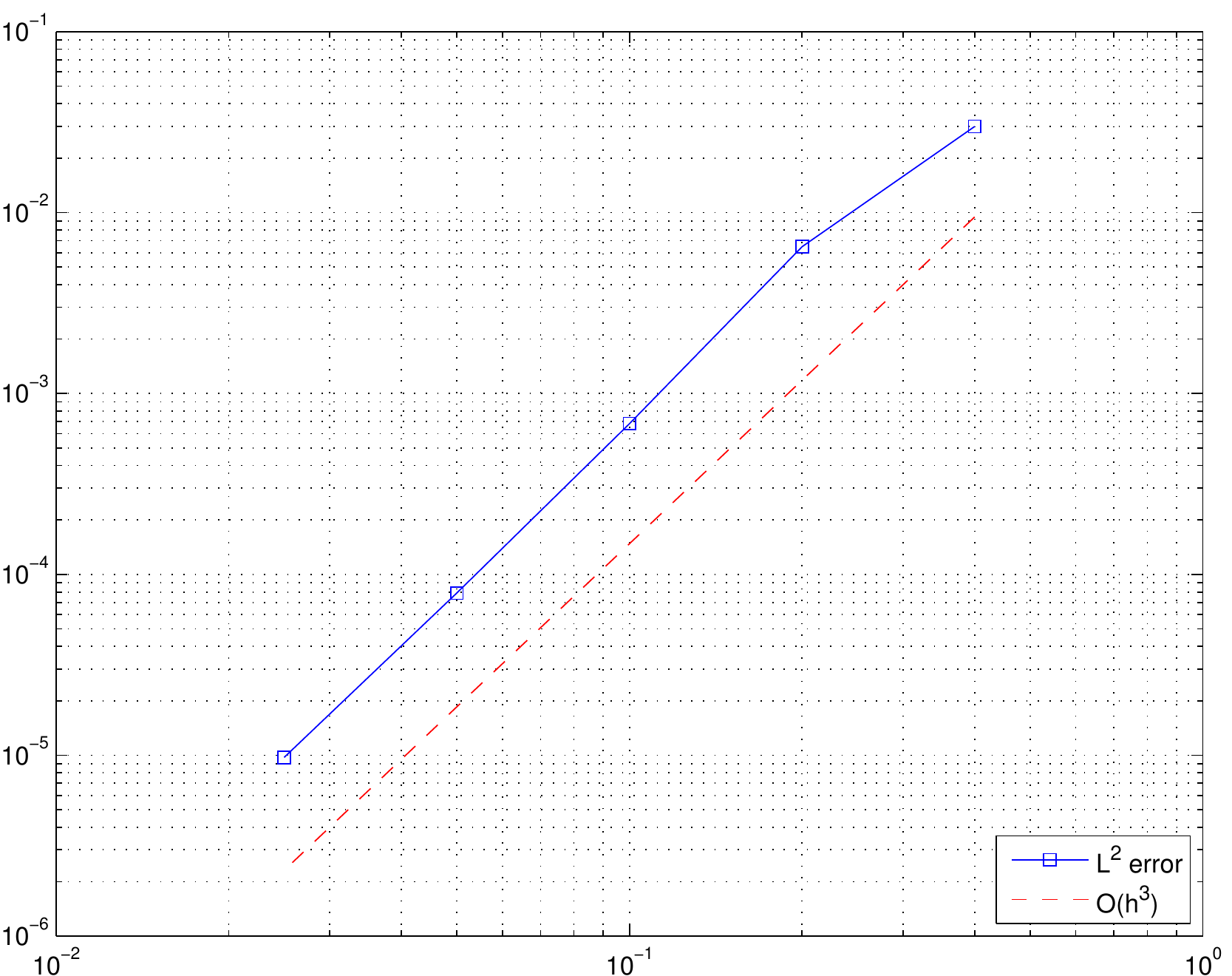}
\hfill{}\includegraphics[width=0.32\textwidth]{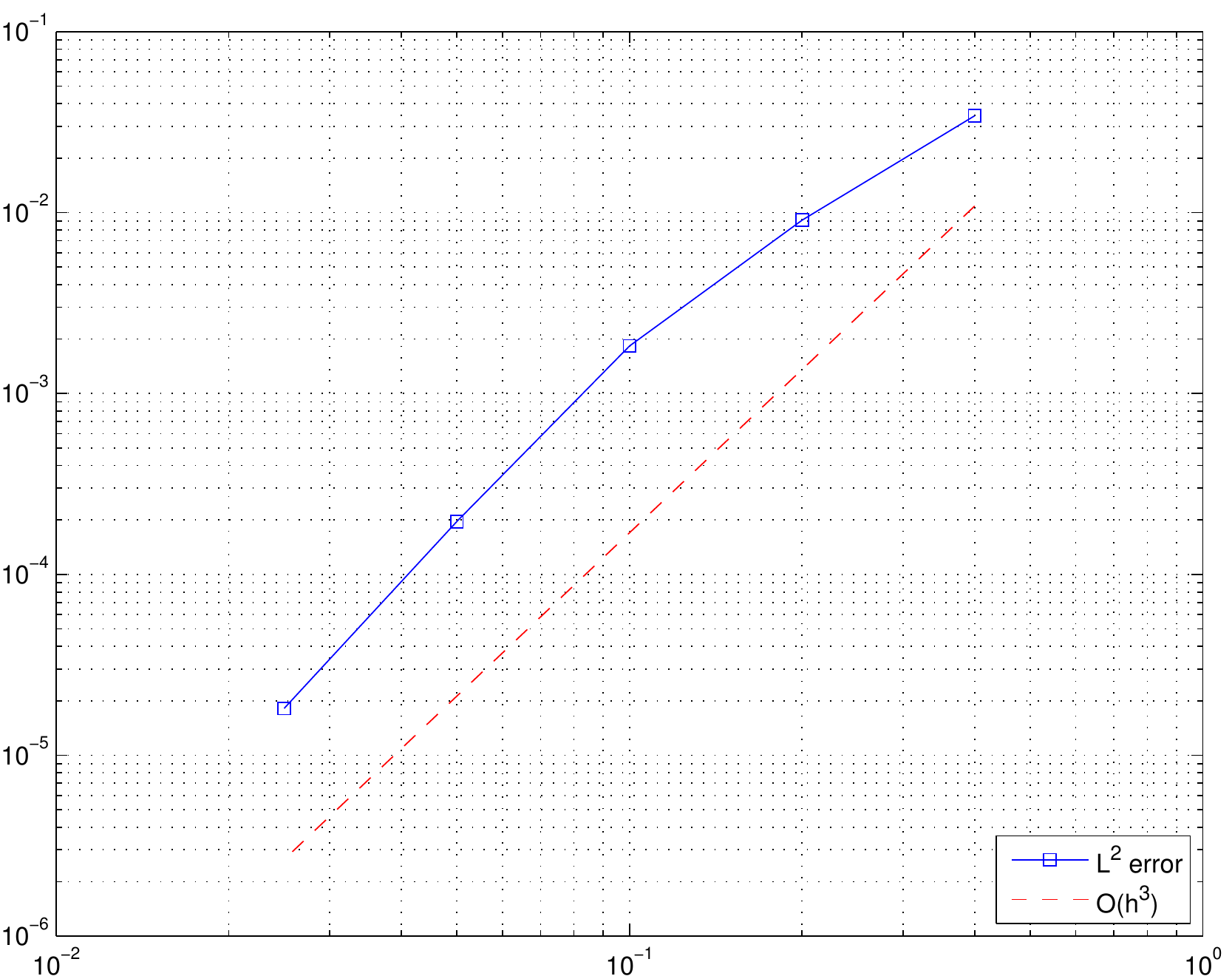}\hfill{}

\hfill{}\includegraphics[width=0.32\textwidth]{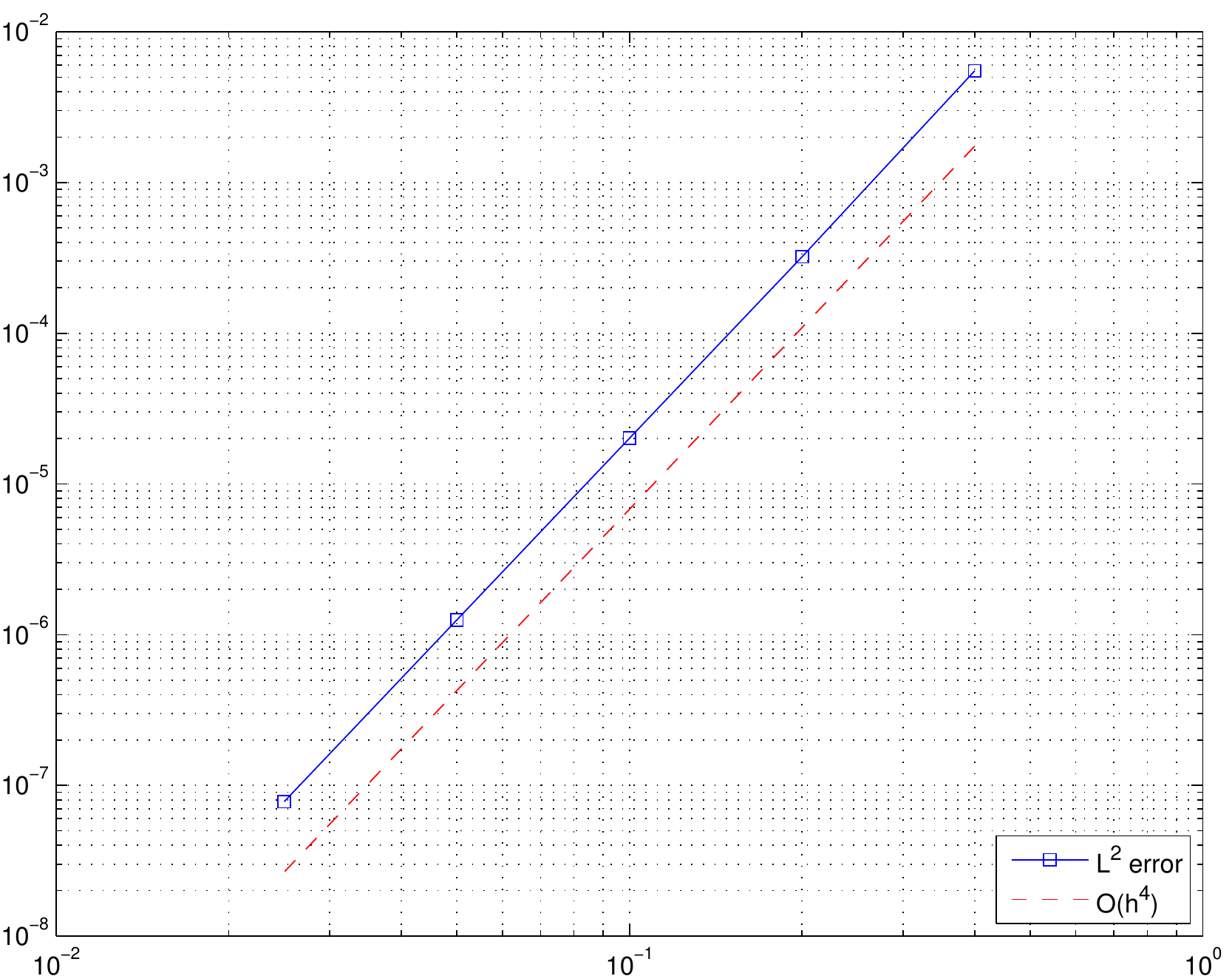}
\hfill{}\includegraphics[width=0.32\textwidth]{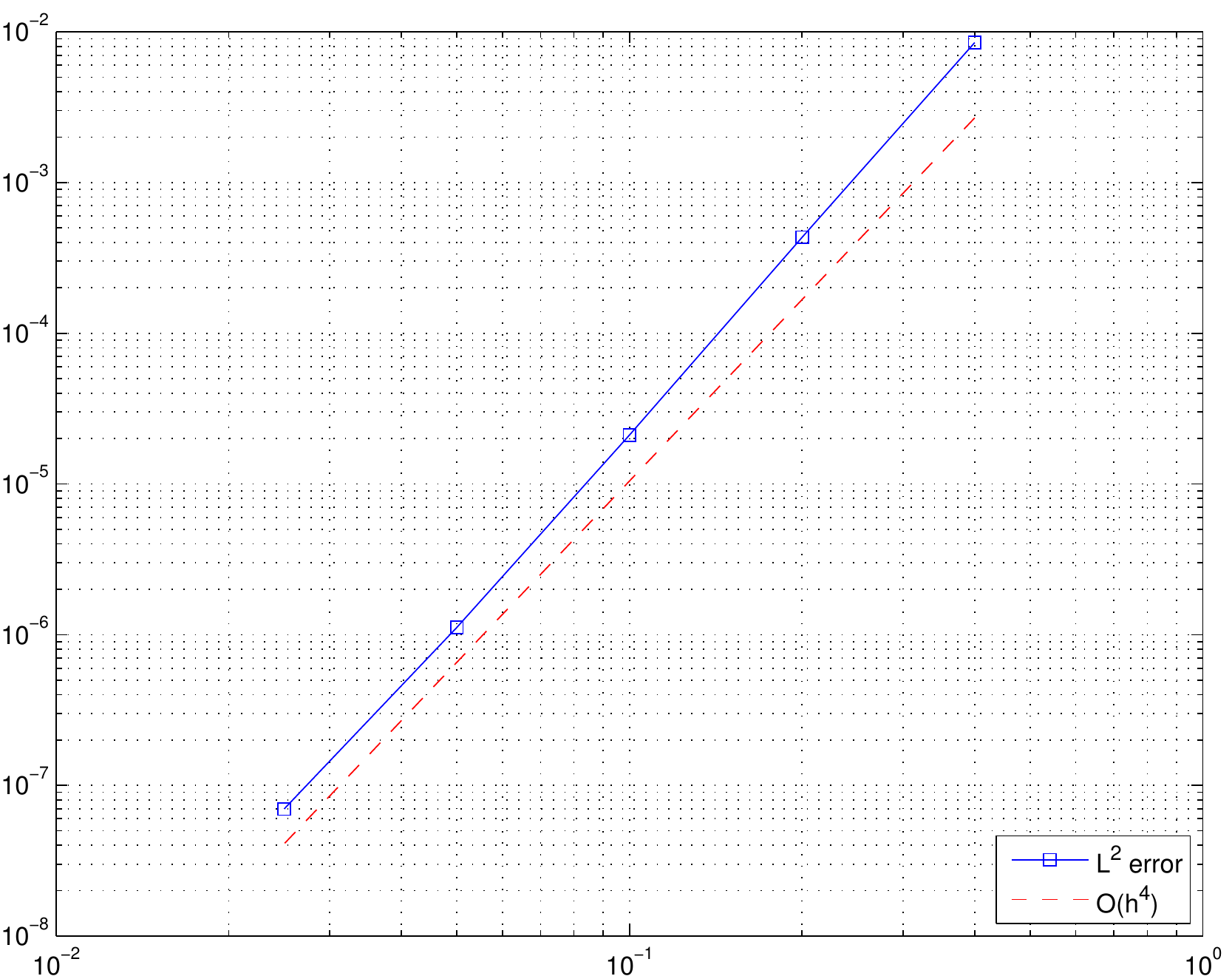}
\hfill{}\includegraphics[width=0.32\textwidth]{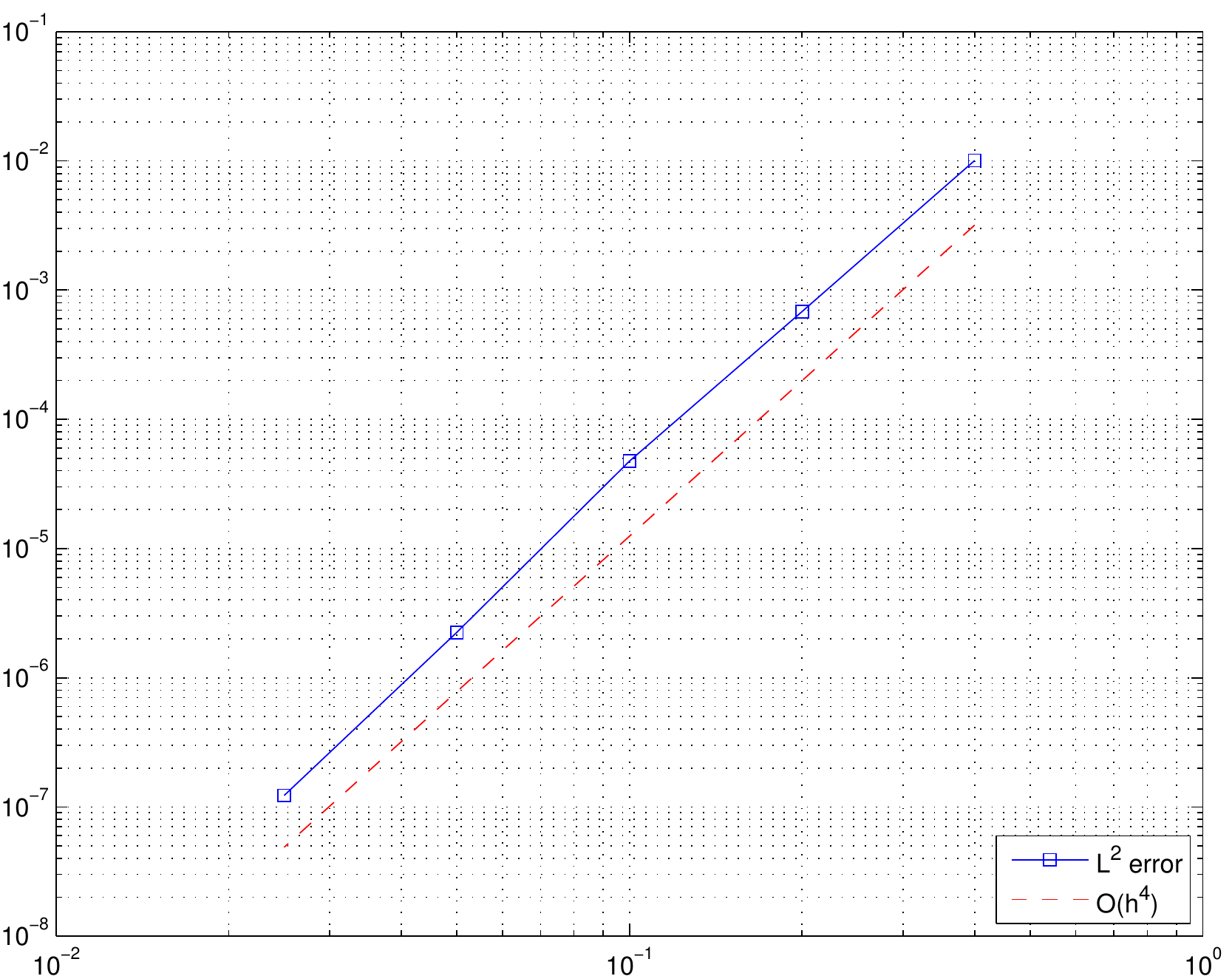}\hfill{}

\caption{Convergence test for a smooth solution. Left--Right: $\varepsilon=1$,
$10^{-3}$ and $10^{-9}$. Top--Bottom: $P1$--$P3$.}

\label{smooth_test}
\end{figure}

\subsection{A rotating flow test}
We take $\epsilon = 10^{-6}$, $\boldsymbol{\beta} = [y - 1/2, 1/2 - x]^T$, and  $f = 0$. 
The solution $u$ is prescribed along the slip $1/2\times [0, 1/2]$ by
\[
 u(1/2,y) = \sin^2(2\pi\,y)\qquad y \in [0, 1/2]. 
\]
We refer to \cite{HughesScovazziBochevBuffa2006} for a detailed description of this test.
In Fig.~\ref{rotating1}, we plot $u_h$ obtained from the two first order least squares methods 
for various polynomial degrees in an unstructured triangular grid of 592 elements. 
Notice that both methods produce similar results, and there is a significant improvement of the result from $P$1 to $P$2.
In addition, the use of higher polynomial degree leads to smoother approximations, see Fig.~\ref{rotating3} for a comparison
between different polynomial degrees.

\begin{figure}[ht!]
\centering
\includegraphics[width = 0.4\textwidth]{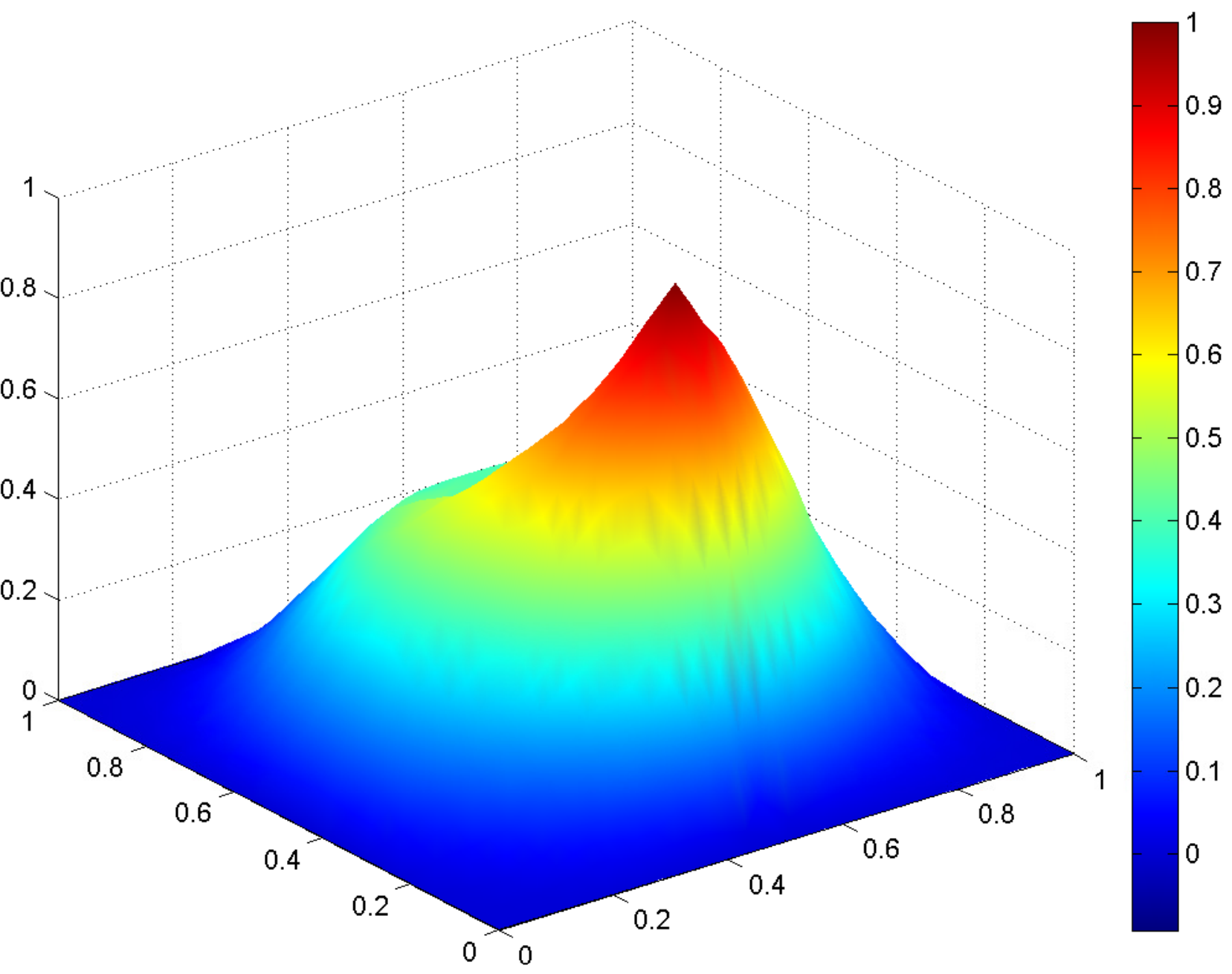}
\includegraphics[width = 0.4\textwidth]{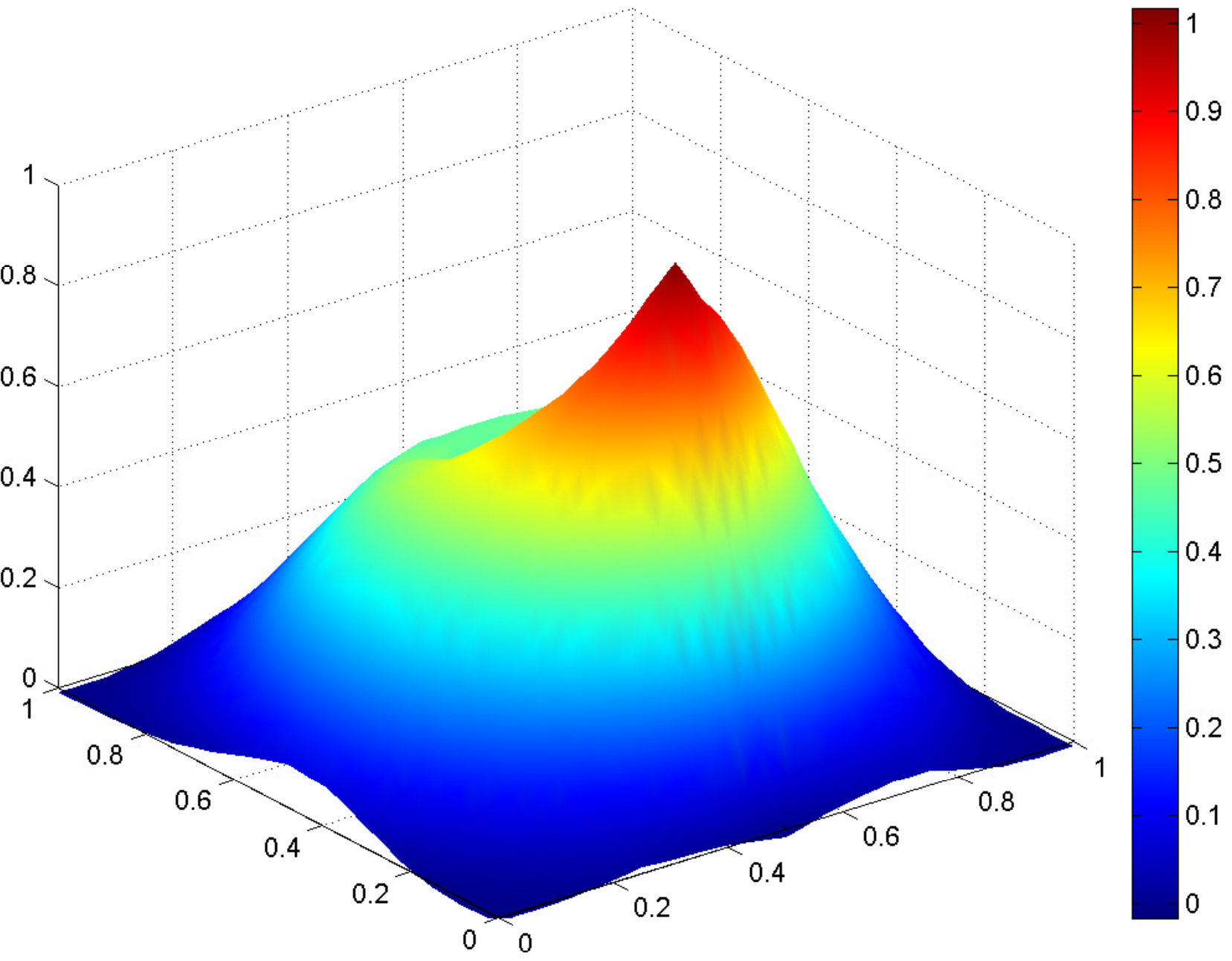}
\includegraphics[width = 0.4\textwidth]{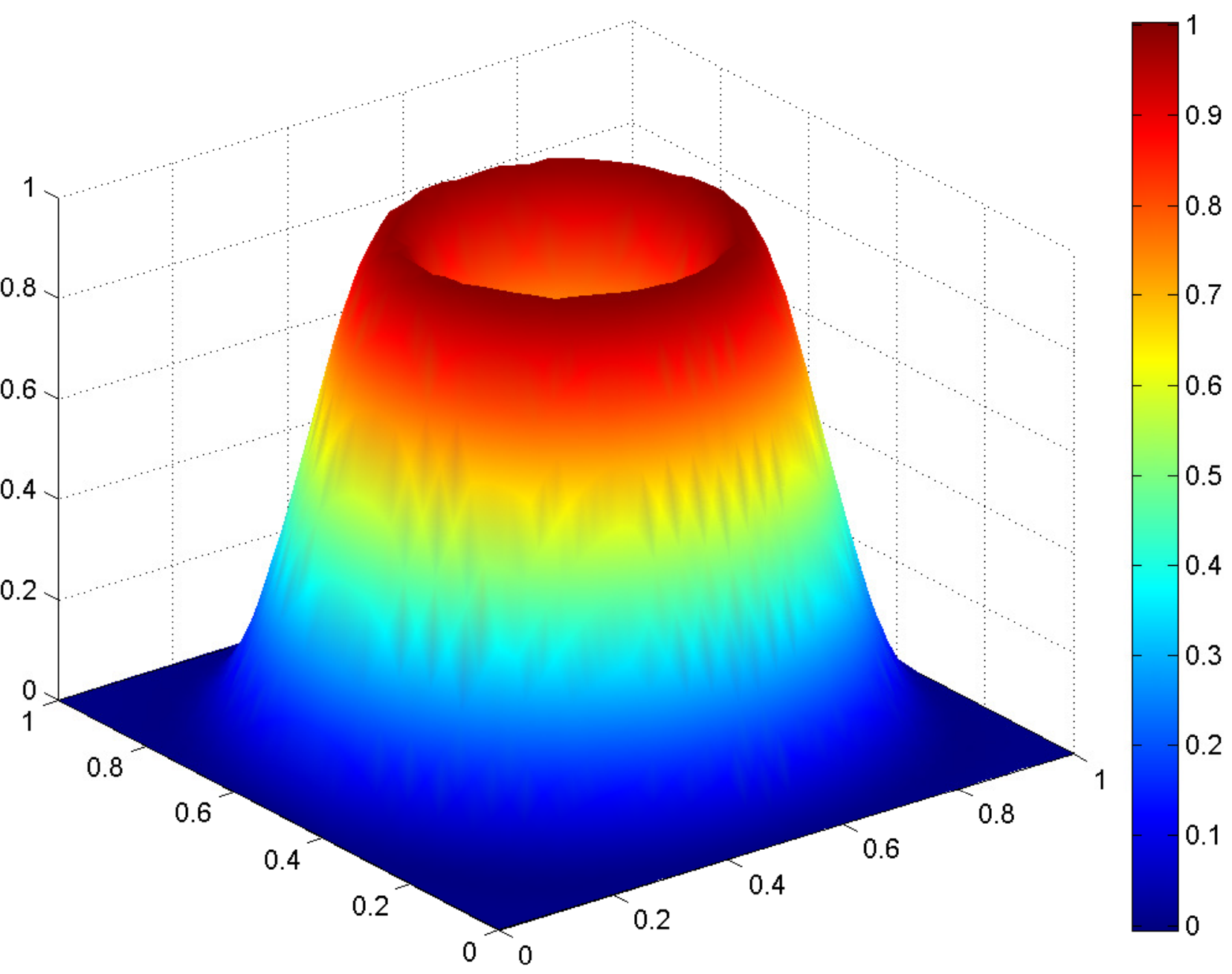}
\includegraphics[width = 0.4\textwidth]{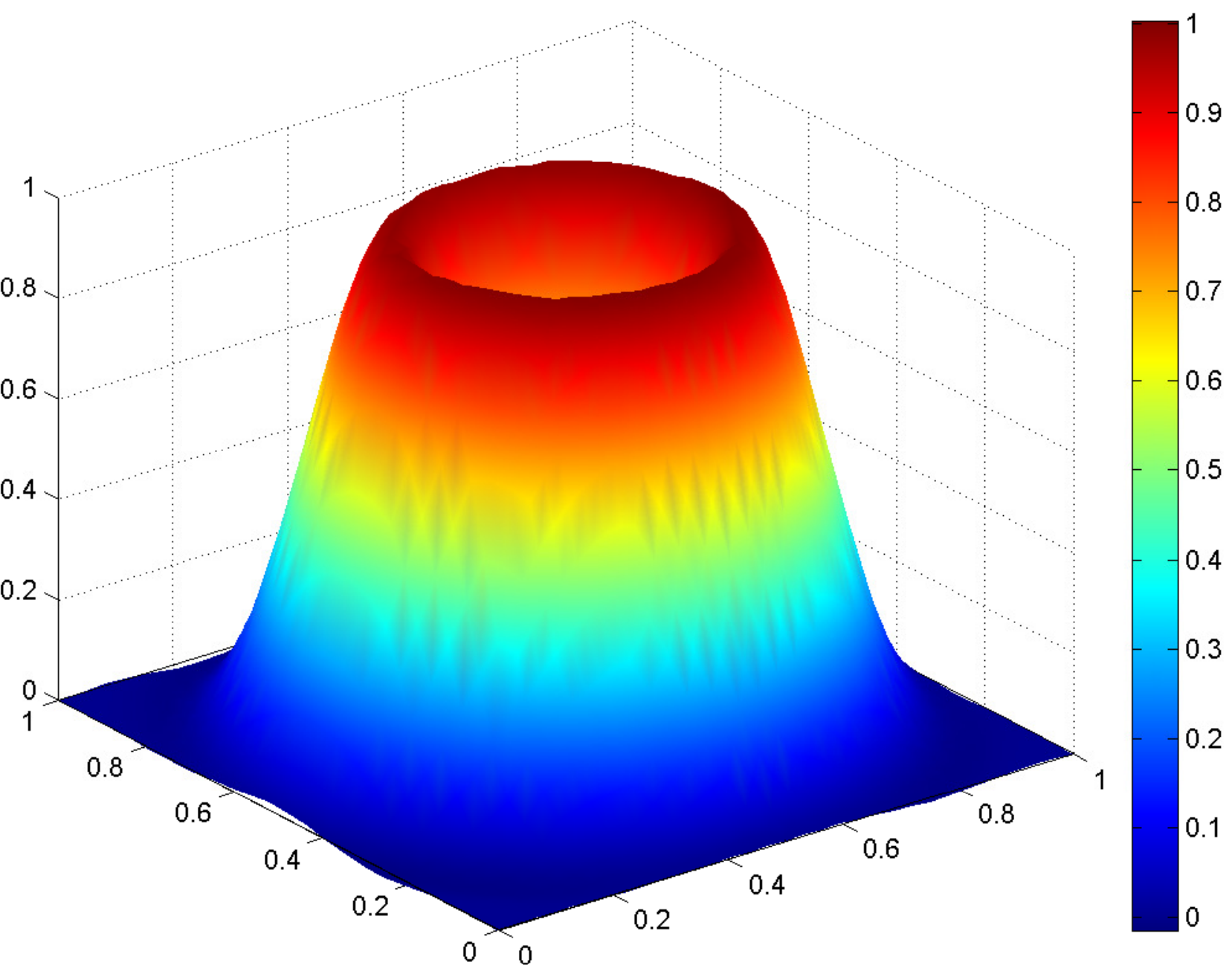}
\includegraphics[width = 0.4\textwidth]{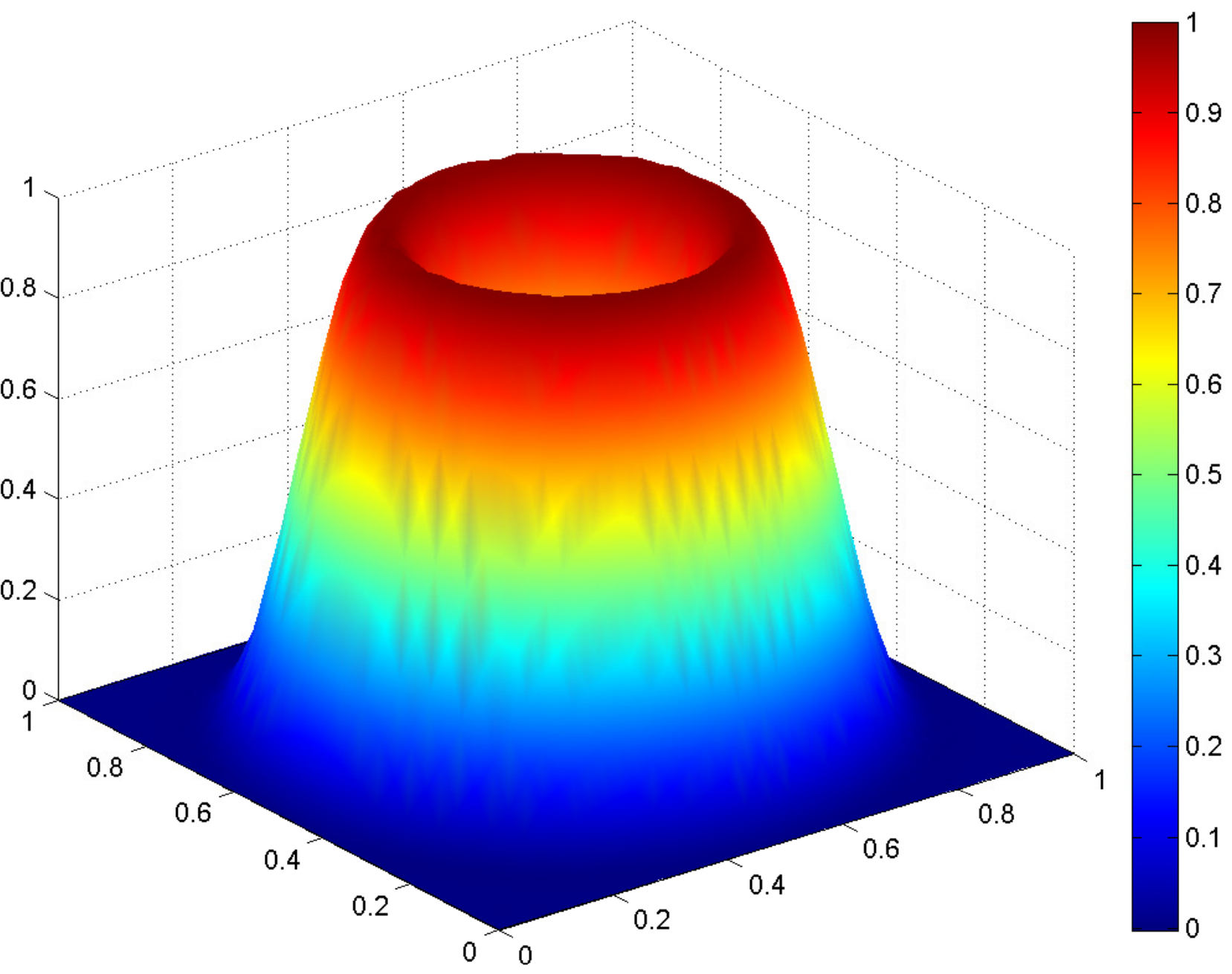}
\includegraphics[width = 0.4\textwidth]{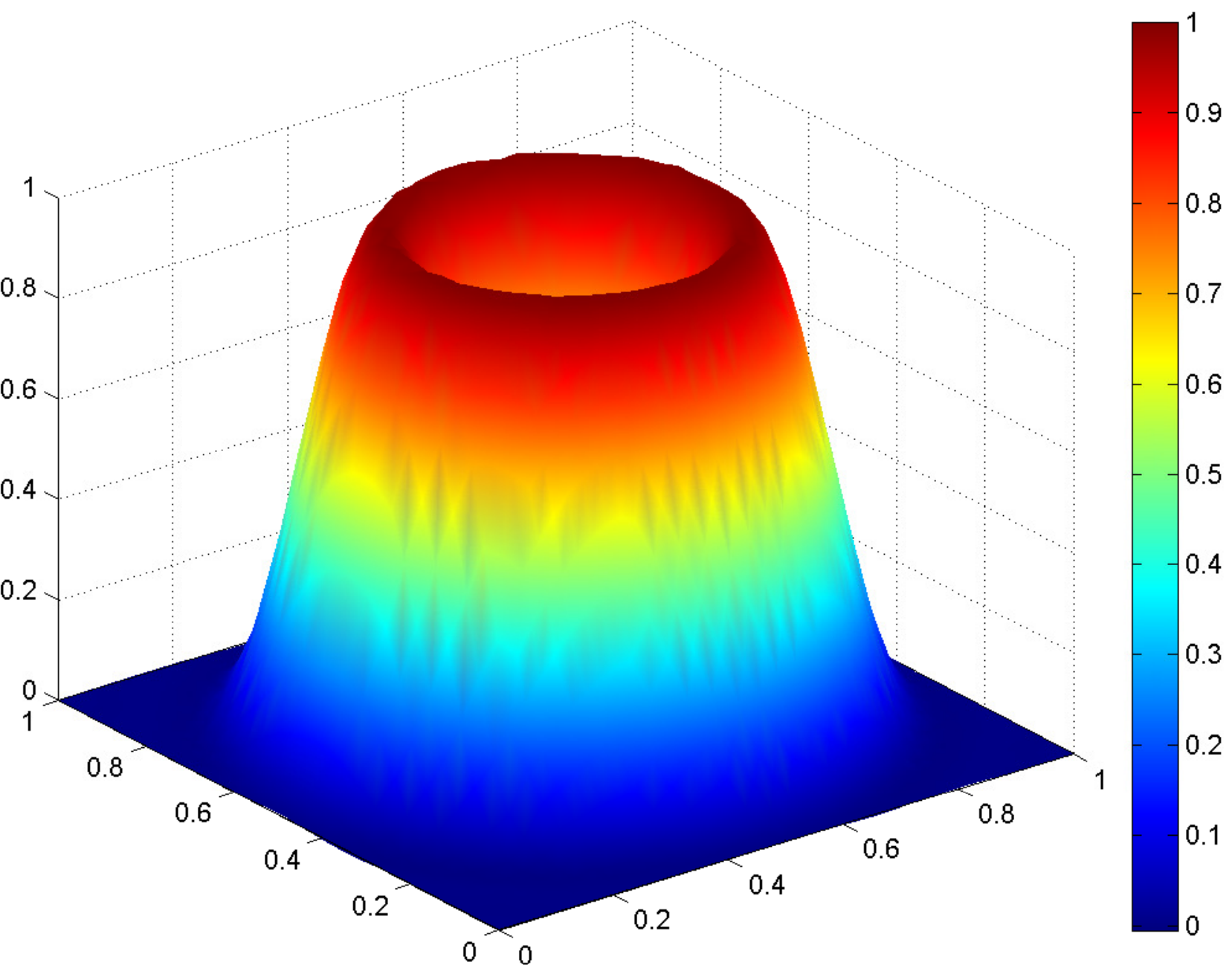}
\caption{3D plot of $u_h$ for rotating flow test with $\epsilon = 10^{-6}$ in 592 elements. Left--Right: LS--strong, LS--weak.
Top--Bottom: $P1$--$P3$.}
\label{rotating1} 
\end{figure}

\begin{figure}
\hfill{}\includegraphics[width=0.4\textwidth]{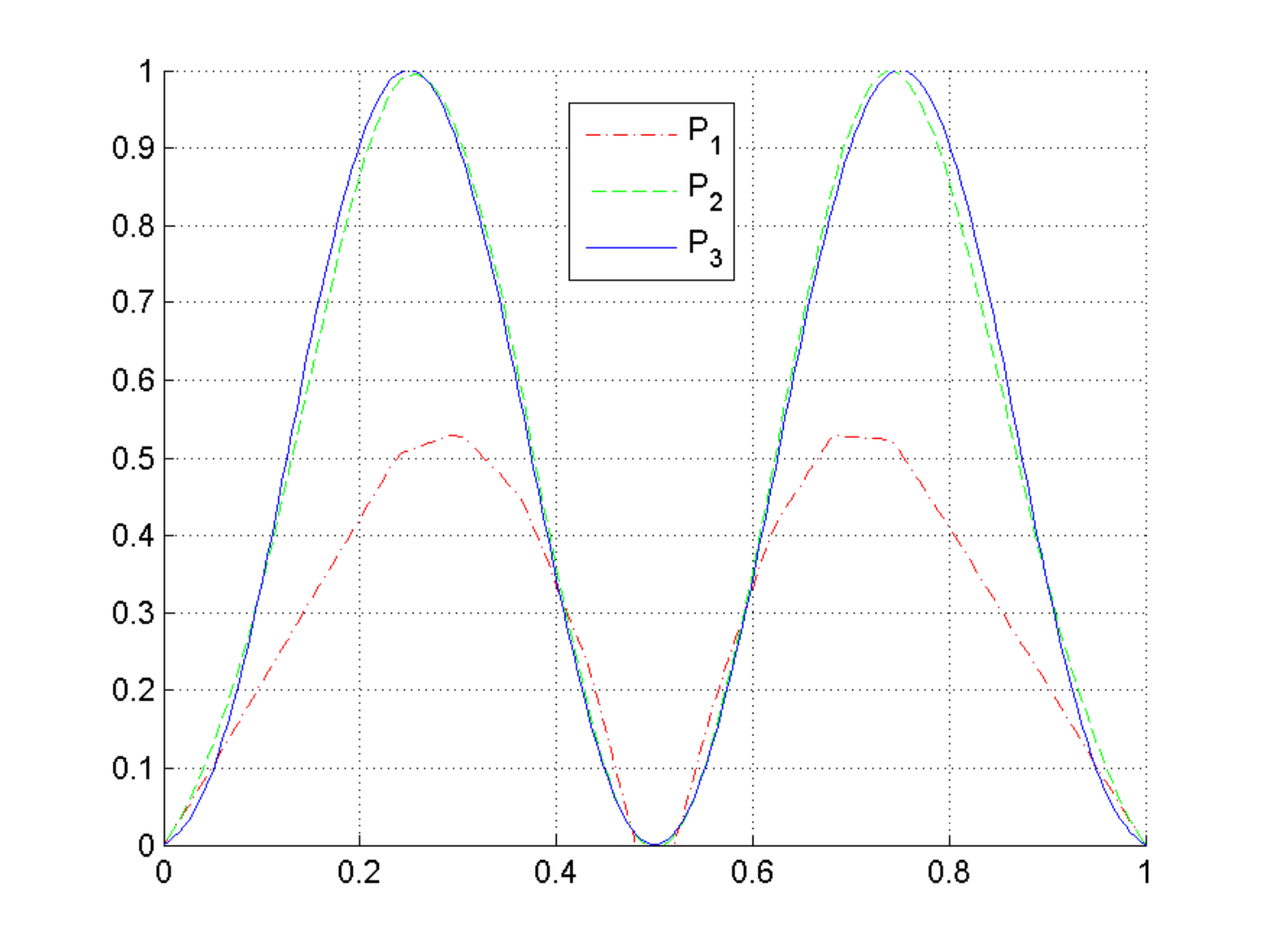}
\hfill{}\includegraphics[width=0.4\textwidth]{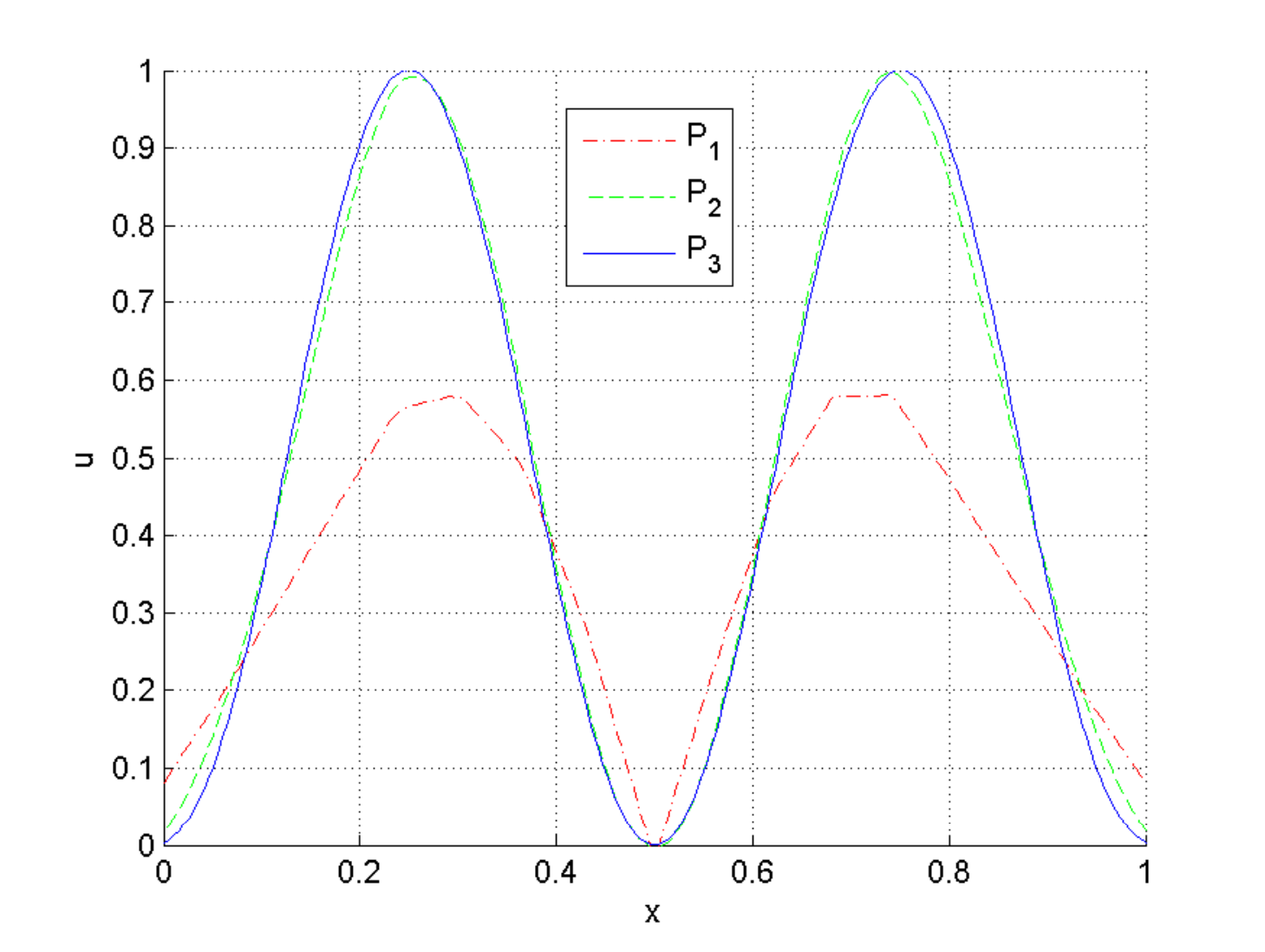}
\hfill{}
\caption{Profile of the approximate solution at $y=1/2$ when $\varepsilon=10^{-9}$. Left--Right: LS--strong, LS--weak.}
\label{rotating3}
\end{figure}

\subsection{An interior layer test}
We take $\boldsymbol{\beta} = [1/2, \sqrt{3}/2]^T$, $f = 0$, and Dirichlet boundary conditions as follows:
\begin{align*}
 u = \left\{\begin{tabular}{l l}
            $1$& { on }$\{y = 0, 0\le x\le 1\}$,\\
            $1$& { on }$\{x = 0, 0\le y\le 1/5\}$,\\
            $0$& { elsewhere. }
           \end{tabular}
\right.
\end{align*}
In Fig.~\ref{interior31} and Fig.~\ref{interior91},
we plot $u_{h}$ obtained from the two first order least squares methods for $\epsilon = 10^{-3}$ and $\epsilon = 10^{-9}$, respectively.
Though the exact solution of this example is not available, it is easy to see that the result produced by LS-weak is much more accurate.
Fig.~\ref{interior31} and Fig.~\ref{interior91} imply that $u_{h}$ produced by LS-strong in 11264 element is almost totally collapsed, 
while LS-weak can have much better approximation in a much coarser mesh with 704 elements.

In order to illustrate the behavior of LS-weak in capturing the interior layers, in Fig.~\ref{interior_contour},
we plot the contour plot of LS-weak using $P$2 for three consecutive meshes.  We observe finer mesh leads to sharper layer width.

\begin{figure}[ht!]
\centering
\includegraphics[width = 0.4\textwidth]{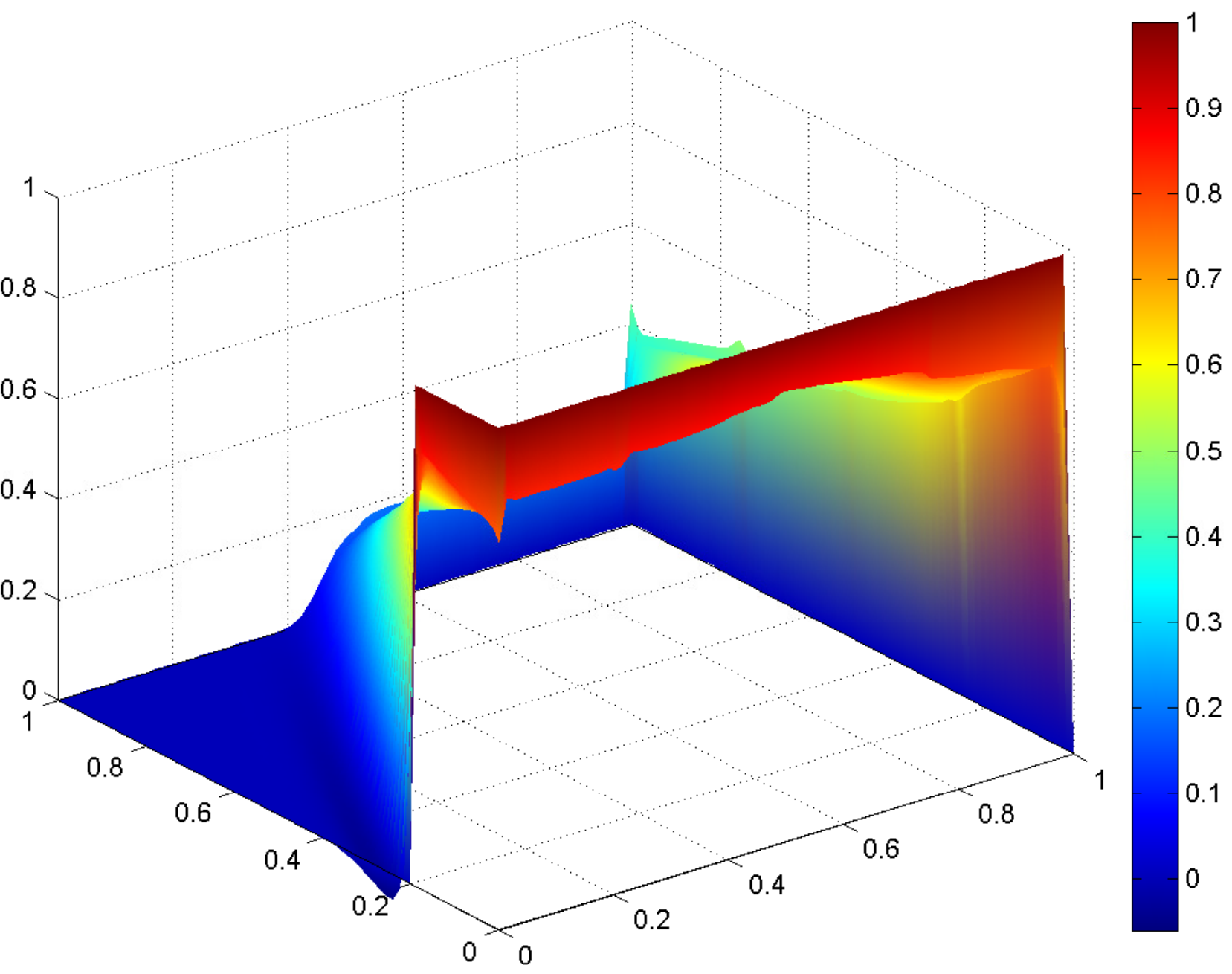}
\includegraphics[width = 0.4\textwidth]{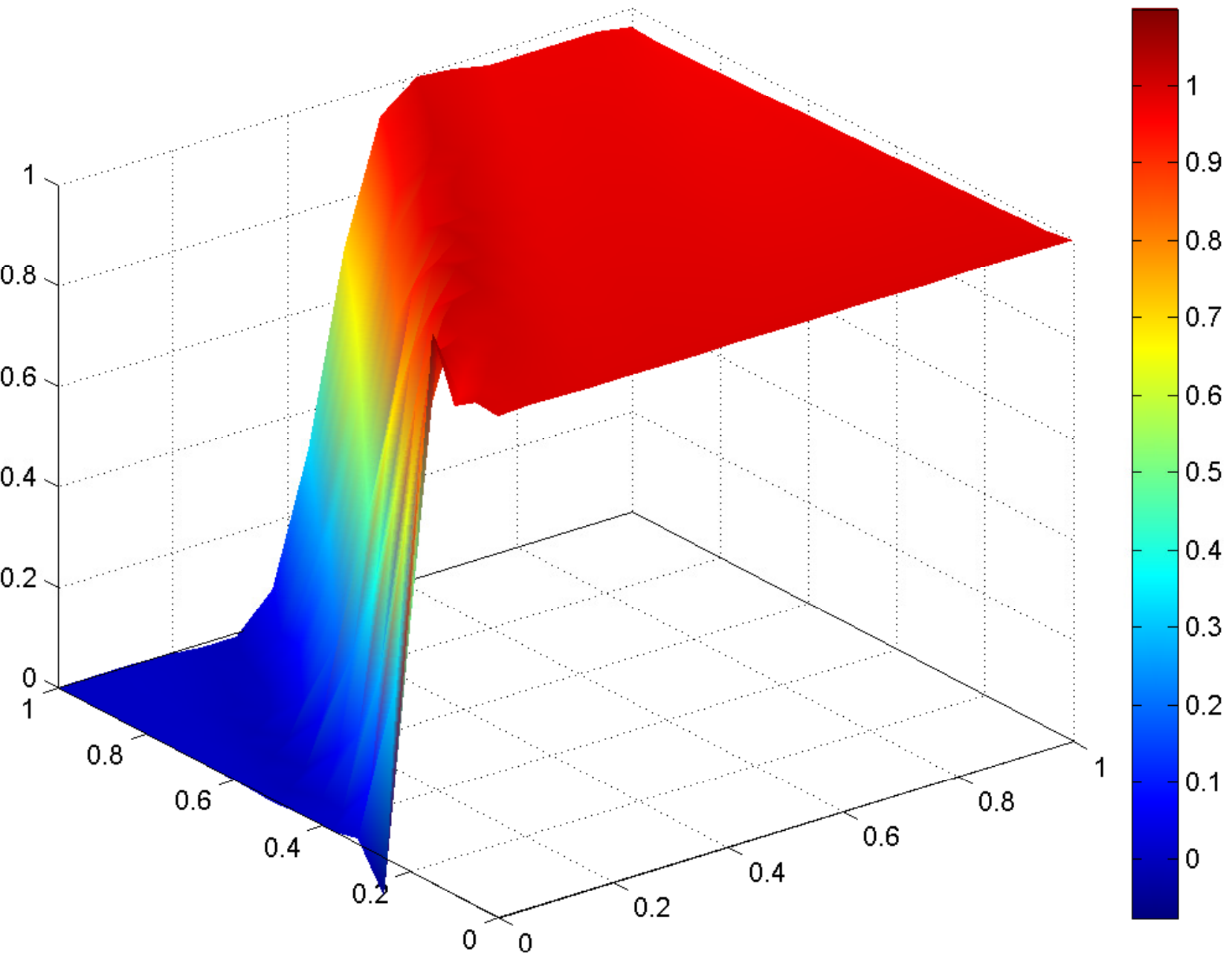}
\includegraphics[width = 0.4\textwidth]{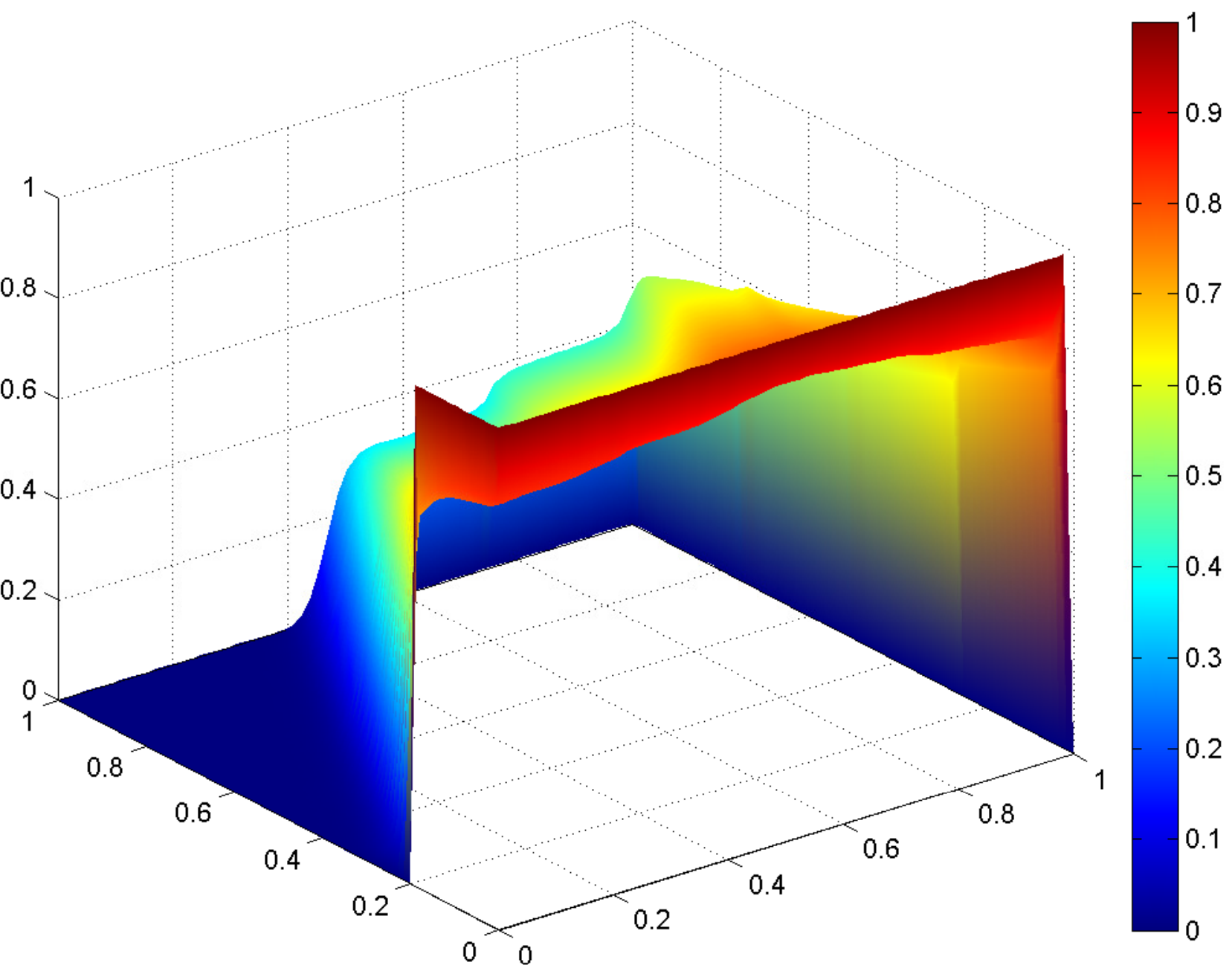}
\includegraphics[width = 0.4\textwidth]{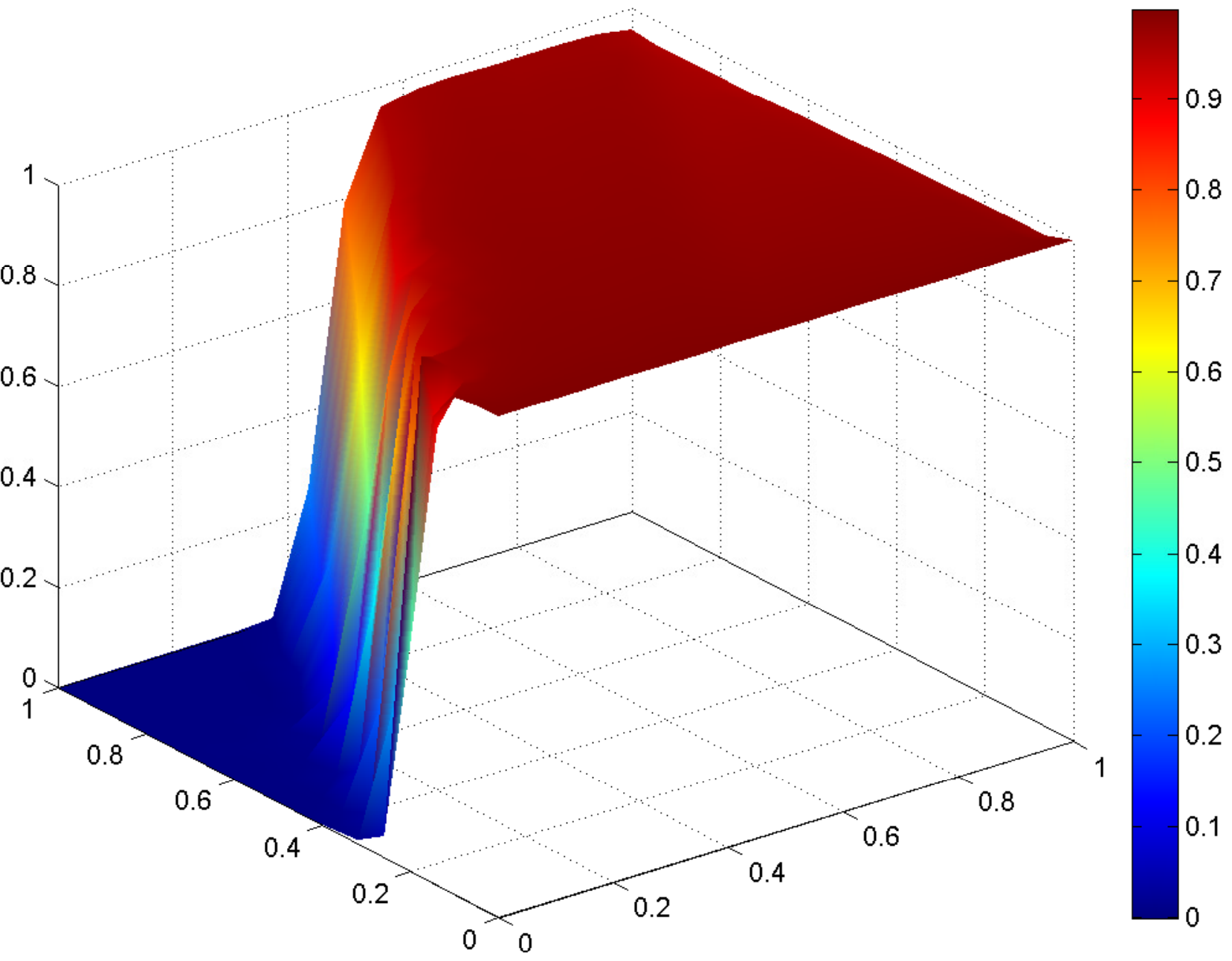}
\includegraphics[width = 0.4\textwidth]{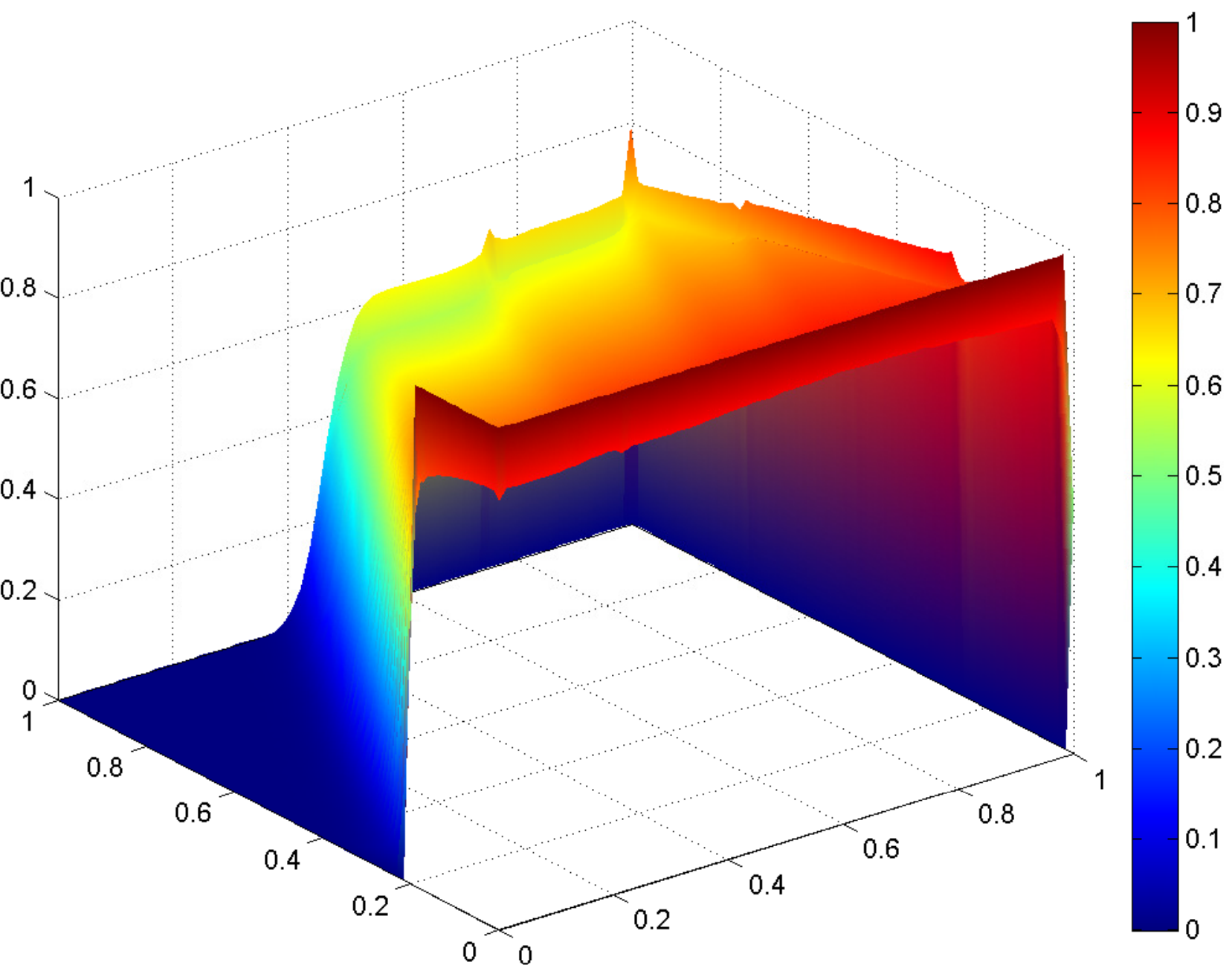}
\includegraphics[width = 0.4\textwidth]{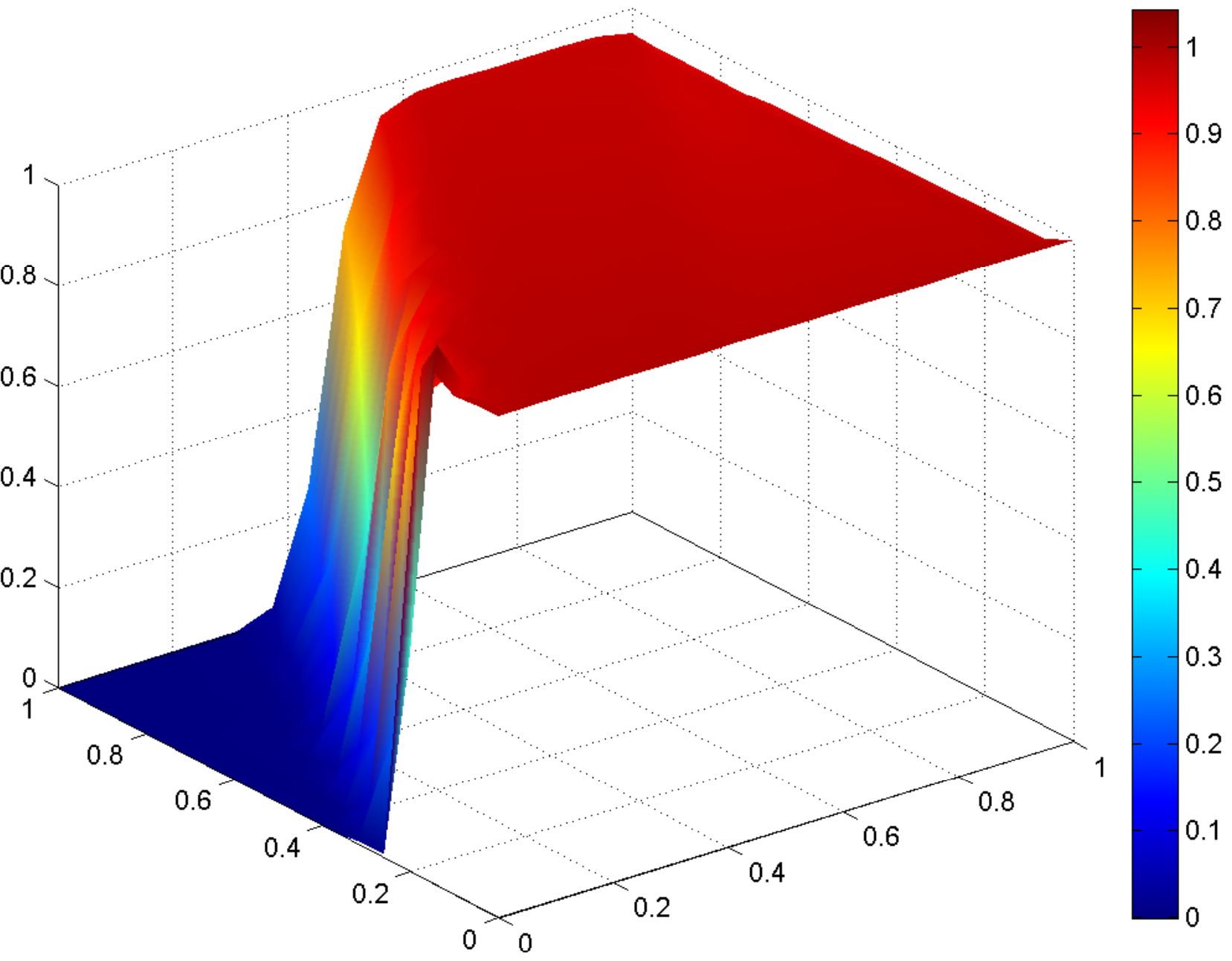}
\caption{3D plot of $u_h$ for interior layer test with $\epsilon = 10^{-3}$. 
Left--Right: LS--strong in 11264 elements, LS-weak in 704 elements. Top--Bottom: $P1$--$P3$.}
\label{interior31} 
\end{figure}

\begin{figure}[ht!]
\centering
\includegraphics[width = 0.4\textwidth]{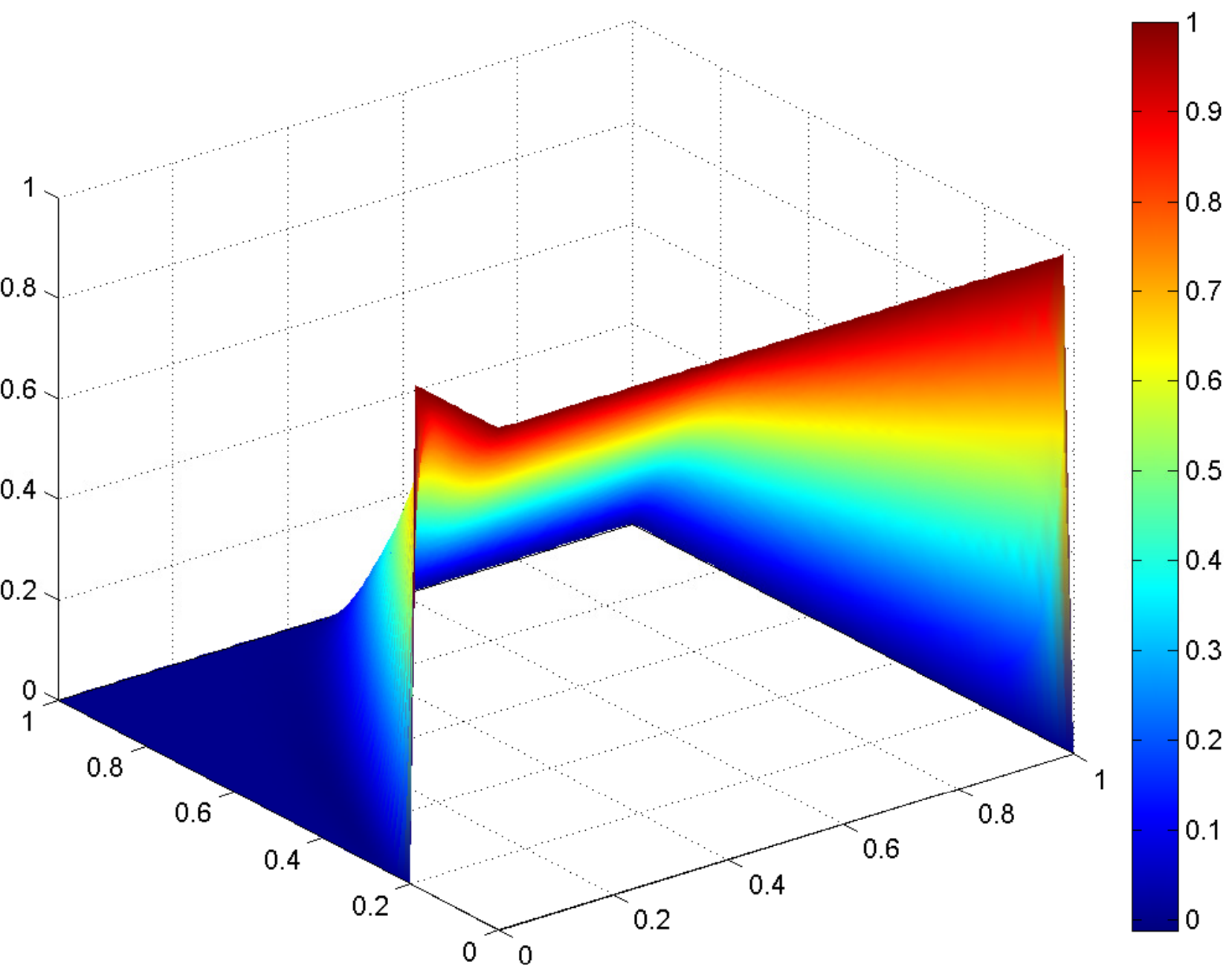}
\includegraphics[width = 0.4\textwidth]{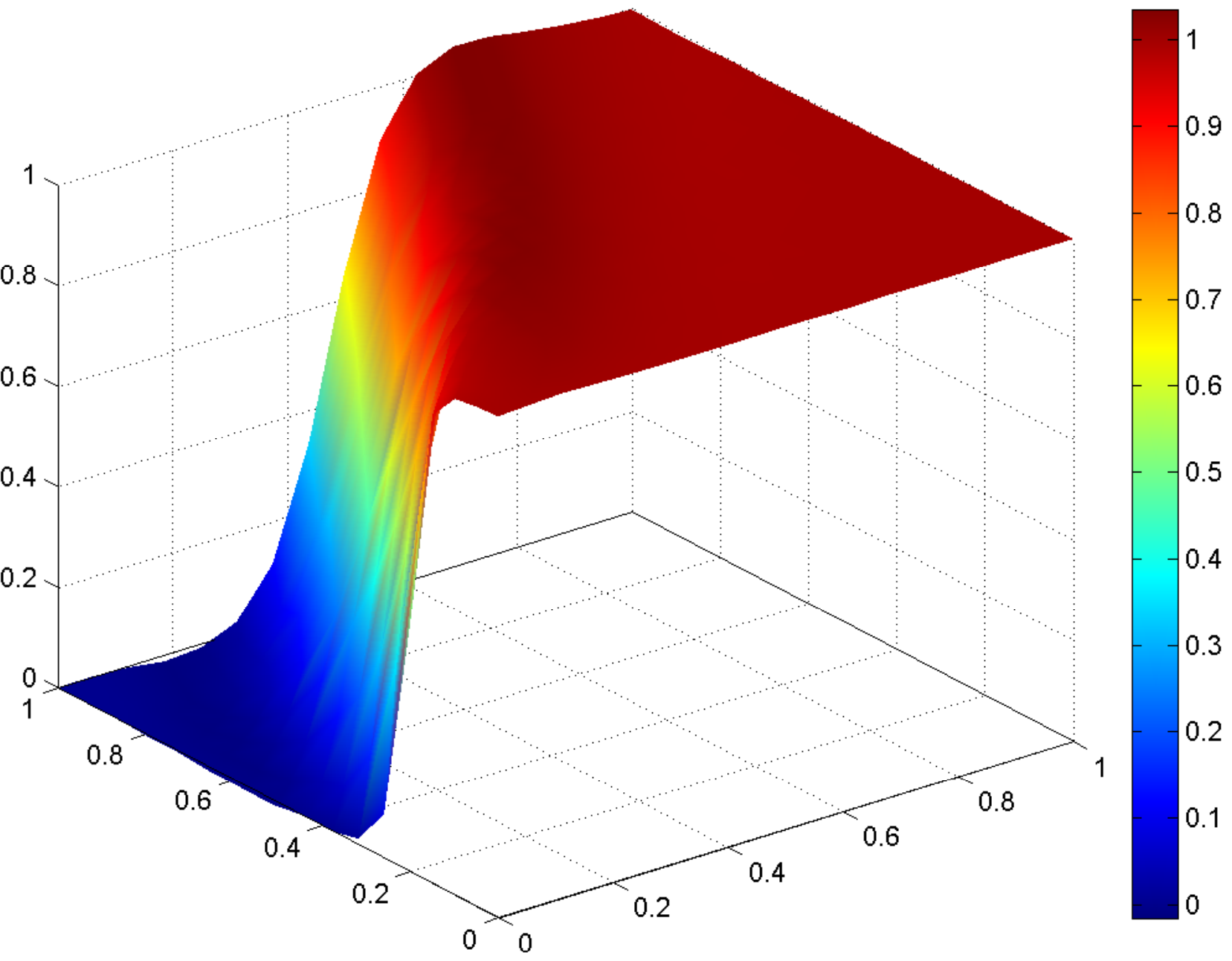}
\includegraphics[width = 0.4\textwidth]{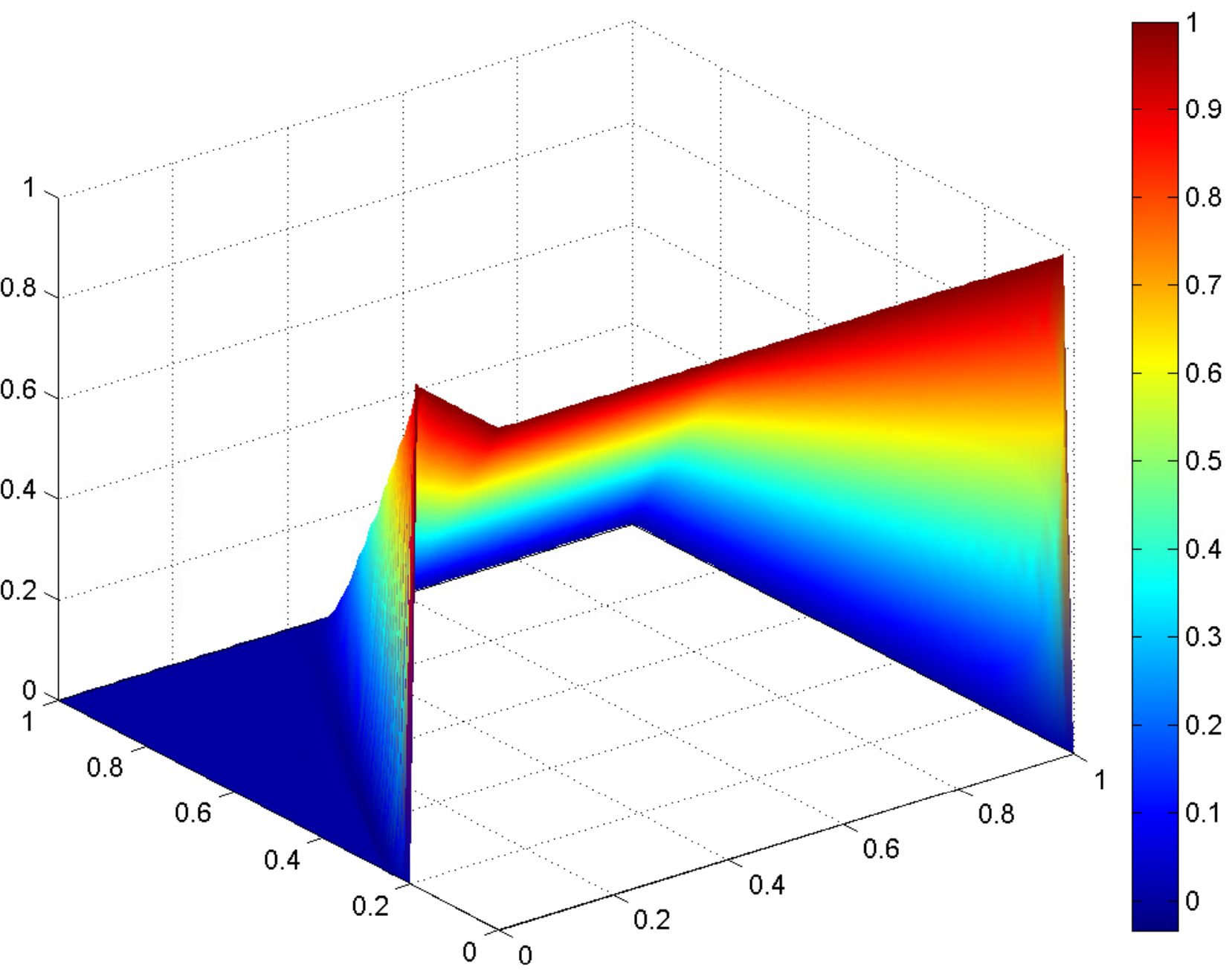}
\includegraphics[width = 0.4\textwidth]{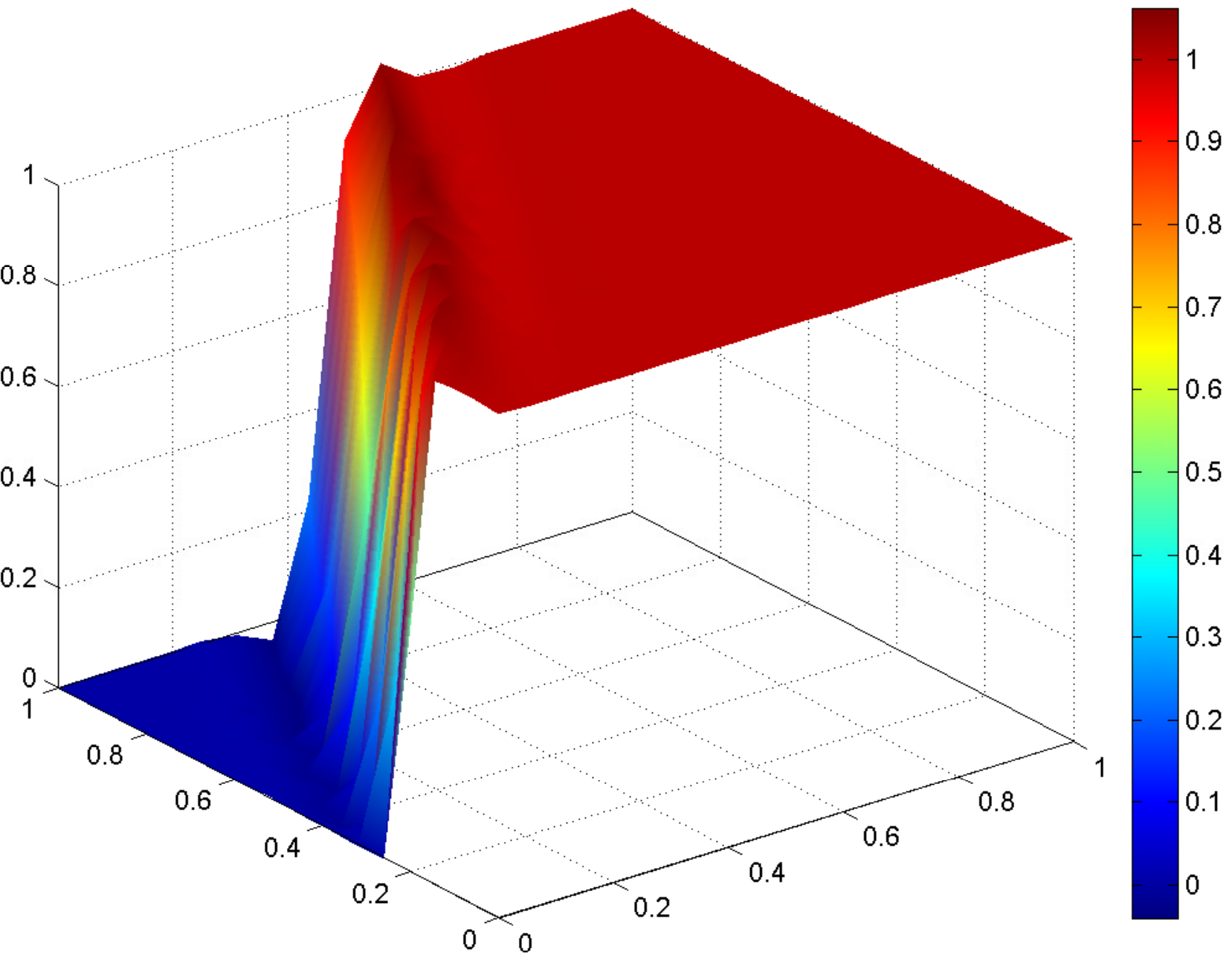}
\includegraphics[width = 0.4\textwidth]{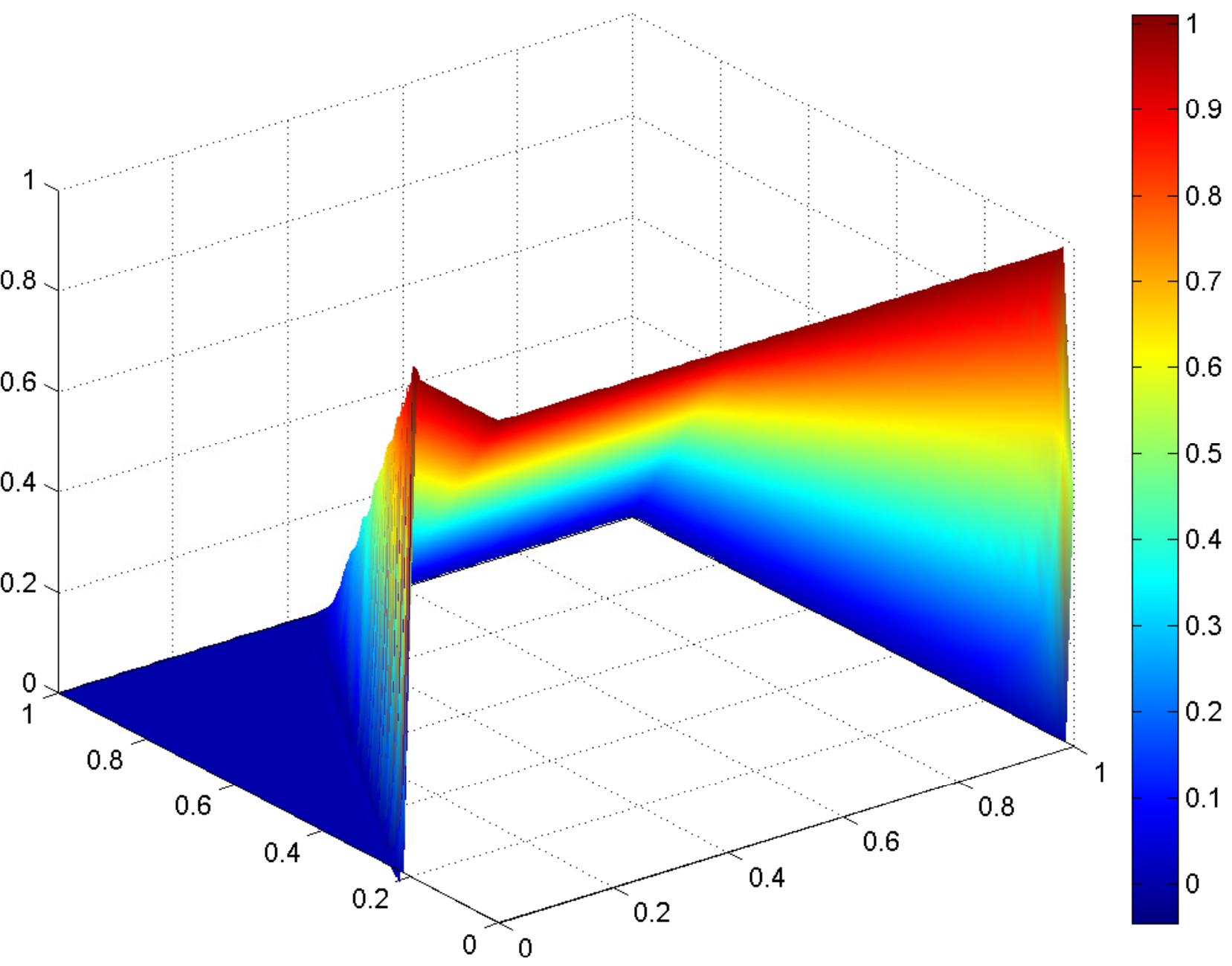}
\includegraphics[width = 0.4\textwidth]{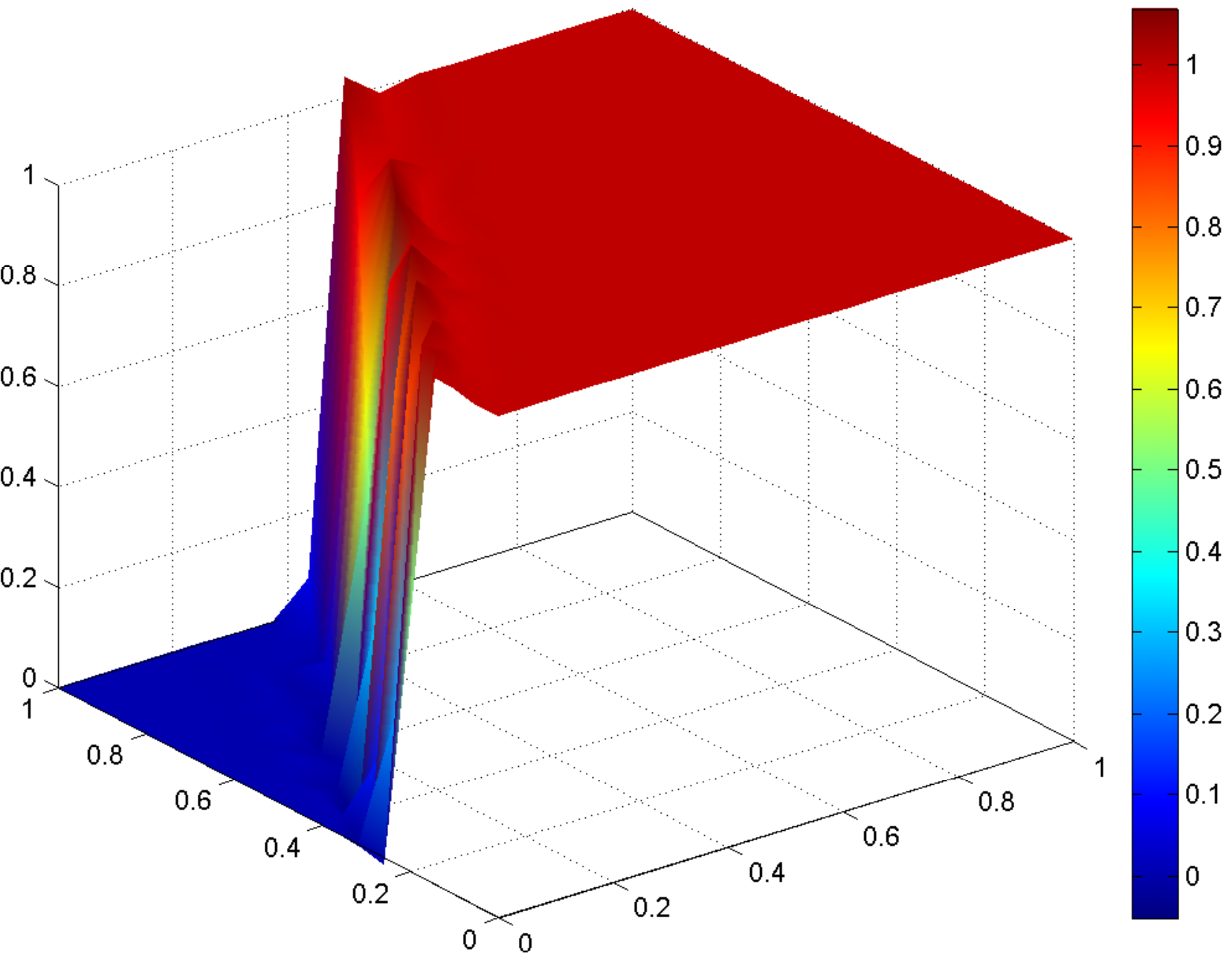}
\caption{3D plot of $u_h$ for interior layer test with $\epsilon = 10^{-9}$. 
Left--Right: LS--strong in 11264 elements, LS-weak in 704 elements. Top--Bottom: $P1$--$P3$.}
\label{interior91} 
\end{figure}

\begin{figure}[ht!]
\centering
\includegraphics[width = 0.4\textwidth]{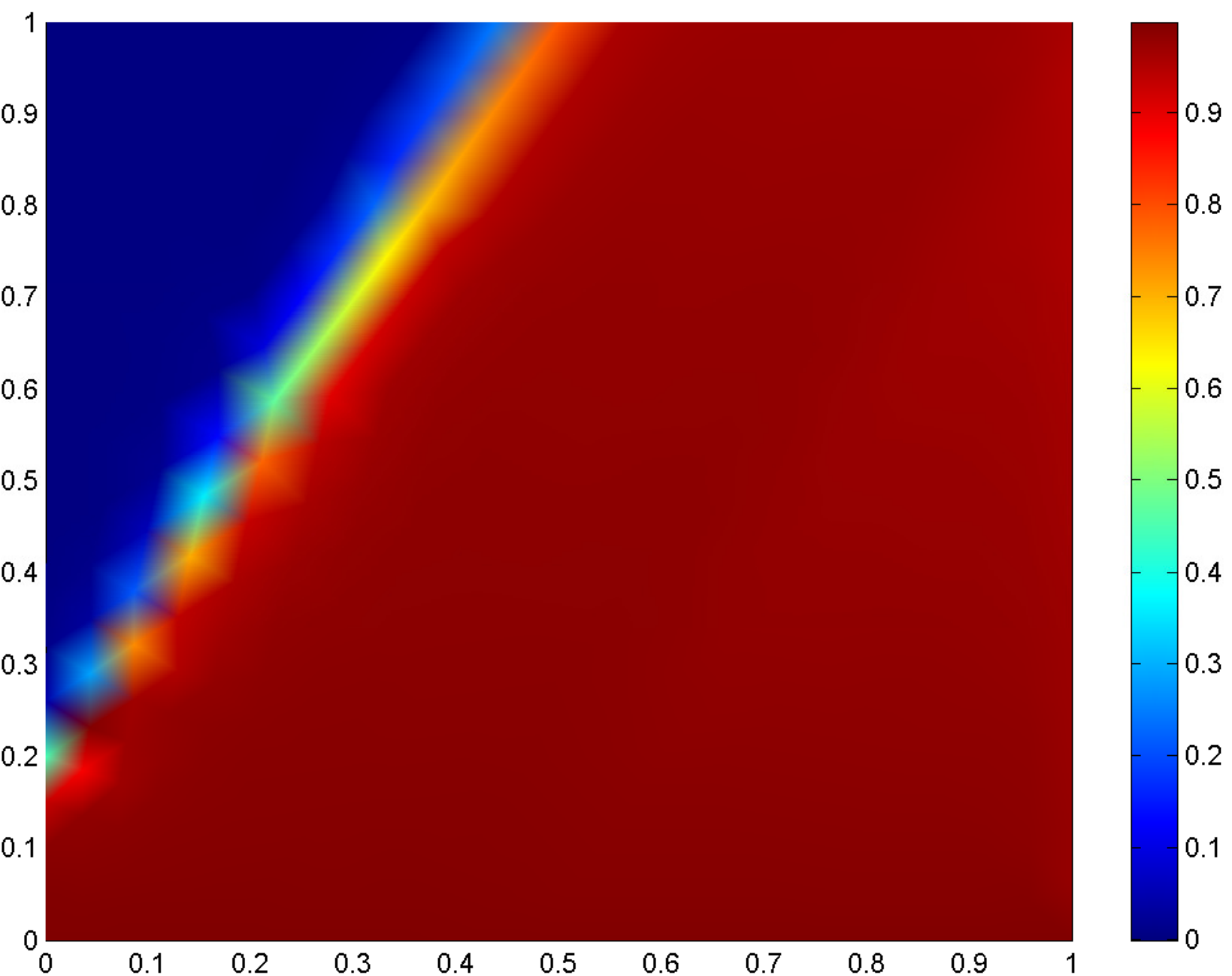}
\includegraphics[width = 0.4\textwidth]{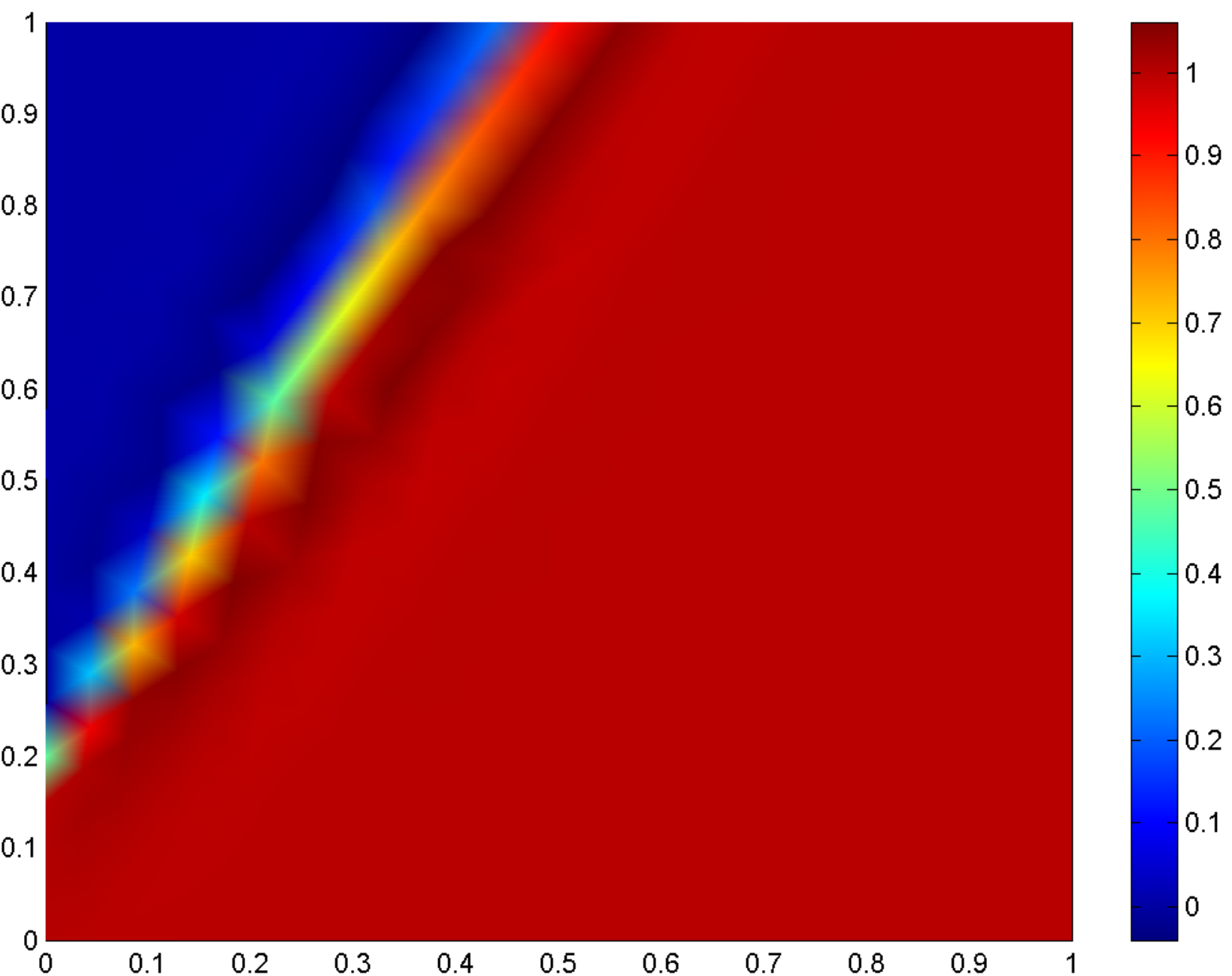}
\includegraphics[width = 0.4\textwidth]{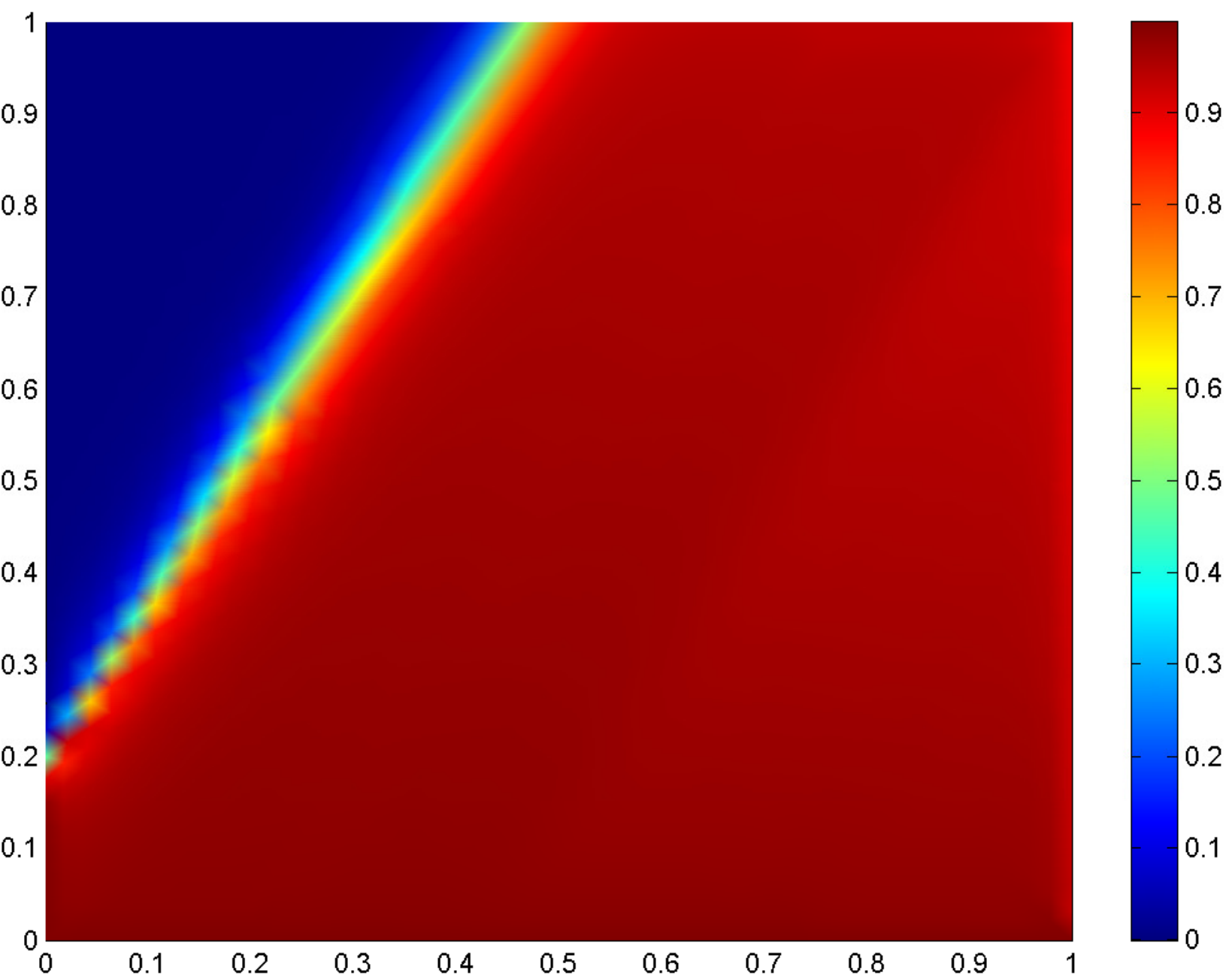}
\includegraphics[width = 0.4\textwidth]{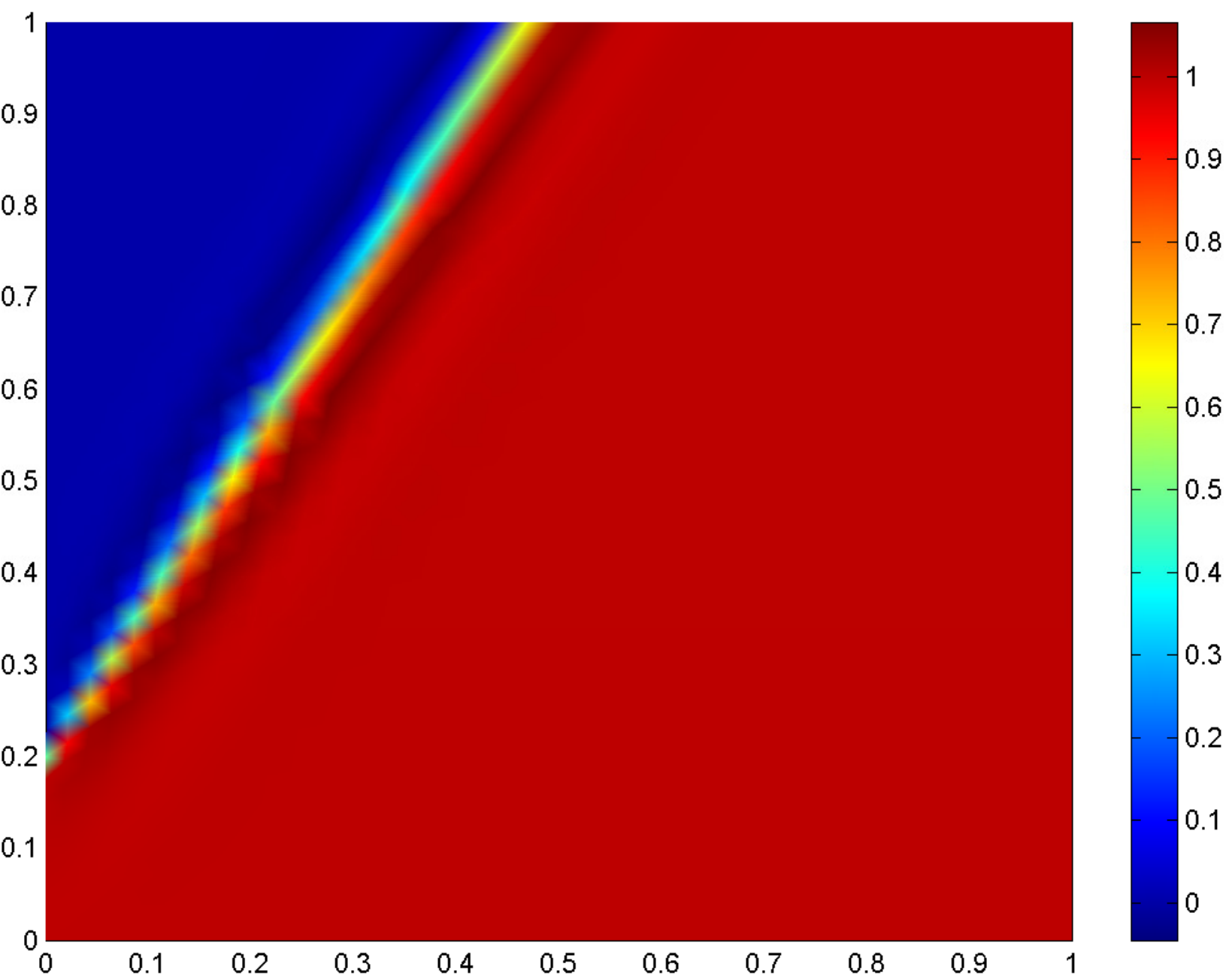}
\includegraphics[width = 0.4\textwidth]{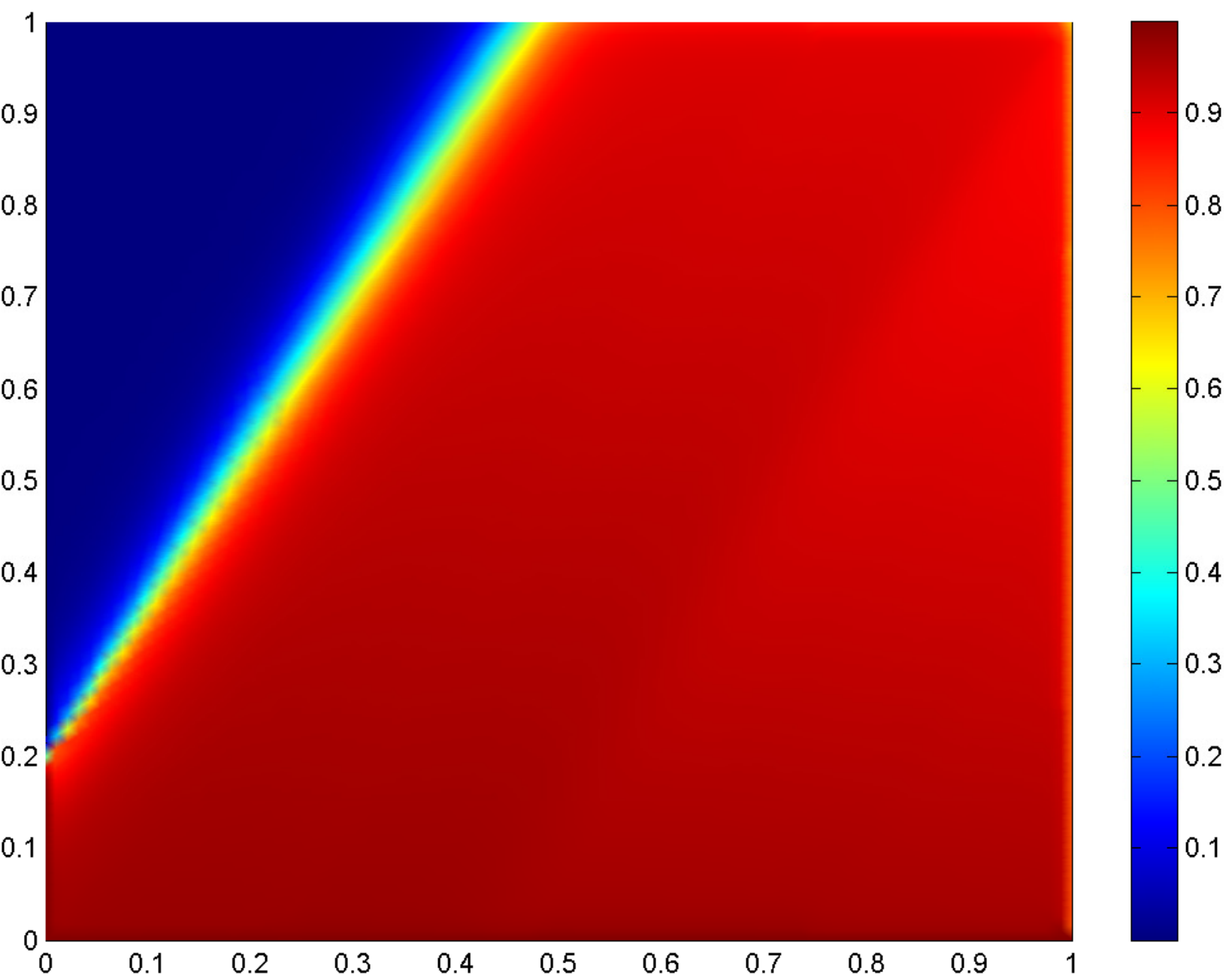}
\includegraphics[width = 0.4\textwidth]{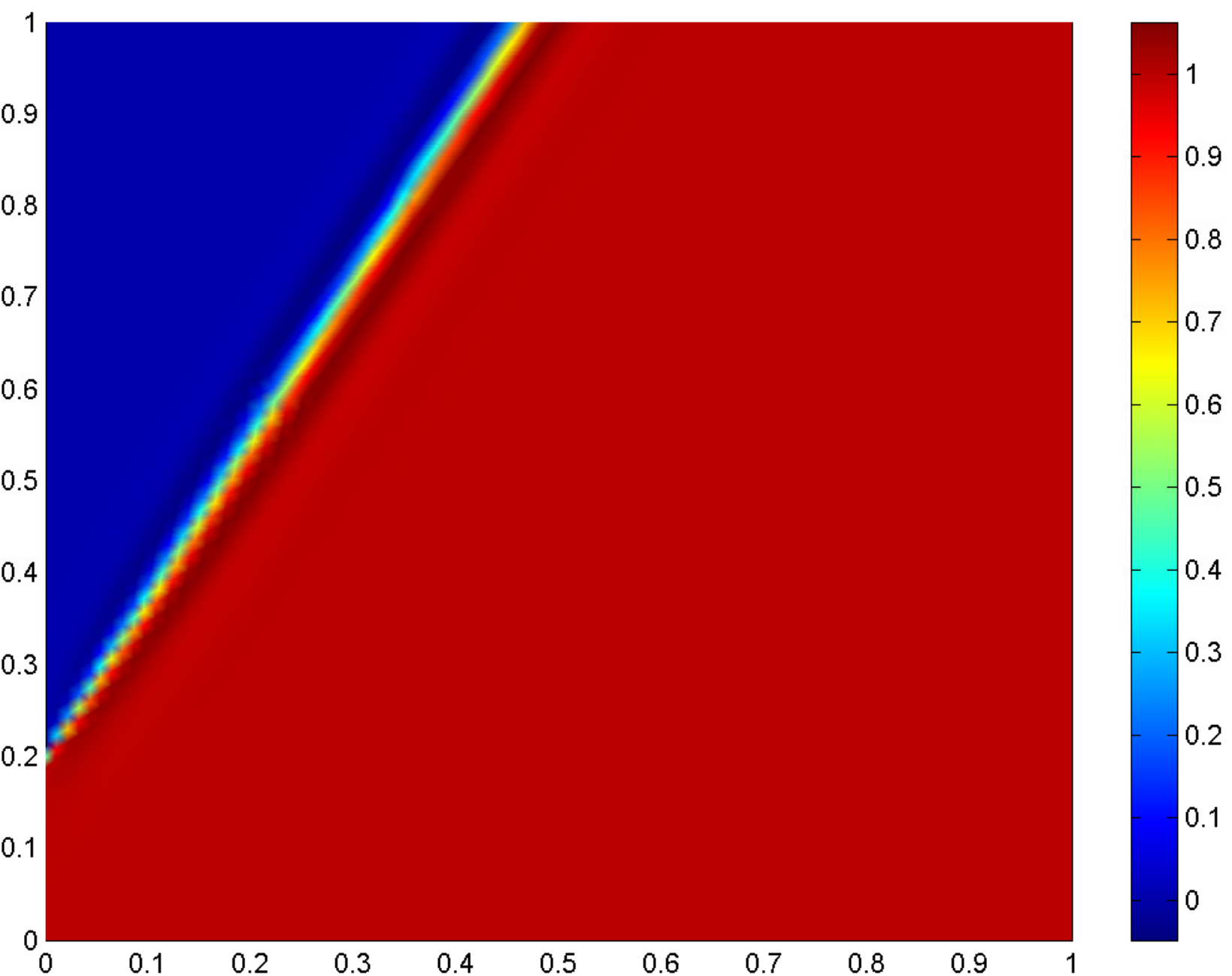}
\caption{Contour plot of $u_h$ for interior layer test using LS--weak $P$2. Left: $\epsilon = 10^{-3}$; 
Right: $\epsilon = 10^{-9}$. Top: 704 elements; Middle: 2816 element; Bottom: 11264 elements.}
\label{interior_contour} 
\end{figure}

\subsection{A boundary layer test}
Finally, we take $\boldsymbol{\beta} = [1, 1]^T$, and choose the source term $f$ such that the exact solution
\[
u(x,y) = \sin\frac{\pi\,x}{2} + \sin\frac{\pi\,y}{2}\left(1-\sin\frac{\pi\,x}{2}\right) 
+\frac{e^{-1/\epsilon} - e^{-(1-x)(1-y)/\epsilon}}{1 - e^{-1/\epsilon}}.
\]
The solution develops boundary layers along the boundaries $x = 1$ and $y = 1$ for small $\epsilon$ (see Fig.~\ref{bdry_plot1}). 
Let us emphasize that our exact solution is different from that of \cite{AyusoMarini:cdf} in order to 
have a better observation of the local convergence results for higher degree approximations.
In Fig.~\ref{bdry_plot1}, we plot the exact solution and computational
results for $\epsilon =  10^{-2},10^{-6},10^{-9}$.
We notice that the numerical solutions produced by LS-strong are totally polluted by boundary layers, 
while LS-weak can produce numerical solutions very close to the exact ones except in area 
very close to boundary layers.

In Fig.~\ref{bdry_test1}, we show the convergence of $u_h$ produced by LS--weak in $L^2$--norm for $\epsilon = 10^{-2},10^{-9}$ 
in the subdomain $\widetilde{\Omega} = [0, 0.9]\times[0,0.9]\subset \Omega$ to exclude the unresolved boundary layer. 
We notice that the accuracy of $u_{h}$ in $\tilde{\Omega}$ is not affected by boundary layers.  

\begin{figure}[ht!]
\centering
\includegraphics[width = 0.3\textwidth]{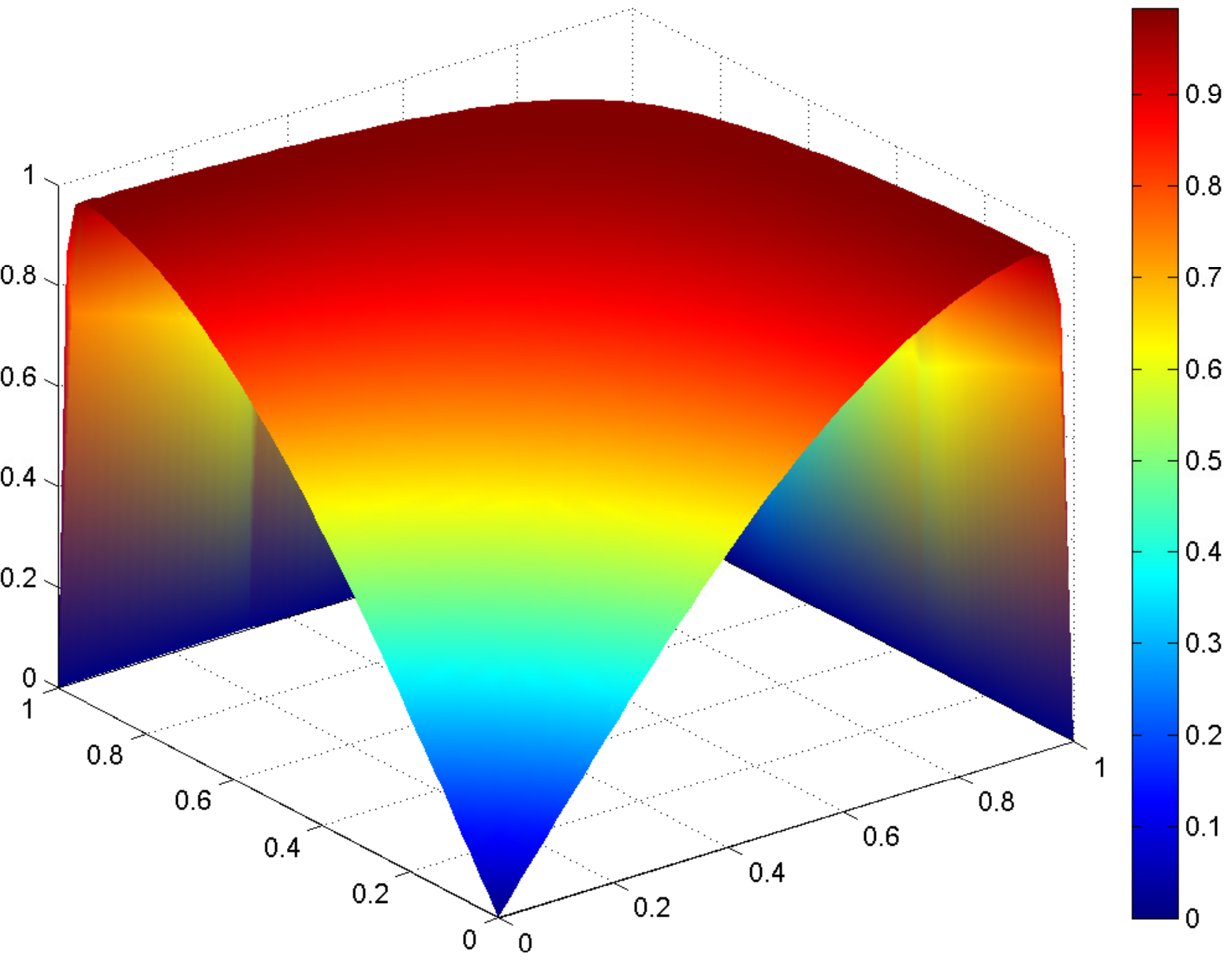}
\includegraphics[width = 0.3\textwidth]{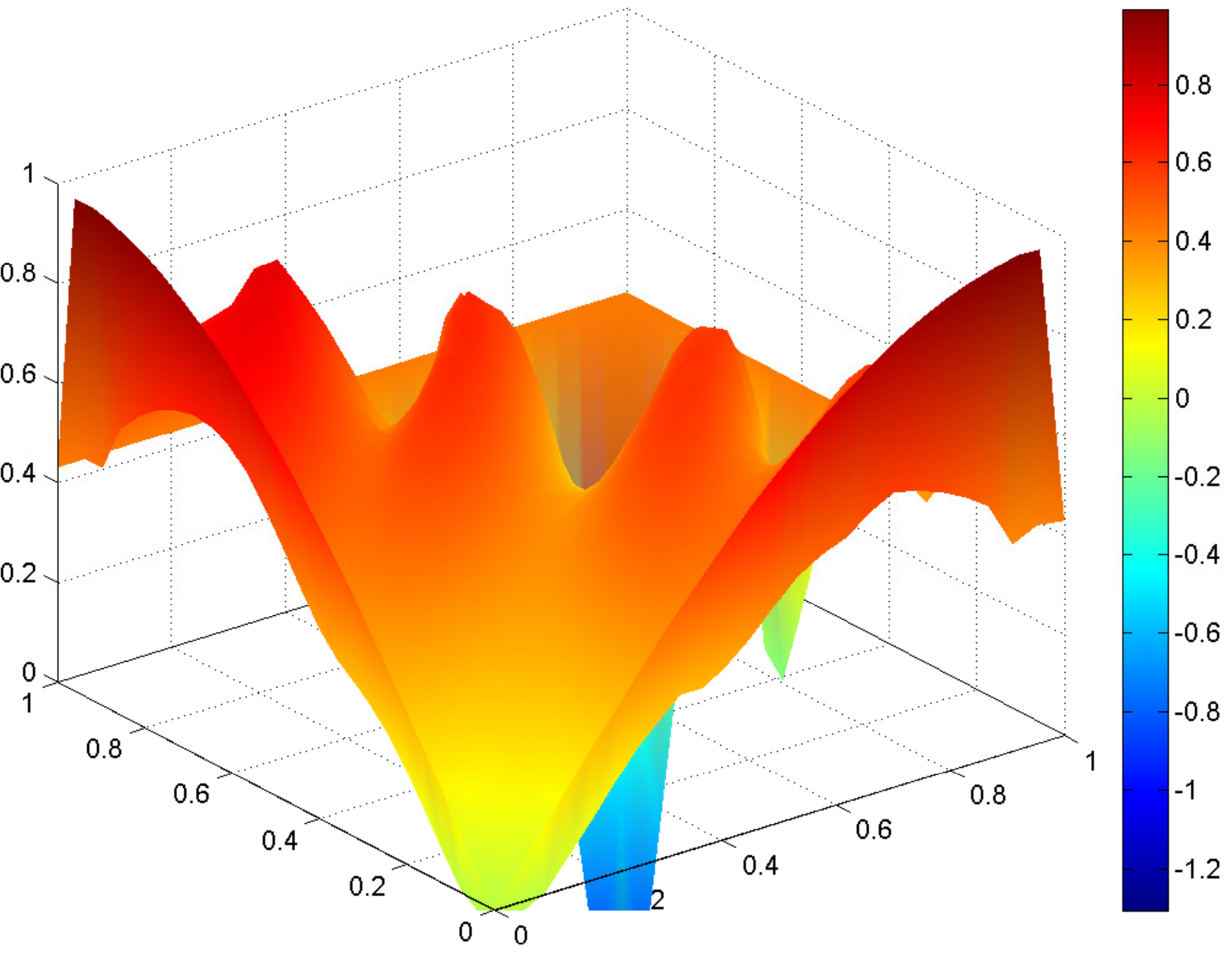}
\includegraphics[width = 0.3\textwidth]{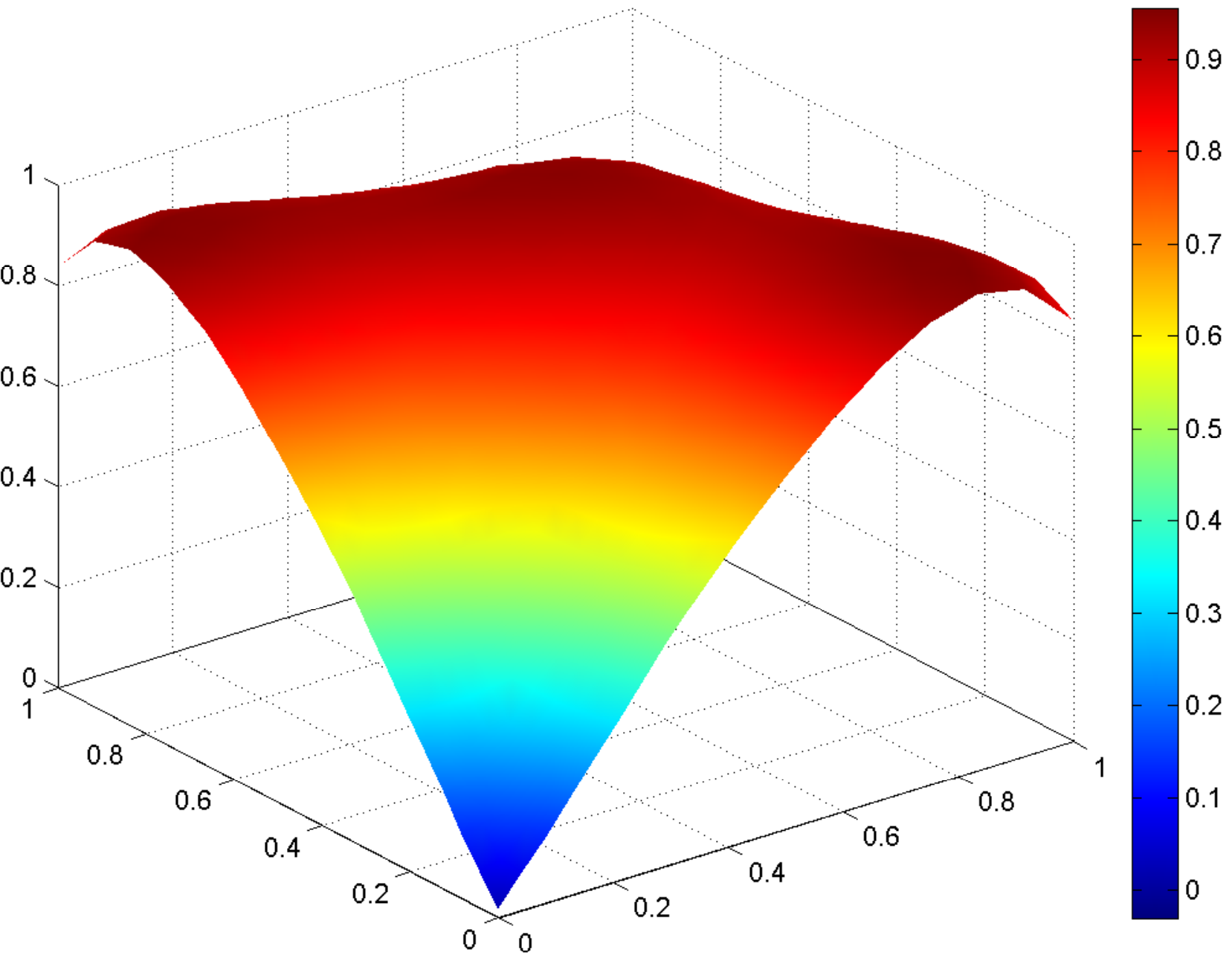}
\includegraphics[width = 0.3\textwidth]{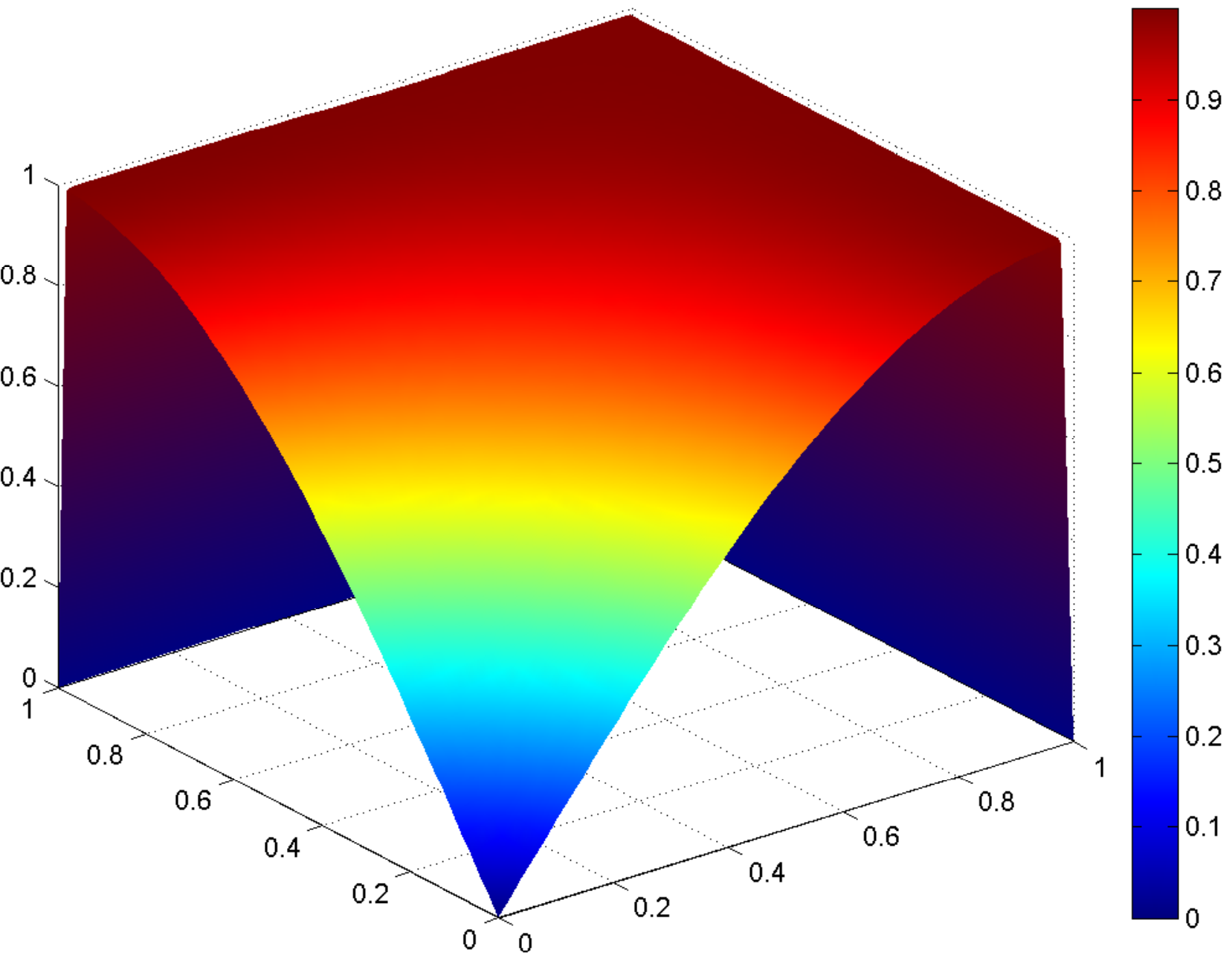}
\includegraphics[width = 0.3\textwidth]{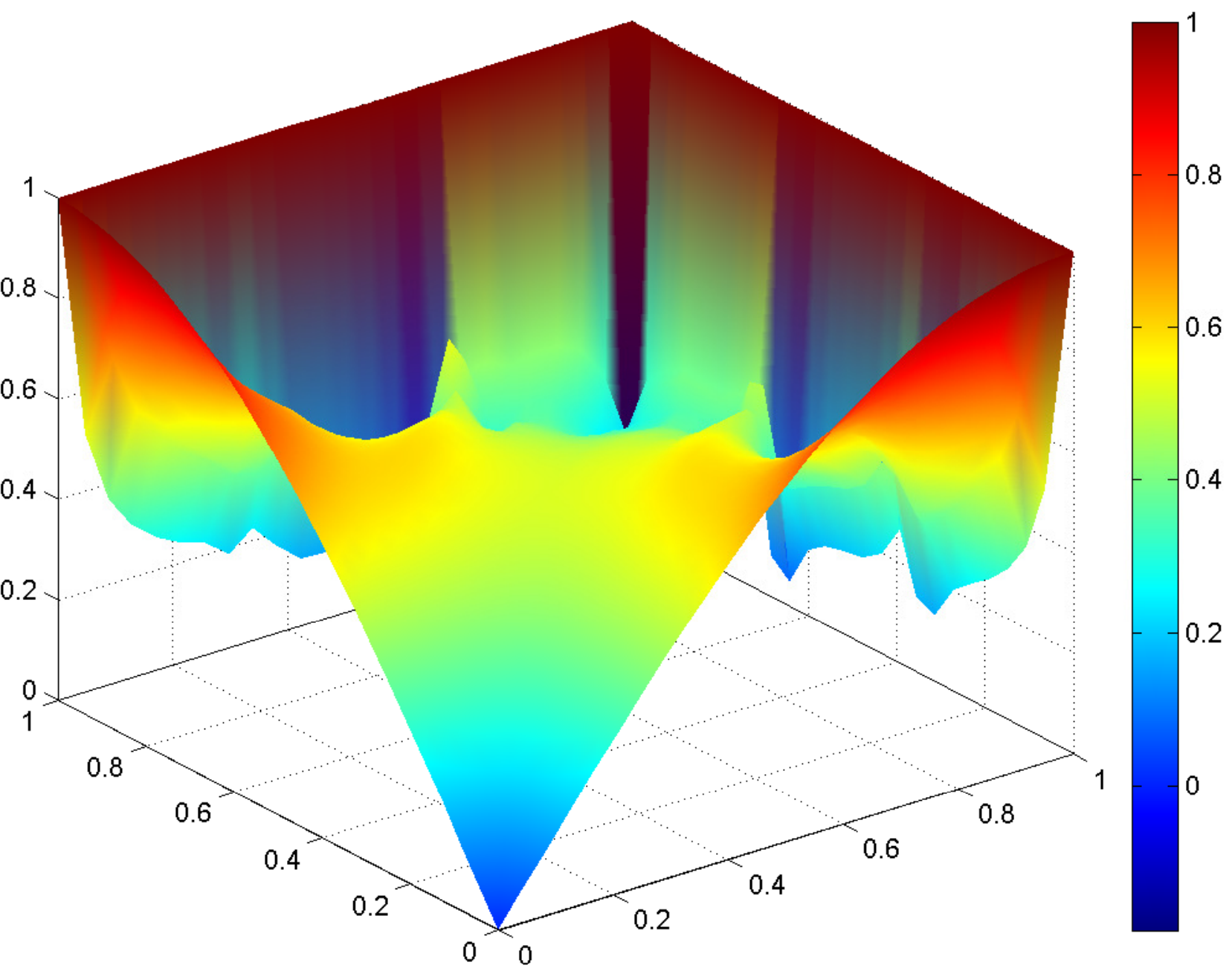}
\includegraphics[width = 0.3\textwidth]{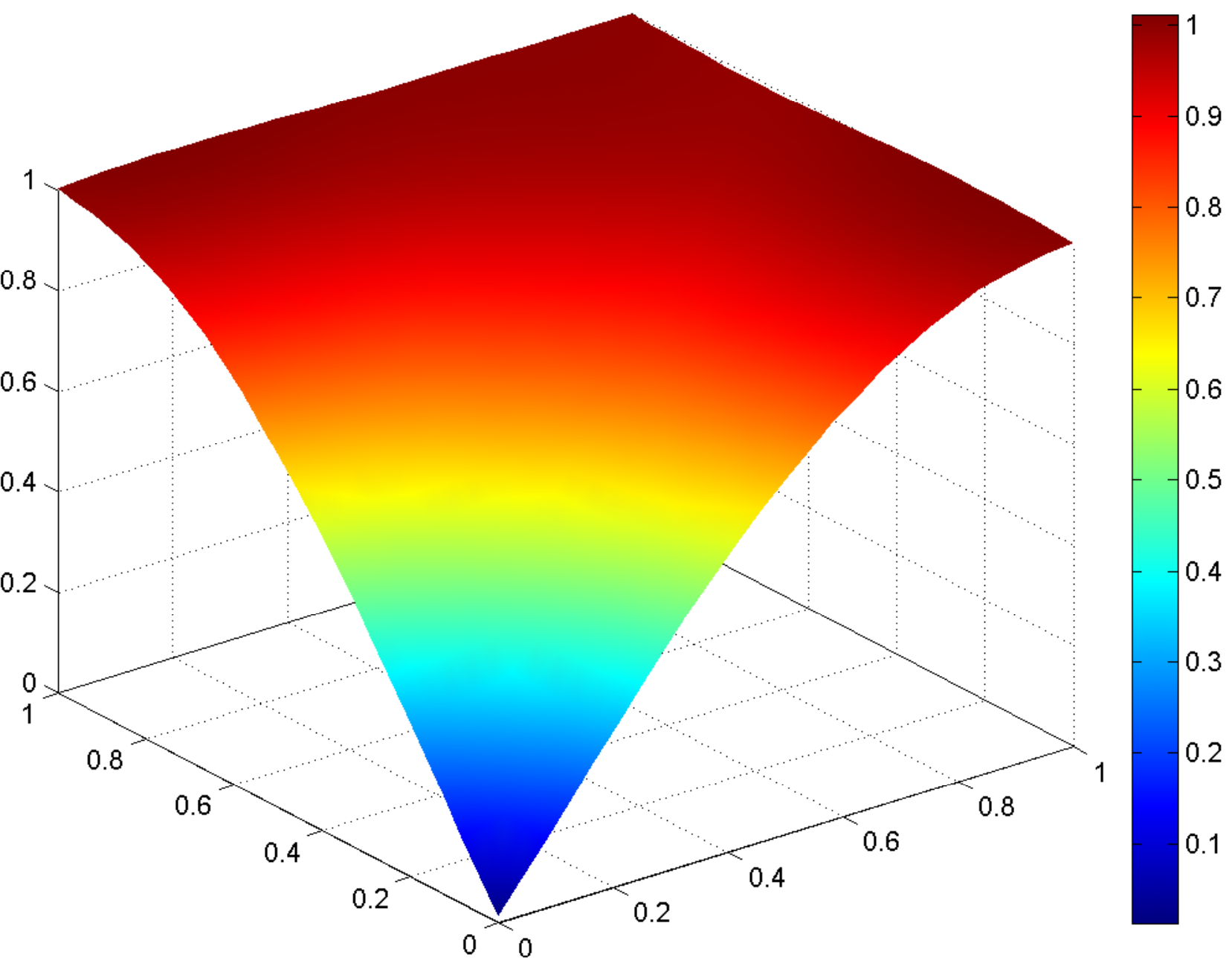}
\includegraphics[width = 0.3\textwidth]{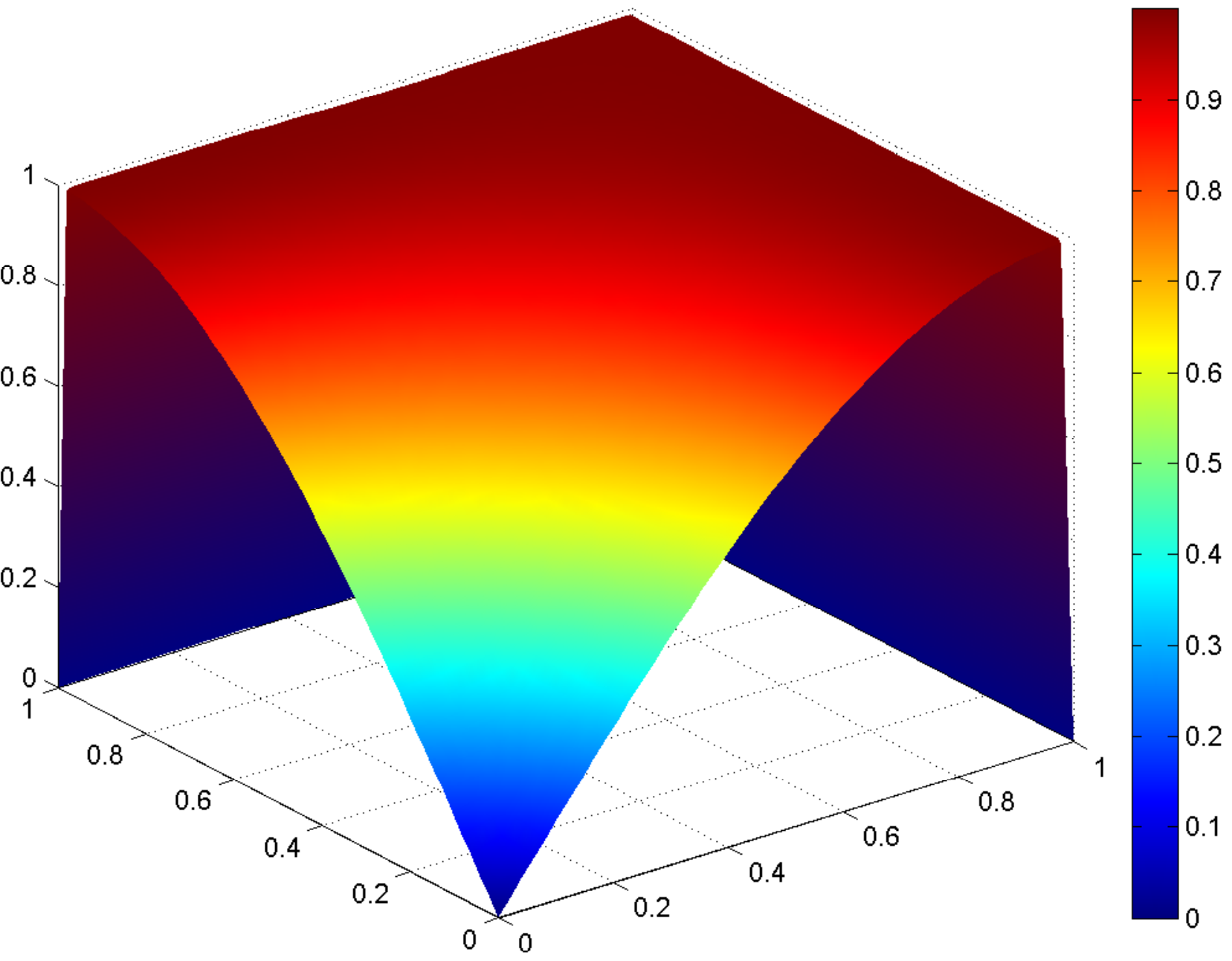}
\includegraphics[width = 0.3\textwidth]{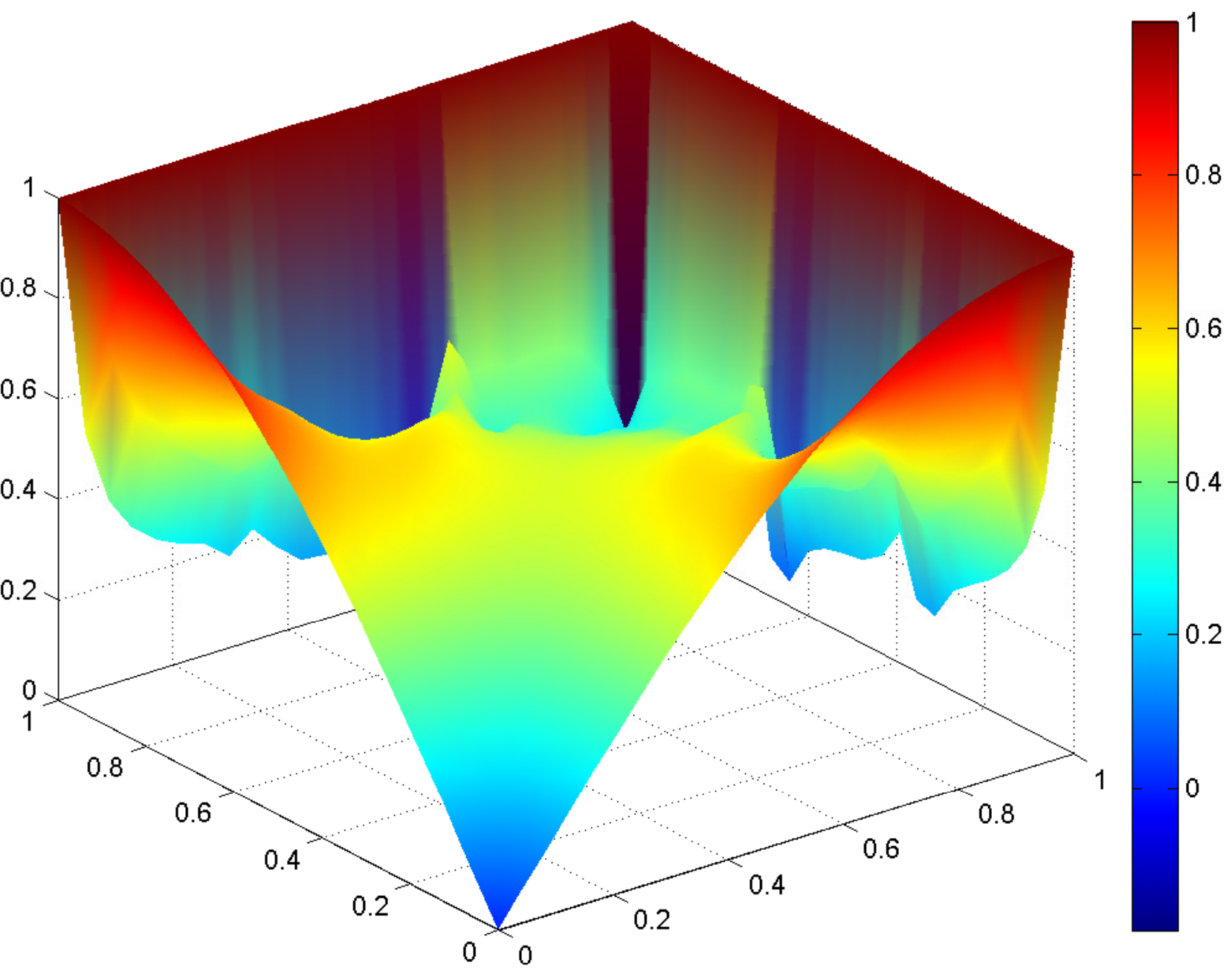}
\includegraphics[width = 0.3\textwidth]{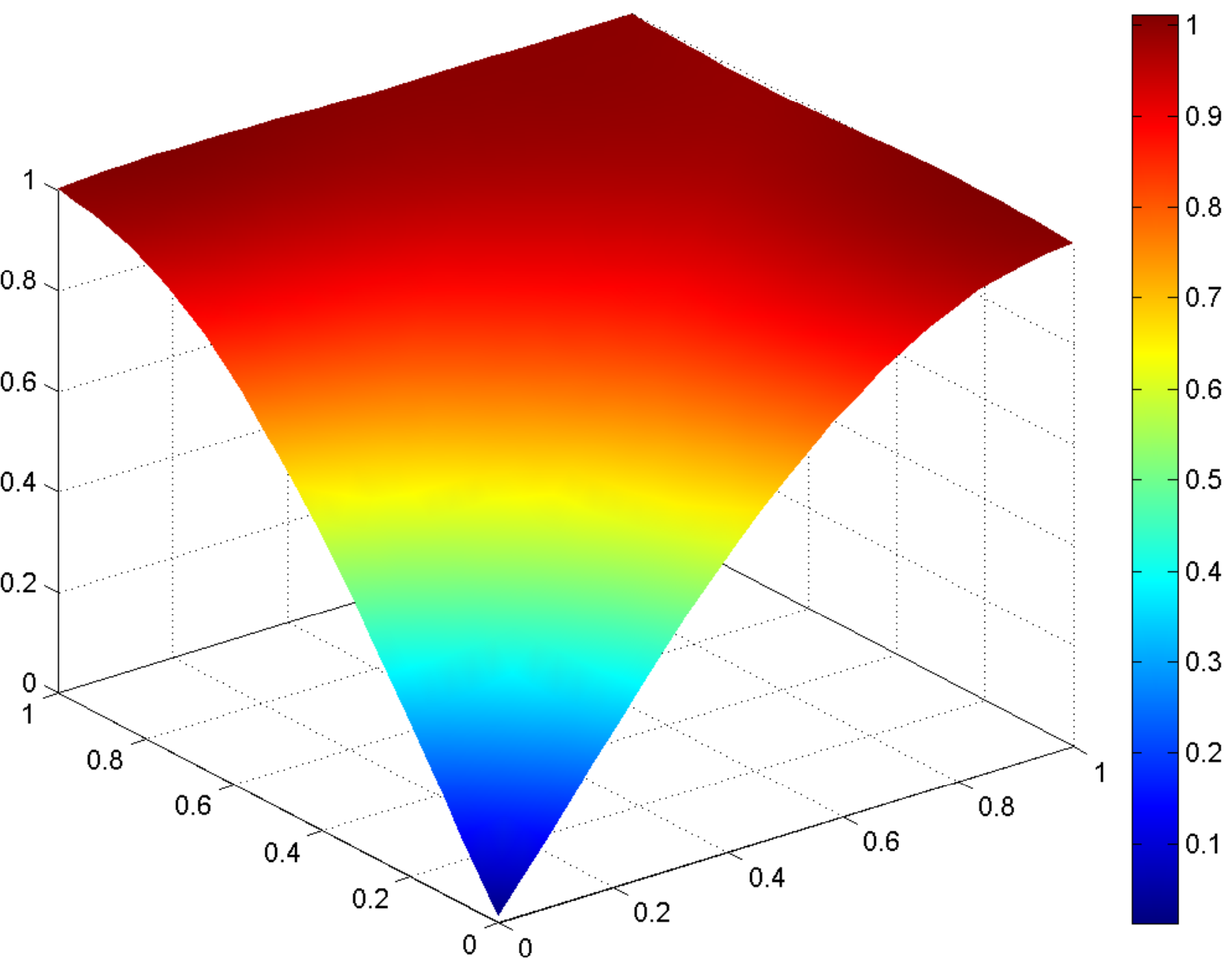}
\caption{3D plot of $u$ for the boundary layer test. Top: $\epsilon = 10^{-2}$; Middle: $\epsilon=10^{-6}$; Bottom: $\epsilon = 10^{-9}$. 
Left: exact solution; Middle: LS--strong in 1664 elements using P2; Right: LS--weak in 416 elements using P1.}
\label{bdry_plot1} 
\end{figure}

\begin{figure}
\hfill{}\includegraphics[width=0.4\textwidth]{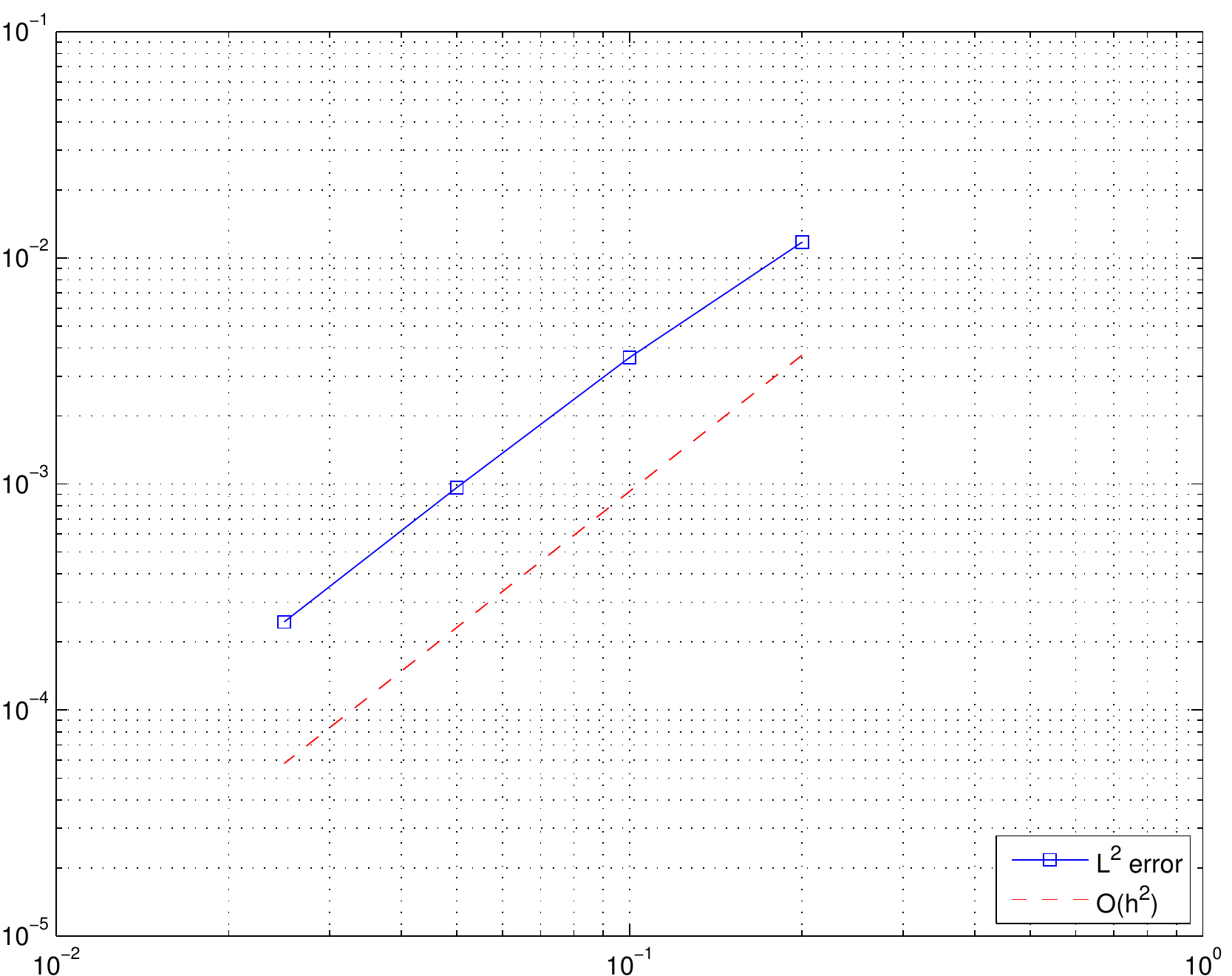}
\hfill{}\includegraphics[width=0.4\textwidth]{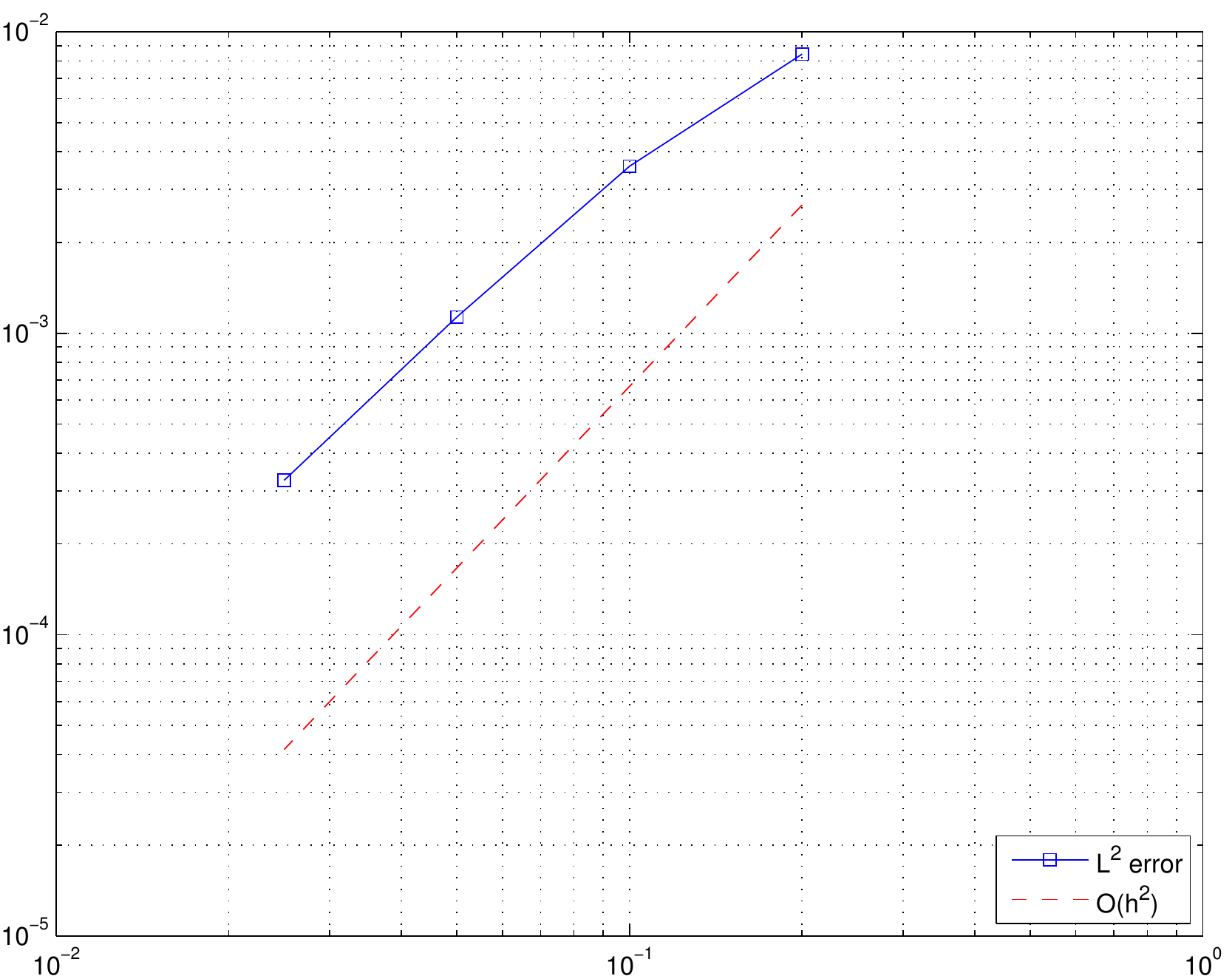}\hfill{}

\hfill{}\includegraphics[width=0.4\textwidth]{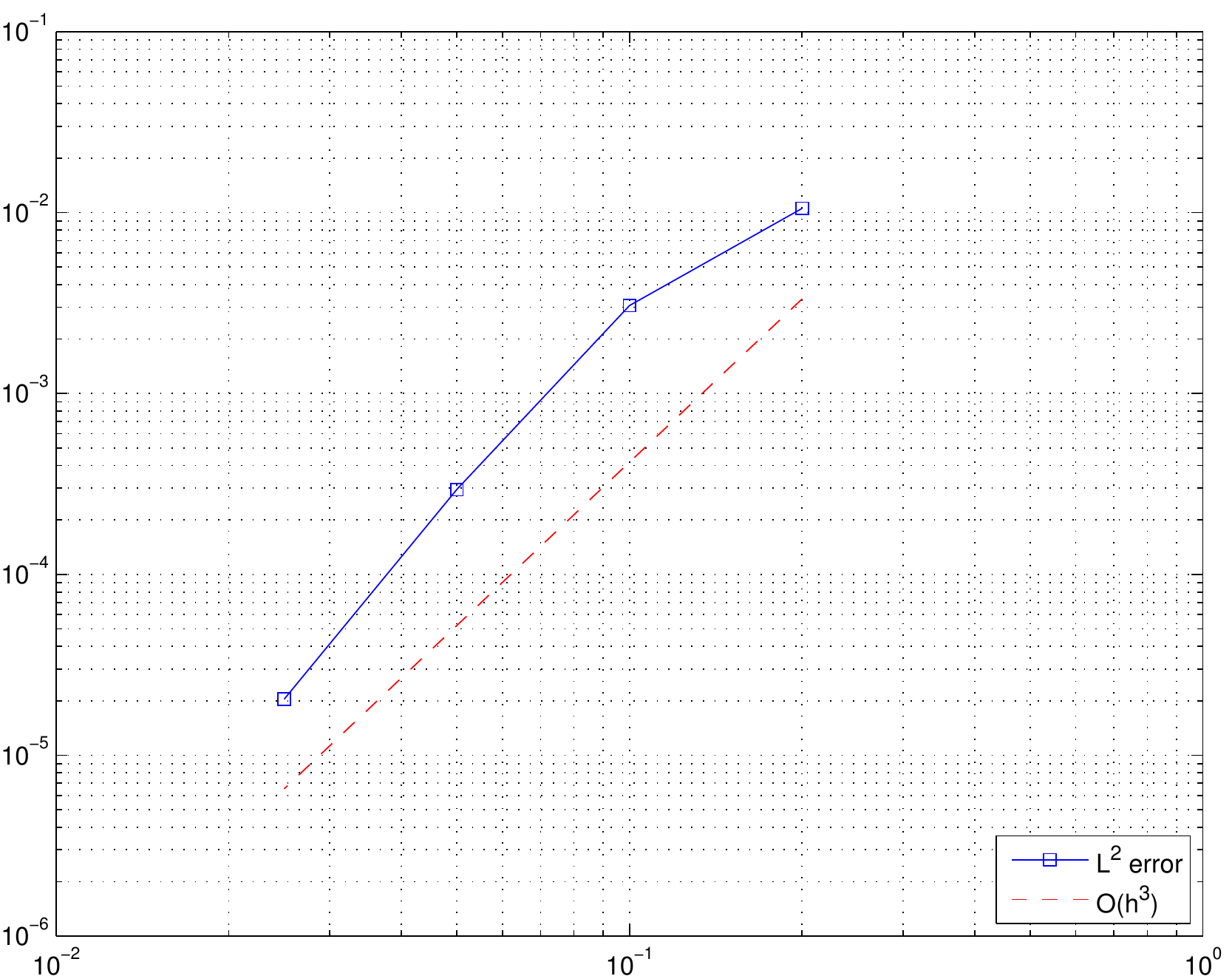}
\hfill{}\includegraphics[width=0.4\textwidth]{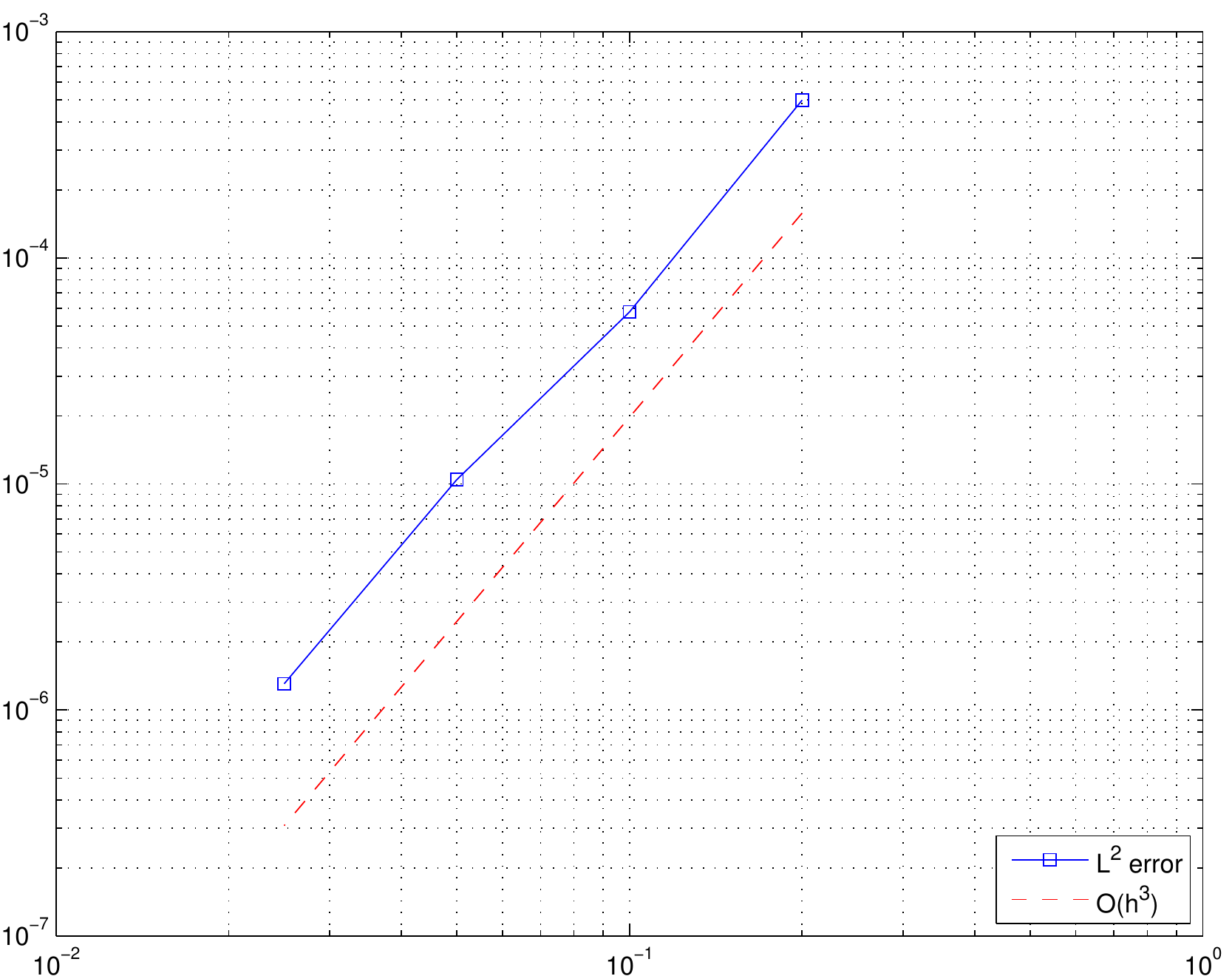}\hfill{}

\hfill{}\includegraphics[width=0.4\textwidth]{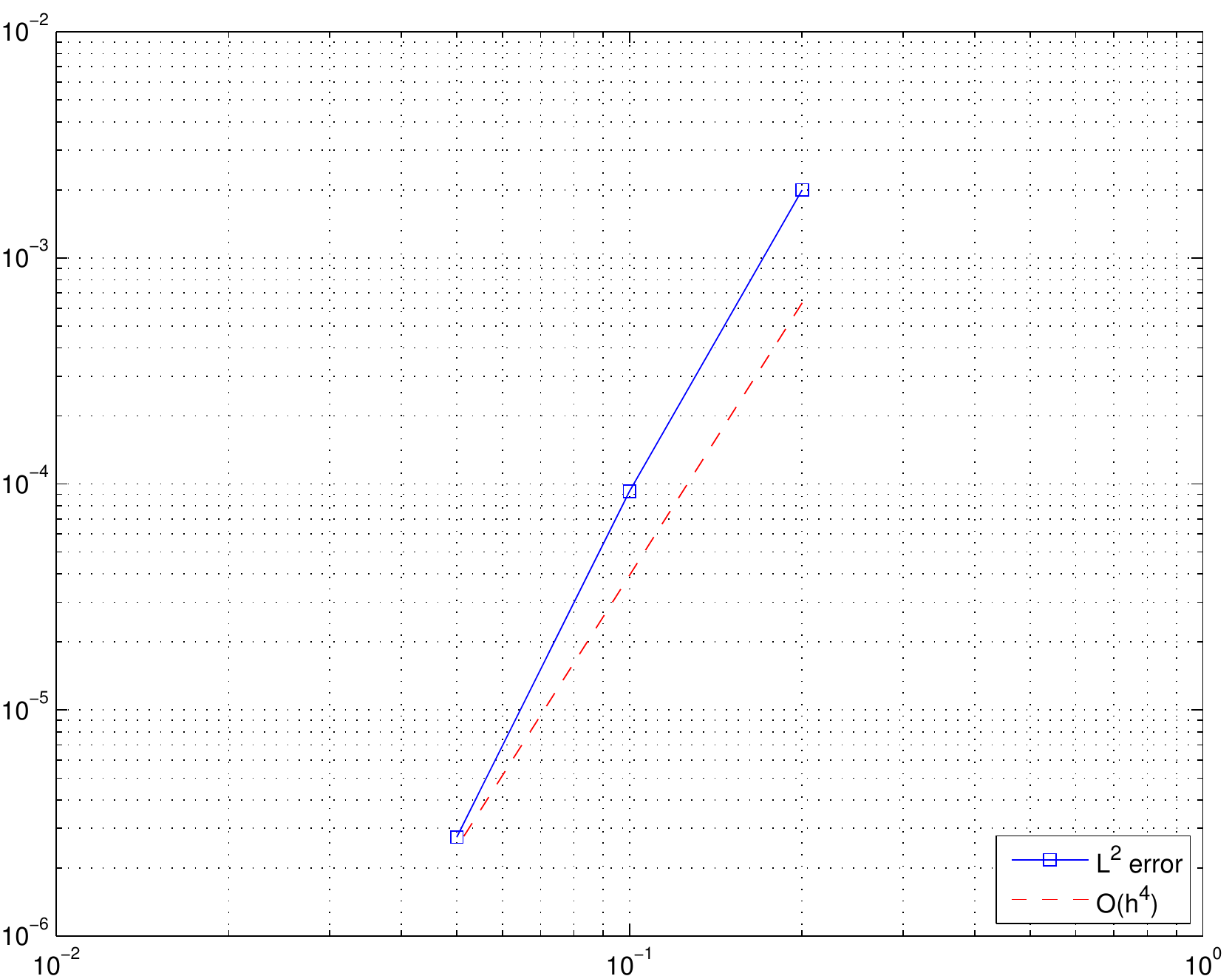}
\hfill{}\includegraphics[width=0.4\textwidth]{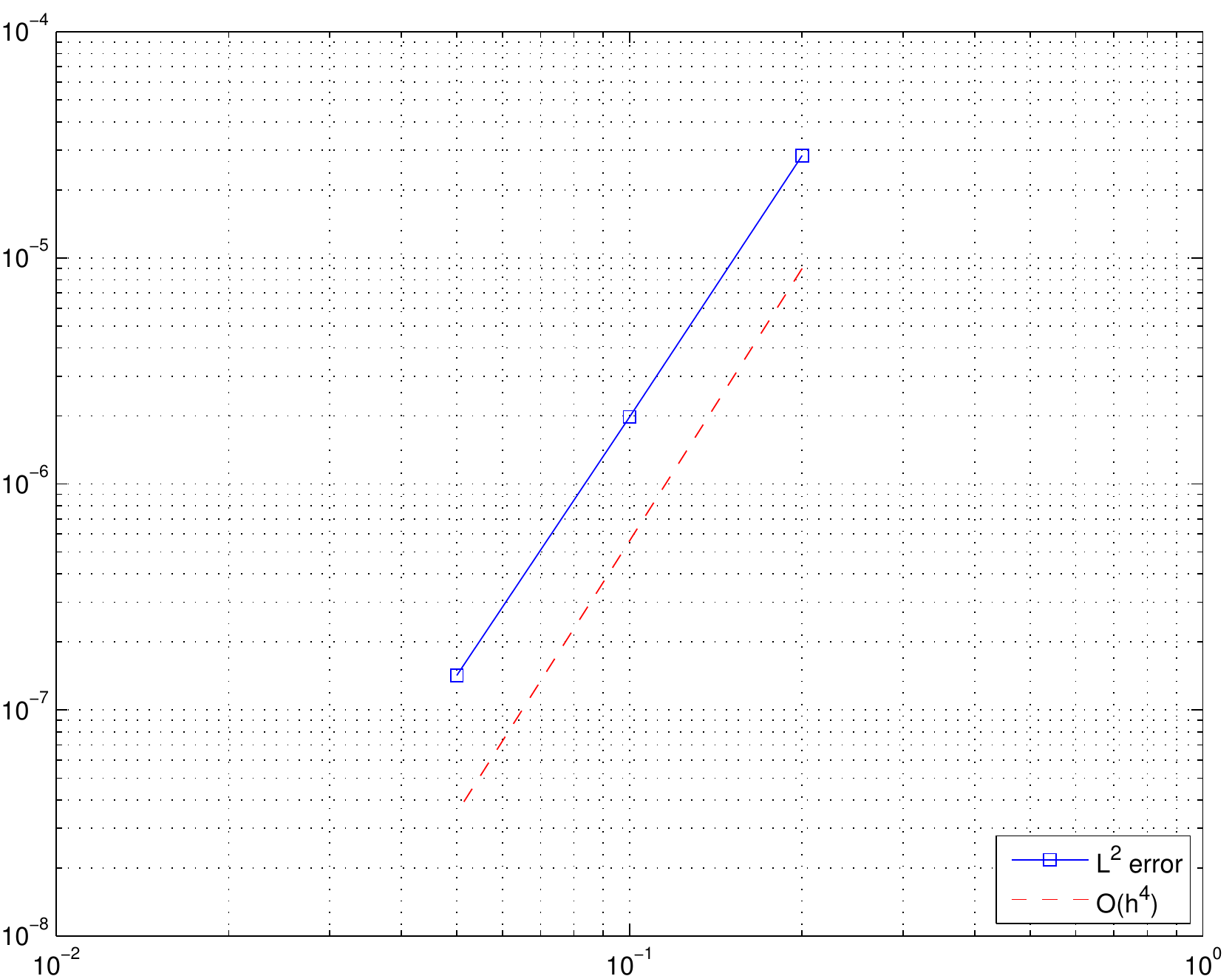}\hfill{}

\caption{Convergence history of $u_{h}$ produced by LS-weak at the subdoamin $(0,0.9)\times(0,0.9)$
away from the boundary layer. Left--right: $\epsilon=10^{-2}$,$\epsilon=10^{-9}$. 
Top--Bottom: $P1$--$P3$.}
\label{bdry_test1}
\end{figure}

\section{Extension to transportation reaction problems}
We present first order least squares method for the transportation reaction problems
\begin{subequations}
\label{transport_eqs}
\begin{align}
\boldsymbol{\beta}\cdot \nabla u + cu = &\; f  \quad \text{ in $\Omega$, }
\\
u = &\; g \quad \text{ on $\Gamma^{-}=\{x\in\partial \Omega : \boldsymbol{\beta}\cdot\boldsymbol{n}(x)<0\}$. }
\end{align}
\end{subequations}

The first order least squares method is to find $u_{h}\in W_{h}$ satisfying
\begin{align}
\label{ls_formulation_transport}
& \left( \boldsymbol{\beta}\cdot \nabla u_{h} + c u_{h}, 
\boldsymbol{\beta}\cdot \nabla w + c w \right)_{\Omega}\\
\nonumber
& \quad +\Sigma_{F\in \mathcal{E}_h^{\partial}} h_{F}^{-1}
\langle \max (-\boldsymbol{\beta}\cdot \boldsymbol{n}(x),0) u_{h}, w\rangle_{F}\\
\nonumber
 = & (f, \boldsymbol{\beta}\cdot \nabla w + c w)_{\Omega}\\
 \nonumber
&\quad +\Sigma_{F\in \mathcal{E}_h^{\partial}} h_{F}^{-1}
\langle \max (-\boldsymbol{\beta}\cdot \boldsymbol{n}(x),0) g, w\rangle_{F},
\quad \forall w\in \mathcal{U}_{h}.
\end{align}
Here, $\mathcal{U}_{h} = \{w|_{K}\in H^{1}(\Omega): w\in P_{k+1}(K),\quad\forall K\in\mathcal{T}_{h} \}$.

Obviously, (\ref{ls_formulation_transport}) is a special case of first order least squares method (\ref{ls_formulation}) 
when the diffusion coefficient is zero. By using similar argument in section~\ref{subsec:rewrite}, we have the following Theorem~\ref{Thm_transport}.
\begin{theorem}
\label{Thm_transport}
Let $u$ be the solution of transport reaction equation (\ref{transport_eqs}) and $u_{h}$ the 
numerical solution of first order least squares method (\ref{ls_formulation_transport}).
\begin{align}
\label{convergence_rate_global_transport}
\Vert u - u_{h}\Vert_{L^{2}(\Omega)} + \Vert \boldsymbol{\beta}\cdot \nabla (u - u_{h}) \Vert_{L^{2}(\Omega)} 
\leq C h^{k+1} \Vert u\Vert_{H^{k+2}(\Omega)}.
\end{align} 
\end{theorem}

\begin{remark}
The first order least squares method (\ref{ls_formulation_transport}) is the same as the method in \cite{Bochev_transport}, 
except the way of imposing boundary condition. In \cite{Bochev_transport}, the term $c-\frac{1}{2}\nabla\cdot\boldsymbol{\beta}$ 
is required to be uniformly bounded from below by a positive constant. We get rid of this restriction.
\end{remark}

{\bf Acknowledgements}. The work of Huangxin Chen was supported by the NSF of China (No. 11201394) and 
the NSF of Fujian Province (No. 2013J05016). The work of Jingzhi Li was supported by the NSF of China 
(No. 11201453 and 91130022). The work of Weifeng Qiu was supported by City University of Hong Kong 
under start-up grant (No. 7200324) and by the GRF of Hong Kong (Grant No. 9041980).

\end{document}